\documentclass[12pt, reqno]{article}
\usepackage{amsmath,amssymb,amsfonts,amsthm,graphicx}
\usepackage{setspace} 
\usepackage{color}
\usepackage{hyperref}
\usepackage{mathrsfs,mathabx}
\hypersetup{
    colorlinks=true,
    linkcolor=blue,
    citecolor=red,
    urlcolor=blue,
    pdfborder={0 0 0}
}            
\usepackage{subfigure}
\usepackage{mdframed}

\newtheorem{thm}{Theorem} 
\newtheorem{prop}{Proposition} 
\newtheorem{lemma}{Lemma}
\newtheorem{corollary}{Corollary} 
\theoremstyle{definition}
\newtheorem{defn}{Definition}
\newtheorem{remark}{Remark}

\renewcommand{\P}{\mathbb P}

\newcommand{\Z}{\mathbb Z}
\newcommand{\E}{\mathbb E}
\newcommand{\R}{\mathbb R}
\newcommand{\N}{\mathbb N}
\newcommand{\eps}{\varepsilon}

\newcommand{\eqd}{\ensuremath{=_d}}

\newcommand{\DS}{\textup{\textsf{DS}}}

\newcommand{\BD}{\textup{\textsf{BD}}}
\newcommand{\UIHPQ}{\textup{\textsf{UIHPQ}}}
\newcommand{\UIPT}{\textup{\textsf{UIPT}}}

\newcommand{\BHP}{\textup{\textsf{BHP}}}

\newcommand{\CRT}{\textup{\textsf{CRT}}}
\newcommand{\ICRT}{\textup{\textsf{SCRT}}}
\newcommand{\cT}{\mathcal{T}}
\newcommand{\cR}{\mathcal{R}}
\newcommand{\mC}{\mathcal{C}}

\newcommand{\dmap}{d_{\textup{map}}}
\newcommand{\dgr}{d_{\textup{gr}}}
\newcommand{\bdgr}{{\boldsymbol d}_{\textup{gr}}}
\newcommand{\dgh}{d_{\textup{GH}}}
\newcommand{\dha}{d_{\textup{H}}}
\newcommand{\dwst}{\textup{d}^\downarrow}
\newcommand{\dis}{\textup{dis}}

\newcommand{\f}{\textup{$\mathfrak{f}$}}
\newcommand{\la}{\textup{$\mathfrak{l}$}}
\newcommand{\bv}{{\bf v}}
\newcommand{\q}{\textup{$\mathsf{q}$}}

\newcommand{\m}{\textup{\textsf{m}}}
\newcommand{\Wha}{W^{\theta}}
\newcommand{\Xha}{X^{\theta}}
\newcommand{\Zha}{Z^{\theta}}

\newcommand{\Br}{\textup{$\mathfrak{B}$}}
\newcommand{\Fo}{\textup{$\mathfrak{F}$}}
\newcommand{\La}{\textup{$\mathfrak{L}$}}
\newcommand{\cQ}{\mathcal{Q}}

\newcommand{\cb}{\textup{Ball}}
\newcommand{\suc}{\textup{succ}}
\newcommand{\br}{\textup{\textsf{b}}}
\newcommand{\vd}{v^{\bullet}}

\newcommand{\Fq}{\textup{F}(\q)}
\newcommand{\Bq}{\partial\q}
\newcommand{\Pgz}{\P_{g,z}}
\newcommand{\Pgs}{\P^{\sigma}_{g}}

\DeclareMathOperator{\Loop}{\textup{Loop}}
\DeclareMathOperator{\Scoop}{\textup{Scoop}}
\DeclareMathOperator{\Tree}{\textup{Tree}}
\DeclareMathOperator{\Treeb}{\textup{\textsf{Tree}}}
\DeclareMathOperator{\Cut}{\textup{Cut}}

\newcommand{\tr}{\mathsf{t}}
\newcommand{\lt}{\mathsf{l}}
\newcommand{\nuw}{\nu_\circ}
\newcommand{\nub}{\nu_\bullet}
\newcommand{\muw}{\mu_\circ}
\newcommand{\mub}{\mu_\bullet}

\begin{document}

\title{Uniform infinite half-planar quadrangulations with skewness}
\author{Erich Baur\footnote{Bern University of Applied Sciences 
    (BFH),
    erich.baur@bfh.ch},\quad Lo\"ic
  Richier\footnote{CMAP, \'Ecole Polytechnique, loic.richier@polytechnique.edu}} \date{\small
  \today}
\maketitle
\thispagestyle{empty}

\begin{abstract} 
  We introduce a one-parameter family of random infinite quadrangulations
  of the half-plane, which we call the {\it uniform infinite half-planar
    quadrangulations with skewness} ($\UIHPQ_p$ for short, with
  $p\in[0,1/2]$ measuring the skewness). They interpolate between Kesten's
  tree corresponding to $p=0$ and the usual $\UIHPQ$ with a general
  boundary corresponding to $p=1/2$.  As we make precise, these models
  arise as local limits of uniform quadrangulations with a boundary when
  their volume and perimeter grow in a properly fine-tuned way, and they
  represent all local limits of (sub)critical Boltzmann quadrangulations
  whose perimeter tend to infinity. Our main result shows that the family
  $(\UIHPQ_p)_p$ approximates the Brownian half-planes $\BHP_\theta$,
  $\theta\geq 0$, recently introduced in~\cite{BaMiRa}. For $p<1/2$, we
  give a description of the $\UIHPQ_p$ in terms of a looptree associated to
  a critical two-type Galton-Watson tree conditioned to survive.
\end{abstract} {\bf Key words:} Uniform infinite half-planar
quadrangulation, Brownian half-plane, Kesten's tree, multi-type
Galton-Watson tree, looptree, Boltzmann map.\newline {\bf
  Subject Classification:} 05C80; 05C81; 05C05; 60J80; 60F05.

\footnote{{\it Acknowledgment of support.} The research of EB was supported
  by the Swiss National Science Foundation grant P300P2\_161011, and
  performed within the framework of the LABEX MILYON (ANR-10-LABX-0070) of
  Universit\'e de Lyon, within the program ``Investissements d'Avenir''
  (ANR-11-IDEX-0007) operated by the French National Research Agency (ANR).
  EB and LR thank the Institute for Mathematical Research (FIM) of ETH
  Zurich for hospitality, where parts of this work were completed. }
\newpage
\tableofcontents
\section{Introduction}
\label{sec:intro}
\subsection{Overview}
The purpose of this paper is to introduce and study a one-parameter family
of random infinite quadrangulations of the half-plane, which we denote by
$(\UIHPQ_p)_{0\leq p\leq 1/2}$ and call the {\it uniform infinite
  half-planar quadrangulations with skewness}. Two members play a
particular role: The choice $p=0$ corresponds to Kesten's tree,
cf. Proposition~\ref{prop:kesten} below, whereas the choice $p=1/2$
corresponds to the standard uniform
infinite half-planar quadrangulation $\UIHPQ$ with a general boundary.

Kesten's tree~\cite{Ke} is a random infinite planar tree, which we may view
as a degenerate quadrangulation with an infinite boundary, but no inner
faces. It arises as the local limit of critical Galton-Watson trees
conditioned to survive. The standard $\UIHPQ(=\UIHPQ_{1/2})$ forms the
half-planar analog of the uniform infinite planar quadrangulation
introduced by Krikun~\cite{Kr}, after the seminal work of Angel and
Schramm~\cite{AnSc} on triangulations of the plane. Curien and
Miermont~\cite{CuMi} showed that the $\UIHPQ$ arises as a local limit of
uniformly chosen quadrangulations of the two-sphere with $n$ inner faces
and a boundary of size $2\sigma$, upon letting first $n\rightarrow\infty$
and then $\sigma\rightarrow\infty$ (see Angel~\cite{An} for the case of
triangulations with a simple boundary).

We will define each $\UIHPQ_p$ in Section~\ref{sec:UIHPQp} by means of an
extension of the Bouttier-Di Francesco-Guitter mapping to infinite
quadrangulations with a boundary.  In the first part of this paper, we will
discuss various local limits and scaling limits which involve the family
$(\UIHPQ_p)_p$. More precisely, in Theorem~\ref{thm:local-conv}, we will
see that each $\UIHPQ_p$ appears as a local limit as $n$ tends to infinity
of uniform quadrangulations $Q_n^{\sigma_n}$ with $n$ inner faces and a
boundary of size $2\sigma_n$, for an appropriate choice of
$\sigma_n=\sigma_n(p)\rightarrow\infty$. In
Proposition~\ref{prop:Boltzmann}, we argue that the family $(\UIHPQ_p)_p$
consists precisely of the infinite quadrangulations with a boundary which
are obtained as local limits $\sigma\rightarrow\infty$ of subcritical
Boltzmann quadrangulations with a boundary of size $2\sigma$. This result
will prove helpful in our description of the $\UIHPQ_p$ given in
Theorem~\ref{thm:BranchingStructure}.

We will then turn to distributional scaling limits of the family
$(\UIHPQ_p)_p$ in the so-called {\it local Gromov-Hausdorff topology}.  In
Theorems~\ref{thm:GHconv-BHPtheta} and~\ref{thm:GHconv-ICRT}, we will
clarify the connection between the (discrete) quadrangulations $\UIHPQ_p$
and the family $(\BHP_\theta)_{\theta\geq 0}$ of Brownian half-spaces with
skewness $\theta$ introduced in~\cite{BaMiRa}. More specifically, upon
rescaling the graph distance by a factor $a_n^{-1}\rightarrow 0$, we prove
that each $\BHP_\theta$ is the distributional limit of the rescaled spaces
$a_n^{-1}\cdot\UIHPQ_{p_n}$, if $p_n=p_n(\theta,a_n)$ is adjusted in the
right manner (Theorem~\ref{thm:GHconv-BHPtheta}). In our setting,
convergence in the local Gromov-Hausdorff sense amounts to show convergence
of rescaled metric balls around the roots of a fixed but arbitrarily large
radius in the usual Gromov-Hausdorff topology; see Section~\ref{sec:GH}.

In~\cite{BaMiRa}, a classification of all possible non-compact scaling
limits of pointed uniform random quadrangulations with a boundary
$(V(Q_n^{\sigma_n}),a_n^{-1}\dgr,\rho_n)$ has been given, depending on the
asymptotic behavior of the boundary size $2\sigma_n$ and on the choice of the
scaling factor $a_n\rightarrow\infty$ (in the local Gromov-Hausdorff
topology, with the distinguished point $\rho_n$ lying on
the boundary). In this paper, we address the boundary regime corresponding
to the portion $x\geq 1$ of the $y=0$ axis in Figure~\ref{fig:usersmanual}
(in hashed marks), which was left untouched in~\cite{BaMiRa}. As we show,
it corresponds to a regime of unrescaled local limits, namely the family
$(\UIHPQ_p)_p$.
\begin{figure}[ht!]
  \centering
  \includegraphics[width=0.5\textwidth]{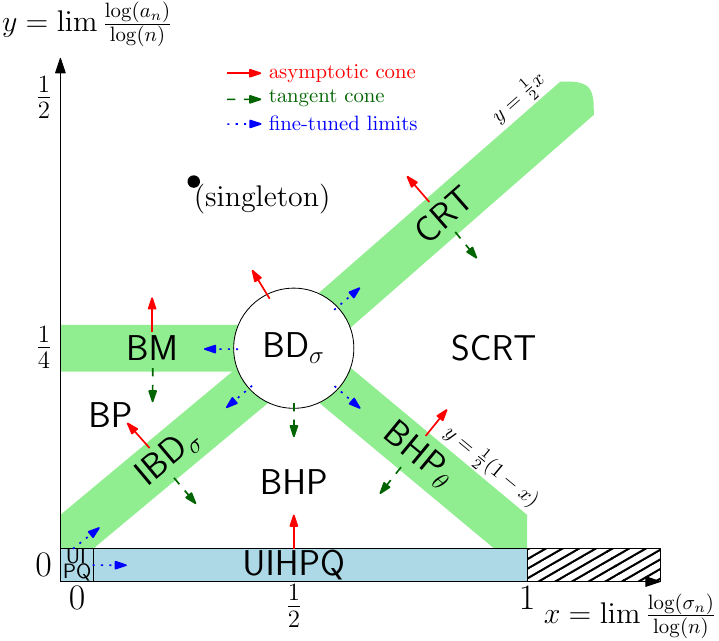}
  \caption{In~\cite{BaMiRa}, all possible limits for the rescaled spaces
    $(V(Q_n^{\sigma_n}),a_n^{-1}\dgr,\rho_n)$ are discussed. The $x$-axis
    represents the limit values for the logarithm of the boundary length
    $\log(\sigma_n)/\log(n)$ in units of $\log(n)$, and the $y$-axis
    corresponds to the limit of the logarithm of the scaling factor
    $\log(a_n)/\log(n)$ in units of $\log(n)$. The focus of this paper lies
    on the hashed region.}
  \label{fig:usersmanual}
\end{figure}

We finally give a branching characterization of the $\UIHPQ_p$ when
$p<1/2$.  For that purpose, we will adapt the concept of discrete random
looptrees introduced by Curien and Kortchemski~\cite{CuKo1}. We will see
that the $\UIHPQ_p$ admits a representation in terms of a looptree
associated to a two-type version of Kesten's infinite tree. Informally, we
will replace each vertex $u$ at odd height in Kesten's tree by a cycle of
length $\deg(u)$, which connects the vertices incident to $u$. Here,
$\deg(u)$ stands for the degree (i.e., the number of neighbors) of $u$ in
the tree. We then fill in the cycles of the looptree with a collection of
independent quadrangulations with a simple boundary, which are drawn
according to a subcritical Boltzmann law. As we show in
Theorem~\ref{thm:BranchingStructure}, the space constructed in this way has
the law of the $\UIHPQ_p$. Discrete looptrees and their scaling limits have
found various applications in the study of large-scale properties of random
planar maps, for instance in the description of the boundary of percolation
clusters on the uniform infinite planar triangulation; see the
work~\cite{CuKo2}, which served as the main inspiration for our
characterization of the $\UIHPQ_p$. From our description, we immediately
infer that simple random walk is recurrent on the $\UIHPQ_p$ for $p<1/2$.

It is well-known that the standard $\UIHPQ$ with a simple boundary
satisfies the so-called~\textit{spatial Markov property}, which allows, in
particular, the use of peeling techniques. In~\cite{AnRa1}, Angel and Ray
classified all triangulations (without self-loops) of the half-plane
satisfying the spatial Markov property and translation invariance. They
form a one-parameter family $(\mathbb{H}_{\alpha})_\alpha$ parametrized by
$\alpha \in [0,1)$. The parameter $\alpha=2/3$ corresponds to the standard
$\textsf{UIHPT}$ with a simple boundary, the triangular equivalent of the
$\UIHPQ$ with a simple boundary. When $\alpha > 2/3$ (the supercritical
case), $\mathbb{H}_{\alpha}$ is of hyperbolic nature and exhibits an
exponential volume growth. On the contrary, when $\alpha < 2/3$ (the
subcritical case), it has a tree-like structure. We believe that the family
$(\UIHPQ_p)_p$ is a quadrangular equivalent to the triangulations in the
subcritical phase of~\cite{AnRa1}. Note that contrary to the $\UIHPQ_p$,
the spaces $\mathbb{H}_\alpha$ for $\alpha<2/3$ have a half-plane topology,
due to the conditioning to have a simple boundary. However, there exists
almost surely infinitely many cut-edges connecting the left and right
boundaries; see~\cite[Proposition 4.11]{Ra}. This should be seen as an
equivalent to the branching structure formulated in
Theorem~\ref{thm:BranchingStructure} below. Our methods in this paper are
different from~\cite{AnRa1,Ra} as we do not use peeling techniques.

In~\cite{Cu2}, Curien studied full-plane analogs of the family
$(\mathbb{H}_\alpha)_\alpha$. With similar (peeling) techniques, he
constructed a (unique) one-parameter family of random infinite planar
triangulations indexed by $\kappa \in(0,2/27]$, which satisfy a slightly
adapted spatial Markov property. The critical case $\kappa=2/27$
corresponds to the standard $\textsf{UIPT}$ with a simple boundary of Angel
and Schramm~\cite{AnSc}. The regime $\kappa \in (0,2/27)$ parallels the
supercritical (or hyperbolic) phase $\alpha> 2/3$ of~\cite{AnRa1}, whereas
it is shown that there is no subcritical phase. Recently, a near-critical
scaling limit of hyperbolic nature called the hyperbolic Brownian
half-plane has been studied by Budzinski~\cite{Bz}. It is obtained from
rescaling the triangulations of Curien~\cite{Cu2} and letting
$\kappa\rightarrow 2/27$ at the right speed. Theorem 1 of~\cite{Bz} bears
some structural similarities with our Theorem~\ref{thm:GHconv-BHPtheta}
below, although it concerns a different regime.

\bigskip
\noindent{\bf Structure of the paper}\newline
The rest of this paper is structured as follows. In the following section,
we introduce some (standard) concepts and notation around quadrangulations,
which will be used throughout this text. Moreover, we recapitulate the
local topology and the local Gromov-Hausdorff topology. In
Section~\ref{sec:mainresults}, we state our main results, which concern
local limits, scaling limits, and structural properties of the family
$(\UIHPQ_p)_p$. Section~\ref{sec:halfplanes-trees} reviews the definition
of the family of Brownian half-planes $(\BHP_\theta)_\theta$, and of
various random trees, which are used both to describe the distributional
limits of the family $(\UIHPQ_p)_p$ as well as their branching structure.

In Section~\ref{sec:UIHPQp}, we construct the $\UIHPQ_p$. We first explain
the Bouttier-Di Francesco-Guitter encoding of quadrangulations with a
boundary and then define the $\UIHPQ_p$ in terms of the encoding
objects. We are then in position to prove our limit statements; see
Section~\ref{sec:proofs-limits}. In the final
Section~\ref{sec:proofs-struct}, we prove our main result characterizing
the tree-like structure of the $\UIHPQ_p$ when $p<1/2$, as well as
recurrence of simple random walk.

\subsection{Some standard notation and definitions}
\subsubsection{Notation}
We write 
\[\N=\{1,2,\ldots\},\quad \N_0 = \Z_{\geq
  0}=\N\cup\{0\},\quad\Z_{<0}=\{-1,-2,\ldots\}.
\]
For two sequences $(a_n)_n$, $(b_n)_n\subset \N$, we write $a_n\ll b_n$ or
$b_n\gg a_n$ if $a_n/b_n\rightarrow 0$ as $n\rightarrow\infty$.  Given two
measurable subsets $U,V\subset\R$, we denote by $\mathcal{C}(U,V)$ the
space of continuous functions from $U$ to $V$, equipped with the usual
compact-open topology, i.e., uniform convergence on compact subsets. We
write $\|\nu\|_{\textup{TV}}$ for the total variation norm of a probability
measure $\nu$.

As a general notational rule for this paper, if we drop $p$ from the
notation, we work with the case $p=1/2$. For example, we
write $\UIHPQ$ (and not $\UIHPQ_{1/2}$) for the standard uniform infinite
half-planar quadrangulation.

\subsubsection{Planar maps}
By {\it planar map} we mean, as usual, an equivalence class of a proper
embedding of a finite connected graph in the two-sphere, where two
embeddings are declared to be equivalent if they differ only by an
orientation-preserving homeomorphism of the sphere. Loops and multiple
edges are allowed. Our planar maps will be rooted, meaning that we
distinguish an oriented edge called the {\it root edge}. Its origin is the
root vertex of the map. The faces of a planar map are formed by the
components of the complement of the union of its edges.

\subsubsection{Quadrangulations with a boundary}
A {\it quadrangulation with a boundary} is a finite planar
map $\q$, whose faces are quadrangles except possibly one face called the
{\it outer face}, which may an have arbitrary even degree. The edges
incident to the outer face form the {\it boundary} $\partial \q$ of $\q$,
and their number $\#\partial \q$ (counted with multiplicity) is the size or
{\it perimeter} of the boundary. In general, we do not assume that the
boundary edges form a simple curve. We will root the map  by selecting an
oriented edge of the boundary, such that the
outer face lies to its right. The {\it size} of $\q$ is given by the number
of its {\it inner faces}, i.e., all the faces different from the outer face.

We write $\cQ_n^{\sigma}$ for the (finite) set of all rooted
quadrangulations with $n$ inner faces and a boundary of size $2\sigma$,
$\sigma\in\N_0$. By convention, $\cQ_0^{0}=\{\dag\}$ consists of the
unique vertex map.

More generally, $\cQ_f$ will denote the set of all finite
rooted quadrangulations with a boundary, and $\cQ^{\sigma}_f\subset \cQ_f$
the set of all finite rooted quadrangulations with $2\sigma$ boundary edges,
for $\sigma\in\N_0$. 

Similarly, we let $\widehat{\cQ}_f$ be the set of all finite rooted
quadrangulations with a simple boundary, meaning that the edges of their
outer face form a cycle without self-intersection. We denote by
$\widehat{\cQ}^{\sigma}_f\subset\widehat{\cQ}_f$ the subset of finite
rooted quadrangulations with a simple boundary of size $2\sigma$.  Note
that $\cQ_0^{1}$ consists of the map having one oriented edge and thus a
simple boundary.

\subsubsection{Uniform quadrangulations with a boundary}
Throughout this text, we write $Q_n^{\sigma}$ for a quadrangulation chosen
uniformly at random in $\cQ_n^\sigma$. We denote by $\rho_n$ the root
vertex of $Q_n^{\sigma}$, i.e., the origin of the root edge. By equipping
the set of vertices $V(Q_n^{\sigma})$ with the graph distance $\dgr$, we
view the triplet $(V(Q_n^{\sigma}),\dgr,\rho_n)$ as a random rooted metric
space.
\subsubsection{Boltzmann quadrangulations with a boundary}
\label{sec:BoltzmannQuadrangulations}
We will also work with various
Boltzmann measures. For a finite rooted quadrangulation $\q\in\cQ_f$, we write $\Fq$ for the
set of inner faces of $\q$. Given non-negative weights $g$ per inner face
and $\sqrt{z}$ per boundary edge, we let
\[F(g,z)=\sum_{\q \in \cQ_f}g^{\#\Fq}z^{\#\Bq/2}.\]
When this partition function is finite, we may define the associated
Boltzmann distribution
\[\Pgz(\q)=\frac{g^{\#\Fq}z^{\#\Bq/2}}{F(g,z)}, \quad \q \in \cQ_f.\]
The statement of Proposition~\ref{prop:Boltzmann} below deals with 
Boltzmann-distributed quadrangulations of a fixed boundary size
$2\sigma$, for $\sigma\in\N_0$. In this case, the associated partition function and Boltzmann distribution read
\[F_{\sigma}(g)=\sum_{\q \in \cQ^{\sigma}_f}g^{\#\Fq},\quad
\Pgs(\q)=\frac{g^{\#\Fq}}{F_\sigma(g)}, \quad \q \in \cQ^{\sigma}_f,\]
whenever $g\geq 0$ is such that $F_\sigma(g)$ is finite.  The Boltzmann
distribution $\Pgs$ is related to $\Pgz$ by conditioning the latter with
respect to the boundary length, i.e.,
$\Pgs(\q)=\Pgz(\q \mid \cQ^{\sigma}_f)$.

When studying quadrangulations with a simple boundary, the partition functions are
\[\widehat{F}(g,z)=\sum_{\q \in
  \widehat{\cQ}_f}g^{\#\Fq}z^{\#\Bq/2},\quad\widehat{F}_{\sigma}(g)=\sum_{\q \in
  \widehat{\cQ}^{\sigma}_f}g^{\#\Fq},\]
and the Boltzmann distributions take the form
\[\widehat{\P}_{g,z}(\q)=\frac{g^{\#\Fq}}{\widehat{F}_\sigma(g,z)},\quad \q \in \widehat{\cQ}_f,\quad
\widehat{\P}^{\sigma}_g(\q)=\frac{g^{\#\Fq}}{\widehat{F}_\sigma(g)},
\quad \q \in \widehat{\cQ}^{\sigma}_f.\]

\begin{remark}In the notation of~\cite{BoGu}, the generating function $F$
  is denoted $W_0$, while $\widehat{F}$ is denoted $\tilde{W}_0$. The index
  zero stands for the distance between the origin of the root edge and the
  marked vertex, so that these generating functions count unpointed
  quadrangulations.
\end{remark}

\subsubsection{Local topology}
Our unrescaled limit results hold with respect to the {\it local topology}
first studied by Benjamini and Schramm~\cite{BeSc}:
For two rooted planar maps $\m$ and $\m'$, the local distance between $\m$
and $\m'$ is
\[
\dmap(\m,\m') = \left(1+\sup\{r\geq 0:\cb_r(\m)=\cb_r(\m')\}\right)^{-1},
\]
where $\cb_r(\m)$ denotes the combinatorial ball of radius $r$ around the root
$\rho$ of $\m$, i.e., the submap of $\m$ consisting of all the vertices $v$ of
$\m$ with $\dgr(\rho,v)\leq r$ and all the edges of $\m$ between such vertices.
The set $\mathcal{Q}_f$ of all finite rooted quadrangulations with a
boundary is not complete for the distance $\dmap$; we have to add infinite
quadrangulations. We shall write $\cQ$ for the completion of
$\mathcal{Q}_f$ with respect to $\dmap$. The $\UIHPQ_p$ will be defined as
a random element in $\cQ$.
\subsubsection{Around the Gromov-Hausdorff metric}
\label{sec:GH}
The pointed Gromov-Hausdorff distance measures the distance between
(pointed) compact metric spaces, where the latter are viewed up to
isometries. More specifically, given two elements $\mathbf{E}=(E,d,\rho)$
and $\mathbf{E}' =(E',d',\rho')$ in the space $\mathbb{K}$ of isometry
classes of pointed compact metric spaces, their Gromov-Hausdorff distance
is defined as
\[
\dgh(\mathbf{E},\mathbf{E}') =
\inf\left\{\dha(\varphi(E),\varphi'(E))\vee\delta(\varphi(\rho),\varphi'(\rho'))\right\},
\]
where the infimum is taken over all isometric embeddings $\varphi :
E\rightarrow F$ and $\varphi' : E'\rightarrow F$ of $E$ and $E'$ into the
same metric space $(F,\delta)$, and $\dha$ is the usual Hausdorff distance
between compacts of $F$. The space $(\mathbb{K},\dgh)$ is complete and
separable. 

Our results on scaling limits involve non-compact pointed metric spaces and
hold in the so-called {\it local Gromov-Hausdorff sense}, which we briefly
recall next.  Given a pointed complete and locally compact length space
$\mathbf{E}$ and a sequence $(\mathbf{E}_n)_n$ of such spaces,
$(\mathbf{E}_n)_n$ converges in the local Gromov-Hausdorff sense to
$\mathbf{E}$ if for every $r\geq 0$,
\[
\dgh(B_r(\mathbf{E}_n),B_r(\mathbf{E}))\rightarrow 0\quad\textup{ as }n\rightarrow\infty.
\]
Here and in what follows, given a pointed metric space
$\mathbf{F}=(F,d,\rho)$, $B_r(\mathbf{F})=\{x\in F: d(x,\rho)\leq r\}$
denotes the closed ball of radius $r$ around $\rho$, viewed as a subspace
of $\mathbf{F}$ equipped with the metric structure inherited from
$\mathbf{F}$. For $\lambda>0$, $\lambda\cdot \mathbf{F}$ stands for the
rescaled pointed metric space $(F,\lambda d,\rho)$, so that in particular
$\lambda\cdot B_r(\mathbf{F})=B_{\lambda r}(\lambda\cdot\mathbf{F})$.

As a discrete map, the $\UIHPQ_p$ is not a length space in the sense
of~\cite{BuBuIv}. However, by identifying each edge with a copy of the unit
interval $[0,1]$ (and by extending the metric isometrically), one obtains a
complete locally compact length space (pointed at the root vertex). By
construction, balls of the same radius and around the same points in the
$\UIHPQ_p$ and in the approximating length space are at Gromov-Hausdorff
distance at most $1$ from each other. Therefore, local Gromov-Hausdorff
convergence for the (rescaled) $\UIHPQ_p$, see
Theorems~\ref{thm:GHconv-BHPtheta} and~\ref{thm:GHconv-ICRT} below, follows
indeed from the convergence of balls as stated above.

\section{Statements of the main results}
\label{sec:mainresults}
\subsection{Local limits}
\label{sec:results-locallimits}
Our first result states that each member of the family
$(\UIHPQ_p)_{0\leq p\leq 1/2}$ can be seen as a local limit
$n\rightarrow\infty$ of uniform quadrangulations of with $n$ inner
faces and a boundary of size $2\sigma_n$, provided
$\sigma_n=\sigma_n(p)$ is chosen in the right manner.

\begin{thm}
\label{thm:local-conv}
Fix $0\leq p\leq 1/2$, and let $(\sigma_n,n\in\N)$ be a sequence of positive
integers satisfying
\[\sigma_n= \frac{1-2p}{p}n+o(n)\quad\textup{if }0<p\leq
1/2,\quad\textup{and}\quad\sigma_n\gg n\quad\textup{if }p=0.\]
For every $n\in\N$, let $Q_n^{\sigma_n}$ be uniformly distributed in
$\cQ_n^{\sigma_n}$.  Then we have the local convergence for the metric
$\dmap$ as $n\rightarrow\infty$,
\[
Q_n^{\sigma_n} \xrightarrow[]{(d)} \UIHPQ_p.
\]  
\end{thm}
In fact, we will prove a stronger result than mere local convergence: We
will establish an isometry of balls of growing radii around the roots,
where the maximal growth rate of the radii is given by
$\xi_n=o(\sqrt{n})$.  We defer to
Proposition~\ref{prop:couplingQnUIHPQ} for the exact statement. The case
$p=1/2$ corresponding to the regime $\sigma_n=o(n)$ is already covered
by~\cite[Proposition 3.11]{BaMiRa} and is only included for completeness.

The convergence in the case $p=0$ with $\sigma_n\gg n$ is somewhat
simpler. However, it is {\it a priori} not obvious that the $\UIHPQ_0$ as
defined in Section~\ref{sec:UIHPQp} is actually Kesten's tree (see
Section~\ref{sec:def-kesten} for a definition of the latter).
\begin{prop}
\label{prop:kesten}
The space $\UIHPQ_0$ has the law of Kesten's tree $\cT_\infty$ associated to
the critical geometric probability distribution $(\mu_{1/2}(k),k\in\N_0)$ 
given by $\mu_{1/2}(k)=2^{-(k+1)}$.
\end{prop}
Interestingly, the fact that the $\UIHPQ_0$ is Kesten's tree can also be
derived as a special case from Theorem~\ref{thm:BranchingStructure} below;
see Remark~\ref{rem:UIHPQ0Kesten}. We prefer, however, to give a direct
proof of the proposition based on our construction of the $\UIHPQ_0$.

The $\UIHPQ_p$ for $0\leq p\leq 1/2$ is also obtained as a local limit of
Boltzmann quadrangulations with growing boundary size. This result will be
important to describe the tree-like structure of the $\UIHPQ_p$ when
$p<1/2$. More specifically, the family $(\UIHPQ_p)_p$ is precisely given by
the collection of all local limits $\sigma\rightarrow\infty$ of Boltzmann
quadrangulations with a boundary of size $2\sigma$ and weight
$g\leq g_c=1/12$ per inner face. The value $g_c=1/12$ is critical
(see~\cite[Section 4.1]{BoGu}) and corresponds to the choice $p=1/2$.
\begin{prop}
\label{prop:Boltzmann}
Fix $0\leq p\leq 1/2$, and set $g_p=p(1-p)/3$. For every $\sigma\in\N_0$, let
$Q_{\sigma}(p)$ be a random rooted quadrangulation distributed according to
the Boltzmann measure 
$\P^{\sigma}_{g_p}$. Then we
have the local convergence for the metric $\dmap$ as $\sigma\rightarrow\infty$,
\[
Q_{\sigma}(p)\xrightarrow[]{(d)} \UIHPQ_p.
\]  
\end{prop}

\begin{remark}\label{rem:BoltzmannCvgce}
  For $p=1/2$, the above proposition states convergence of critical
  Boltzmann quadrangulations with a boundary towards the $\UIHPQ$, as it
  was already proved in~\cite[Theorem 7]{Cu1} by means of peeling
  techniques. In view of the above proposition, it is moreover implicit
  from the same theorem that an infinite random map with the law of the $\UIHPQ_p$
  does exist. For the case of half-planar triangulations (with a simple
  boundary), see~\cite{An}. When $p=0$, there is no inner quadrangle almost
  surely and $Q^{\sigma}(0)$ is a uniform tree with $\sigma$ edges (i.e., a
  Galton-Watson tree with geometric offspring law conditioned to have
  $\sigma$ edges), which converges locally towards Kesten's tree; see, for
  example,~\cite[Theorem 7.1]{Ja}.\end{remark}

\begin{remark}
  Let us write $\mathcal{M}(\mathcal{Q})$ for the set of probability
  measures on the completion $\mathcal{Q}$, and equip it with the usual
  weak topology.  Then it is easily seen by our methods that the mapping
  $ [0,1/2]\ni p\mapsto \textup{Law}(\UIHPQ_p)\in\mathcal{M}(\mathcal{Q}) $
  is continuous.
\end{remark}

\subsection{Scaling limits}
\label{sec:results-scalinglimits}
Our next results address scaling limits of the family
$(\UIHPQ_p)_p$. In~\cite{BaMiRa}, a one-parameter family of (non-compact)
random rooted metric spaces called the {\it Brownian half-planes}
$\BHP_\theta$ with skewness $\theta\geq 0$ was introduced. See
Section~\ref{sec:def-BHPtheta} for a quick reminder.  The Brownian
half-plane $\BHP_0$ corresponding to the choice $\theta=0$ forms the
half-planar analog of the Brownian plane introduced in~\cite{CuLG} and
arises from zooming-out the $\UIHPQ$ around the root vertex;
see~\cite[Theorem 3.6]{BaMiRa}, and~\cite[Theorem 1.10]{GwMi}). Here, we
will see more generally that the family $(\UIHPQ_p)_p$ approximates the
space $\BHP_\theta$ for each $\theta\geq 0$ in the local Gromov-Hausdorff
sense, provided $p$ is appropriately fine-tuned (depending on $\theta$).

\begin{thm}
\label{thm:GHconv-BHPtheta}
Let $\theta\geq 0$. Let $(a_n,n\in\N)$ be a sequence of positive reals with
$a_n\rightarrow\infty$ as $n\rightarrow\infty$. Let
$(p_n,n\in\N)\subset [0,1/2]$ be a sequence satisfying
\[
p_n=p_n(\theta,a_n)=\frac{1}{2}\left(1-\frac{2\theta}{3a_n^2}\right) + o\left(a_n^{-2}\right).
\]
Then, in the sense of the local Gromov-Hausdorff topology as $n\rightarrow\infty$,
\[a_n^{-1}\cdot
\UIHPQ_{p_n}\xrightarrow[]{(d)} \BHP_\theta.
\]
\end{thm}
The space $\BHP_\theta$ satisfies the scaling property
$\lambda\cdot \BHP_\theta=_d\BHP_{\theta/\lambda^2}$. It was shown in
Remark 3.19 of~\cite{BaMiRa} that Aldous' self-similar continuum random tree
$\ICRT$, whose definition is reviewed in Section~\ref{sec:def-ICRT}, is the
asymptotic cone of the $\BHP_{\theta}$ around its root, implying
$\BHP_\theta\rightarrow\ICRT$ in law as $\theta\rightarrow\infty$.  In
particular, formally, we may think of the $\BHP_\infty$ as the $\ICRT$. In
view of Theorem~\ref{thm:GHconv-BHPtheta}, it is therefore natural to
expect that the $\ICRT$ appears also as the scaling limit of the
$\UIHPQ_{p_n}$, provided $\theta$ in the definition of $p_n$ is replaced by
a sequence $\theta_n\rightarrow\infty$, that is, if
$a_n^2(1-2p_n)\rightarrow\infty$ as $n\rightarrow\infty$. This is indeed
the case.
\begin{thm}
\label{thm:GHconv-ICRT}
Let $(a_n,n\in\N)$ be a sequence of positive reals with
$a_n\rightarrow\infty$. Let $(p_n,n\in\N)\subset [0,1/2]$ be a sequence
satisfying
\[
a_n^2(1-2p_n)\rightarrow\infty\quad\textup{as }n\rightarrow\infty.
\]
Then, in the sense of the local Gromov-Hausdorff topology as $n\rightarrow\infty$,
\[a_n^{-1}\cdot
\UIHPQ_{p_n}\xrightarrow[]{(d)} \ICRT.
\]
\end{thm}

As special cases of the previous two theorems, we mention
\begin{corollary}
\label{cor:1}
Let $p\in[0,1/2]$, and let $(a_n,n\in\N)$ be a sequence of positive reals
with $a_n\rightarrow\infty$. Then, in the sense of the local
Gromov-Hausdorff topology as $n\rightarrow\infty$,
\[a_n^{-1}\cdot
\UIHPQ_{p}\xrightarrow[]{(d)} \left\{
\begin{array}{lcl}
\ICRT& \mbox{ if }& 0\leq p<1/2\\
\BHP& \mbox{ if }&p=1/2
\end{array}\right..
\]
\end{corollary}
For the family $(\mathbb{H}_\alpha)_\alpha$ of half-planar triangulations
studied in~\cite{AnRa1,Ra}, convergence towards the $\ICRT$ in the
subcritical regime $\alpha< 2/3$ is conjectured in~\cite[Section
2.1.2]{Ra}.

\begin{figure}[ht!]
\centering\includegraphics[width=0.6\textwidth]{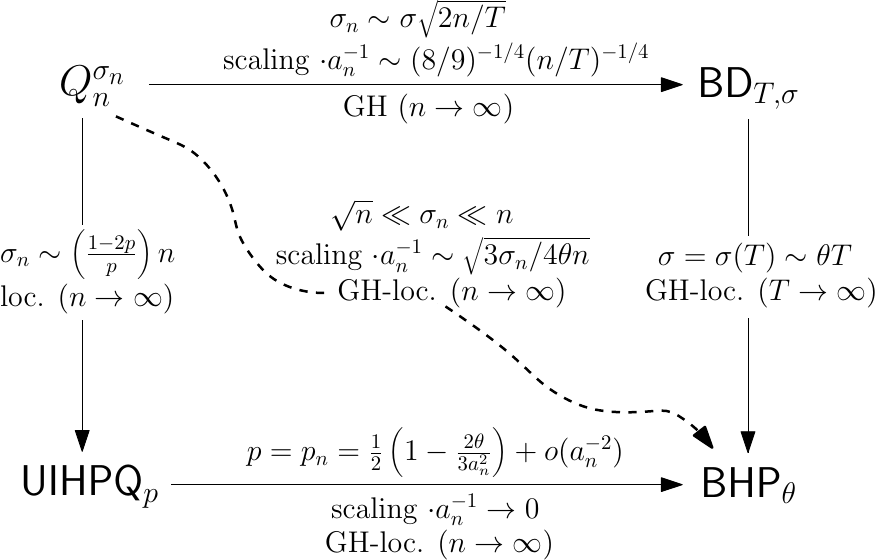}
\caption{Illustration of various convergences explaining the connections
  between the spaces $\UIHPQ_p$, $\BHP_\theta$ and $\BD_{T,\sigma}$. For
  simplicity, the cases $\theta=0$ and $\theta=\infty$ are left out. The
  top-most horizontal convergence represents~\cite[Theorem 1]{BeMi} and
  holds for $T,\sigma>0$ fixed. If the volume $T$ of $\BD_{T,\sigma}$ is
  blown up and the perimeter $\sigma$ grows linearly in $T$ such that
  $\sigma(T)\sim\theta T$, the space $\BHP_\theta$ appears as the
  distributional local Gromov-Hausdorff limit of the disks
  $\BD_{T,\sigma(T)}$ around their roots (\cite[Corollary
  3.17]{BaMiRa}). On the other hand, $\BHP_\theta$ is approximated by
  uniform quadrangulations $Q_n^{\sigma_n}$ (\cite[Theorem 3.4]{BaMiRa}),
  or by the $\UIHPQ_{p}$ when $p=p(a_n,\theta)$ depends in the right way on
  $\theta$ and $a_n$ (Theorem~\ref{thm:GHconv-BHPtheta}). The $\UIHPQ_p$
  for fixed $p\in[0,1/2]$ in turn arises as the local limit of
  $Q_n^{\sigma_n}$, provided the boundary lengths are properly chosen
  (Theorem~\ref{thm:local-conv}).}
\label{fig:diagram}
\end{figure}
\begin{remark}
  We stress that the spaces $\BHP_\theta$ can also be understood as
  Gromov-Hausdorff scaling limits of uniform quadrangulations
  $Q_n^{\sigma_n}\in\cQ_n^{\sigma_n}$; see~\cite[Theorems
  3.3,\,3.4,\,3.5]{BaMiRa}. More specifically, the $\BHP_\theta$ for
  $\theta\in(0,\infty)$ arises when $\sqrt{n}\ll \sigma_n\ll n$ and the
  graph metric is rescaled by a factor $a_n^{-1}$ satisfying
  $3\sigma_na_n^2/(4n)\rightarrow \theta$ as $n$ tends to infinity. The
  Brownian half-plane $\BHP_0$ corresponding to the choice $\theta=0$
  appears more generally when $1\ll \sigma_n \ll n$ and
  $1\ll a_n\ll\min\{\sqrt{\sigma_n},\sqrt{n/\sigma_n}\}$. Finally, the
  $\ICRT$ corresponding to $\theta=\infty$ appears when
  $\sigma_n\gg \sqrt{n}$ and
  $\max\{1,\,\sqrt{n/\sigma_n}\}\ll a_n\ll\sqrt{\sigma_n}$.

  We may as well view the spaces $\BHP_\theta$ as local scaling limits
  around the roots of the so-called Brownian disks $\BD_{T,\sigma}$ of
  volume $T>0$ and perimeter $\sigma>0$ introduced in~\cite{BeMi}.  More
  concretely, it was proved in~\cite[Corollaries 3.17,\,3.18]{BaMiRa} that
  when both $T$ and $\sigma=\sigma(T)$ tend to infinity such that
  $\sigma(T)/T\rightarrow\theta \in [0,\infty]$, then the $\BHP_\theta$ is the
  local Gromov-Hausdorff limit in law of the disk $\BD_{T,\sigma(T)}$ around a
  boundary point chosen according to the boundary measure of the latter.
  Figure~\ref{fig:diagram} depicts some convergences involving the families
  $\UIHPQ_p$ and $\BHP_\theta$.
\end{remark}

\subsection{Tree structure}\label{sec:TreeStructure}
We will prove that for $p<1/2$, the $\UIHPQ_p$ can be represented as a
collection of independent finite quadrangulations with a simple boundary
glued along a tree structure. The tree structure is encoded by the looptree
associated to a two-type version of Kesten's tree, and the finite
quadrangulations are distributed according to the Boltzmann distribution
$\widehat{\P}^{\sigma}_g$ on quadrangulations with a simple boundary of
size $2\sigma$. Precise definitions of the encoding objects are postponed
to Section~\ref{sec:halfplanes-trees}.

For $0\leq p \leq 1/2$, let $g_p=p(1-p)/3$ and $z_p=(1-p)/4$. Let $F(g,z)$
be the partition function of the Boltzmann measure on finite rooted
quadrangulations with a boundary, with weight $g$ per inner face and
$\sqrt{z}$ per boundary edge. Let moreover $\widehat{F}_k(g)$ be the
partition function of the Boltzmann measure on finite rooted
quadrangulations with a simple boundary of perimeter $2k$, with weight
$g$ per inner face. 

We introduce two probability measures $\muw$ and $\mub$ on $\N_0$ by setting
\begin{align*}
  \muw(k)&=\frac{1}{F(g_p,z_p)}\left(1-\frac{1}{F(g_p,z_p)}\right)^k, \quad k\in \N_0,\\
  \mub(2k+1)&=\frac{1}{F(g_p,z_p)-1}\left[z_p F^2(g_p,z_p)
              \right]^{k+1}\widehat{F}_{k+1}(g_p), \quad k\in \N_0, 
\end{align*}
with $\mub(k)=0$ if $k$ even. Exact expressions for $F(g_p,z_p)$ and
$\widehat{F}_{k+1}(g_p)$ are given in~\eqref{eq:EnumF} and
\eqref{eq:EnumFHat} below. The fact that $\mub$ is a probability
distribution is a consequence of Identity (2.8) in~\cite{BoGu}. We will
prove in Lemma~\ref{lem:Criticality} that the pair $(\muw,\mub)$ is
critical for $0\leq p<1/2$, in the sense that the product of their
respective means equals one, and subcritical if $p=1/2$, meaning that the
product of their means is strictly less than one. Moreover, both measures
have small exponential moments. Our main result characterizing the
structure of the $\UIHPQ_p$ for $0\leq p<1/2$ is the following.

\begin{thm}\label{thm:BranchingStructure}Let $0\leq p<1/2$, and let $\Loop(\cT_\infty)$ be
  the infinite looptree associated to Kesten's two-type tree
  $\mathcal{T}_\infty(\muw,\mub)$. Glue into each inner face of
  $\Loop(\cT_\infty)$ of degree $2\sigma$ an independent Boltzmann quadrangulation
  with a simple boundary distributed according to
  $\widehat{\P}^{\sigma}_{g_p}$. Then, the resulting infinite
  quadrangulation is distributed as the $\UIHPQ_p$.
\end{thm}
\begin{figure}[ht!]
  \centering
  \includegraphics[width=0.75\textwidth]{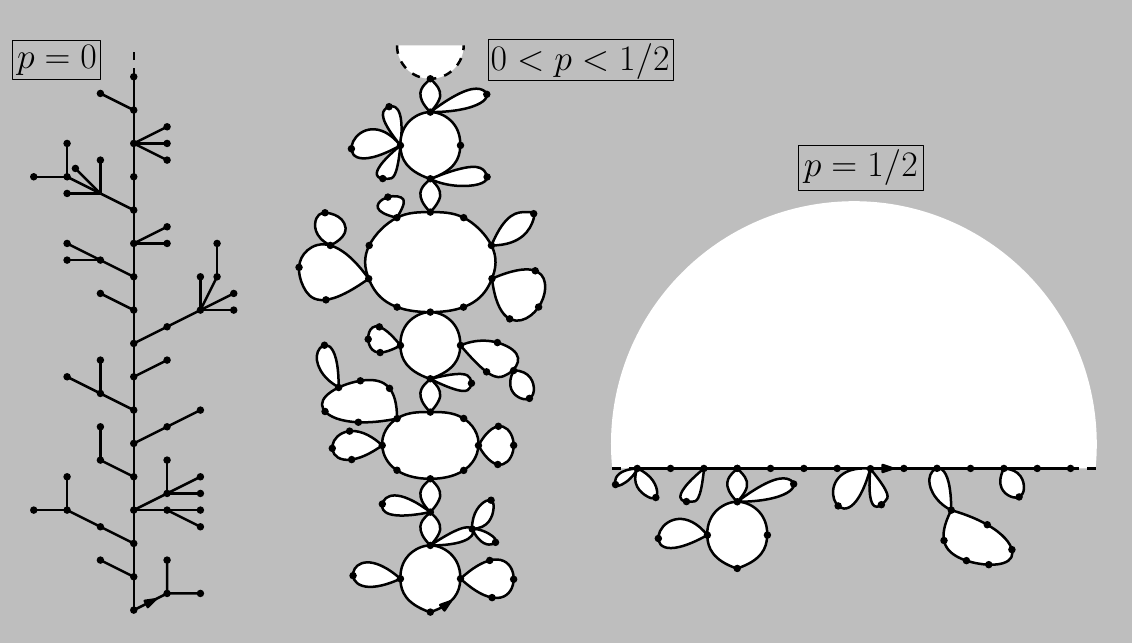}
  \caption{Schematic representation of the $\UIHPQ_p$ for $p\in[0,1/2]$. On
    the left: The $\UIHPQ_0$, that is, Kesten's tree associated to
    the critical geometric offspring distribution $\mu_{1/2}$. On the
    right: The standard uniform infinite half-planar quadrangulation
    $\UIHPQ$ with a general boundary. The white parts are understood to be filled in
    with quadrangulations, the big white semicircle representing the
    half-plane. In the middle: The $\UIHPQ_p$ with skewness parameter
    $p$. The white parts represent the (finite-size) quadrangulations
    with a simple boundary which are glued into the loops of the infinite
    looptree $\Loop(\cT_\infty)$ associated to a two-type version
    $\cT_\infty(\muw,\mub)$ of Kesten's tree.}
  \label{fig:UIHPQp}
\end{figure}
The gluing operation fills in each (rooted) loop a finite-size
quadrangulation with a simple boundary, which has the same perimeter as the
loop. The two boundaries are glued together, such that the root edges of
the loop and the quadrangulation get identified; see
Remark~\ref{rem:gluing}. Figure~\ref{fig:UIHPQp} depicts the above
representation of the $\UIHPQ_p$ in the case $0<p<1/2$, as well as the
borderline cases $p=0$ and $p=1/2$. The branching structure of the standard
$\UIHPQ=\UIHPQ_{1/2}$ has been investigated by Curien and
Miermont~\cite{CuMi}. They show that the $\UIHPQ$ can be seen as the uniform
infinite half-planar quadrangulation with a simple boundary (represented by
the big white semicircle in Figure~\ref{fig:UIHPQp}), together with
a collection of finite-size quadrangulations with a general boundary, which
are attached to the infinite simple boundary. 

\begin{remark}
 \label{rem:UIHPQ0Kesten}
  In the case $p=0$, the above theorem can be seen as a restatement of
  Proposition~\ref{prop:kesten}. Indeed, in this case, one finds that
  $\muw=\mu_{1/2}$ is the critical geometric probability law, and $\mub$ is
  the Dirac-distribution $\delta_1$. By construction, all the inner faces
  of $\Loop(\cT_\infty)$ have then degree $2$, and the
  gluing of a Boltzmann quadrangulation distributed according to
  $\widehat{\P}^{1}_{g_0=0}$ simply amounts to close the face, by identifying
  its edges. One finally recovers Kesten's (one-type) tree associated to
  the offspring law $\mu_{1/2}$, as already found in Proposition~\ref{prop:kesten}.
\end{remark}

\begin{remark}
  In~\cite{BaMiRi}, it has been proved that geodesics in the standard
  $\UIHPQ$ intersect both the left and right part of the boundary
  infinitely many times (see~\cite[Section 2.3.3]{BaMiRi} for the exact
  terminology). However, up to removing finite quadrangulations that hang
  off from the boundary, the $\UIHPQ$ has the topology of a
  half-plane. Consequently, left and right parts of the boundary intersect
  only finitely many times.  The branching structure described in
  Theorem~\ref{thm:BranchingStructure} implies that the left and right
  parts of the boundary of the $\UIHPQ_p$ for $p<1/2$ have {\it infinitely} many
  intersection points. As a consequence, any infinite self-avoiding path
  intersects both boundaries infinitely many times.
\end{remark}

Our tree-like description of the $\UIHPQ_p$ for $0\leq p<1/2$ readily
implies that simple random walk on the $\UIHPQ_p$ is recurrent. For $p=0$,
this result is due to Kesten~\cite{Ke}.
\begin{corollary}
\label{cor:recurrence}
Let $0\leq p< 1/2$. Almost surely, simple random walk on the $\UIHPQ_p$ is recurrent.
\end{corollary}
Somewhat informally, the tree structure describing the $\UIHPQ_p$ in the
case $p<1/2$ shows that there is an essentially unique way for the random
walk to move to infinity. Said otherwise, the walk reduces essentially to a
random walk on the half-line reflected at the origin, which is, of course,
recurrent. We give a precise proof in terms of electrical networks in
Section~\ref{sec:proofs-struct}.

\begin{remark}
  As far as the standard uniform infinite half-planar quadrangulation
  $\UIHPQ$ corresponding to $p=1/2$ is concerned, Angel and
  Ray~\cite{AnRa2} prove recurrence of the triangular analog with a simple
  boundary, the half-plane $\UIPT$. They construct a full-plane extension
  of the half-plane $\UIPT$ using a decomposition into layers and then
  adapt the methods of Gurel-Gurevich and Nachmias~\cite{GuNa}, and
  Benjamini and Schramm~\cite{BeSc}. It is believed that the arguments
  of~\cite{AnRa2} can be extended to the $\UIHPQ$, too. Ray proves
  in~\cite{Ra} of recurrence of the half-plane models $\mathbb{H}_{\alpha}$
  when $\alpha<2/3$. In~\cite{BjSt1}, Bj\"ornberg and Stef\'ansson prove
  that the (local) limit of bipartite Boltzmann planar maps is recurrent,
  for every choice of the weight sequence.
\end{remark}

We believe that the mean displacement of a random walker after $n$ steps on
the $\UIHPQ_p$ for $p<1/2$ is of order $n^{1/3}$, as for Kesten's tree (case
$p=0$). We will not pursue this further in this paper.

Let us finally mention another consequence of
Theorem~\ref{thm:BranchingStructure} concerning percolation
thresholds. See, e.g.,~\cite{AnCu} for the terminology of Bernoulli
percolation on random lattices.
\begin{corollary}
\label{cor:percolation}
Let $0\leq p< 1/2$. The critical thresholds for Bernoulli site, bond and
face percolation on the $\UIHPQ_p$ are almost surely equal to one.
\end{corollary}
Therefore, percolation on the $\UIHPQ_p$ changes drastically depending on
whether the skewness parameter $p$ (not to be confused the the percolation
parameter) is less or equal to $1/2$: In the standard
$\UIHPQ=\UIHPQ_{1/2}$, the critical thresholds are known to be $5/9$ for
site percolation, see~\cite{Ri}, and $1/3$ for edge percolation and $3/4$
for face percolation, see~\cite{AnCu}.  The proof of the corollary follows
immediately from Theorem~\ref{thm:BranchingStructure}.

\section{Random half-planes and trees}
\label{sec:halfplanes-trees}
In this section, we begin with a review of the one-parameter family of
Brownian half-planes $\normalfont{\BHP}_\theta$, $\theta\geq 0$, introduced
in~\cite{BaMiRa} (see also~\cite{GwMi} for the case $\theta=0$).

We then gather certain concepts around trees, which play an important role
throughout this paper.  We properly define the $\ICRT$, two-type
Galton-Watson trees and Kesten's infinite versions thereof, looptrees and
the so-called tree of components.
\subsection{The Brownian half-planes $\normalfont{\BHP}_\theta$}
\label{sec:def-BHPtheta}
We need some preliminary notation.  Given a function $f=(f_t,t\in\R)$, we
set $\underline{f}_t=\inf_{[0,t]}f$ for $t\geq 0$ and
$\underline{f}_{t}=\inf_{(-\infty,t]}f$ for $t<0$. Moreover, if
$f=(f_t,t\geq 0)$ is a real-valued function indexed by the non-negative reals,
its Pitman transform $\pi(f)$ is defined by
\[
\pi(f)_t=f_t-2\underline{f}_{t}.
\]
In case $B=(B_t,t\geq 0)$ is a standard one-dimensional Brownian motion,
its Pitman transform $\pi(B)=(\pi(B)_t,t\geq 0)$ is equal in law to a
three-dimensional Bessel process, which has in turn the law of the modulus
of a three-dimensional Brownian motion.

Now fix $\theta\in[0,\infty)$. The Brownian half-plane $\BHP_{\theta}$ with
skewness $\theta$ is defined in terms of its contour and label processes
$\Xha=(\Xha_t,t\in\R)$ and $\Wha=(\Wha_t,t\in\R)$. They are characterized
as follows.
\begin{itemize}
\item The process $(\Xha_t,t\geq 0)$ has the law of a one-dimensional Brownian motion
  $B=(B_t,t\geq 0)$ with
  drift $-\theta$ and $B_0=0$, and $(\Xha_{-t},t\geq 0)$ has the law of the Pitman
  transform of an independent copy of $B$.
\item Given $\Xha$, the (label) function $\Wha$ has same distribution as
    $(\gamma_{-\underline{X}_t^{\theta}}+\Zha_t,t\in \R)$, where
    \begin{itemize}
   \item The process $\Zha=(\Zha_t,t\in\R)=Z^{\Xha-\underline{X}^{\theta}}$ is a continuous
     modification of the centered Gaussian process with conditional covariances given by
      \[
      \E\left[\Zha_s\Zha_t\,|\,\right]=\min_{[s\wedge t,s\vee t]}\Xha-\underline{X}^{\theta},
      \]  
    
   \item The process $(\gamma_x,x\in \R)$ is a two-sided Brownian motion with
     $\gamma_0=0$ and scaled by the factor $\sqrt{3}$, independent of $\Zha$. 
     \end{itemize}
\end{itemize}
The process $\Zha$ is usually called the (head of the) random snake driven
by $\Xha-\underline{X}^{\theta}$, see~\cite{LG0} for more on this.  Next, we
define two pseudo-metrics $d_{\Xha}$ and $d_{\Wha}$ on $\R$,
\[
d_{\Xha}(s,t)=\Xha_s+\Xha_t-2\min_{[s\wedge t,s\vee t]}\Xha,\quad\textup{and}\quad
d_{\Wha}(s,t)=\Wha_s+\Wha_t-2\min_{[s\wedge t,s\vee t]}\Wha.
\]
The pseudo-metric $D_{\theta}$ associated to $\BHP_\theta$ is defined
as the maximal pseudo-metric $d$ on $\R$ satisfying $d\leq d_{\Wha}$ and
$\{d_{\Xha}=0\}\subseteq \{D_\theta=0\}$. According to Chapter $3$
of~\cite{BuBuIv}, it admits the expression ($s,t\in\R$)
\[
D_{\theta}(s,t)=\inf\left\{\sum_{i=1}^kd_{\Wha}(s_i,t_i):\begin{array}{l}
k\in\N, \, s_1,\ldots,s_k,t_1,\ldots,t_k\in\R,s_1=s,t_k=t,\\ 
d_{\Xha}(t_i,s_{i+1})=0\mbox{ for every }i\in \{1,\ldots,k-1\}
    \end{array}
    \right\}\,.
\]
\begin{defn}
  The Brownian half-plane $\BHP_\theta$ has the law of the pointed metric
  space $(\R/\{D_{\theta}=0\},D_{\theta},\rho_{\theta})$, with the
  distinguished point $\rho_{\theta}$ is given by the equivalence class of
  $0$.
\end{defn}
Note that $D_\theta$ stands here also for the induced metric on the
quotient space. It follows from standard scaling properties of $\Xha$ and $\Wha$ that for
$\lambda>0$, $\lambda\cdot \BHP_\theta=_d\BHP_{\theta/\lambda^2}$.  In
particular, $\BHP_0$ is scale-invariant. It was shown in~\cite{BaMiRa} that
for every $\theta\geq 0$, $\BHP_\theta$ has a.s. the topology of the closed
half-plane $\overline{\mathbb{H}}=\R\times\R_+$.

\subsection{Random trees and some of their properties}
\subsubsection{The self-similar continuum random tree $\normalfont{\ICRT}$}
\label{sec:def-ICRT}
Introduced by Aldous in~\cite{Al}, the $\ICRT$ is a random rooted real tree
that forms the non-compact analog of the usual continuum random tree
$\CRT$. Consider the stochastic process $(X_t,t\in\R)$ such that
$(X_t,t\geq 0)$ and $(X_{-t},t\geq 0)$ are two independent one-dimensional
standard Brownian motions started at zero. Define on $\R$ the
pseudo-metric
\[
d_{X}(s,t)=X_s+X_t-2\min_{[s\wedge t,s\vee t]}X.
\]
\begin{defn}
  The $\ICRT$ is the continuum random real tree $\cT_X$ coded by $X$, i.e.,
  the $\ICRT$ has the law of the pointed metric space
  $(\cT_X,d_{X},[0])$, where $\cT_X=\R/\{d_{X}=0\}$, and the
  distinguished point is given by the equivalence class of $0$.
\end{defn}
The $\ICRT$ is self-similar, meaning that $\lambda\cdot\ICRT=_d\ICRT$ for $\lambda>0$,
and invariant under re-rooting. We remark that the $\ICRT$ is often 
defined in terms of two independent three-dimensional Bessel processes
$(X_t,t\geq 0)$ and $(X_{-t},t\geq 0)$. Since the Pitman transform $\pi$
turns a Brownian motion into a three-dimension Bessel processes, it is
readily seen that both definitions give rise to the same random tree.

\subsubsection{Galton-Watson
  trees}\label{sec:GWTrees}
We recall the formalism of (finite or infinite) plane trees, i.e., rooted
ordered trees. The size $|\tr|\in \N_0\cup\{\infty\}$ of $\tr$ is given by
its number of edges, and we shall write $\mathbb{T}_f$ for the set of all
finite plane trees. 

We will often use the fact that if $\mathsf{GW}_{\nu}$ denotes the law of a
Galton-Watson tree with critical or subcritical offspring distribution
$\nu$, then
\begin{equation}
\label{eq:lawGW1}
\mathsf{GW}_{\nu}(\tr)=\prod_{u\in V(\tr)}\nu(k_u(\tr)),\quad t\in\mathbb{T}_f,
\end{equation}
where for $u\in V(\tr)$, $k_u(\tr)$ is the number of offspring of vertex
$u$. See, for example,~\cite[Proposition 1.4]{LG}). In the case where $\nu=\mu_p$ is the geometric offspring distribution
of parameter $1-p$ with $p\in [0,1/2]$,~\eqref{eq:lawGW1} becomes
\begin{equation}
\label{eq:lawGW2}
\mathsf{GW}_{\mu_p}(\tr)=p^{|\tr|}(1-p)^{|\tr|+1}.
\end{equation}
From the last display, the connection to random walks is apparent. Namely,
let $(S^{(p)}(m),m\in\N_0)$ be a random walk on the integers starting from
$S^{(p)}(0)=0$ with increments distributed according to
$p\delta_1 +(1-p)\delta_{-1}$. Define the first hitting time of $-1$,
\[T_{-1}^{(p)}=\inf\{m\in\N: S^{(p)}(m)=-1\}.\]
Then it is readily deduced from~\eqref{eq:lawGW2} that the size
$|\tr|$ of $\tr$ under $\mathsf{GW}_{\mu_p}$ and $(T^{(p)}_{-1}-1)/2$ are equal in
distribution. To be precise, by Kemperman's formula \cite[Section 6.1]{Pi}, we have
\[\P\left(T_{-1}^{(p)}=2n+1\right)=\frac{1}{2n+1}\P\left(S^{(p)}_{2n+1}=-1\right)=\frac{1}{2n+1}{2n+1 \choose n+1} p^n(1-p)^{n+1}, \quad n \in \N_0, \] and the $n$th Catalan number $\tfrac{1}{n+1}{2n \choose n}$ is precisely the number of plane trees with $n$ edges.

Given a finite or infinite plane tree, it will be convenient to say that
vertices at even height of $\tr$ are white, and those at odd height are
black. We use the notation $V(\tr_\circ)$ and $V(\tr_\bullet)$ for the associated
subsets of vertices. We next define two-type Galton-Watson trees associated
to a pair $(\nuw,\nub)$ of probability measures on $\N_0$.
\begin{defn}
  The {\it two-type Galton-Watson tree} with a pair of offspring
  distributions $(\nuw,\nub)$ is the random plane tree such that vertices
  at even height have offspring distribution $\nuw$, vertices at odd height
  have offspring distribution $\nub$, and the numbers of children of the
  different vertices are independent.
\end{defn}
In this context, the pair $(\nuw,\nub)$ is said to be {\it critical} if and
only if the mean vector $(m_\circ,m_\bullet)$ satisfies
$m_{\circ}m_{\bullet}=1$. Then, the law $\mathsf{GW}_{\nuw,\nub}$ of such a tree is characterized by
\[
\mathsf{GW}_{\nuw,\nub}(\tr)=\prod_{u\in V(\tr_\circ)}\nuw(k_u(\tr))\prod_{u\in V(\tr_\bullet)}\nub(k_u(\tr)),\quad t\in\mathbb{T}_f.
\]

\subsubsection{Kesten's tree and its two-type version}
\label{sec:def-kesten}
We next briefly review critical Galton-Watson trees conditioned to survive;
see~\cite{Ke} or~\cite{LyPe}, and~\cite{St} for the multi-type case.
 
\begin{prop}[Theorem 3.1 in~\cite{St}]\label{prop:KestenTreeCvgce}
  Let $\mathsf{GW}$ be the law of a critical (either one or two-type) Galton-Watson
  tree. For every $n\in\N$, assume that $\mathsf{GW}(\{\#V(\tr)=n\})>0$, and let
  $\cT_n$ be a tree with law $\mathsf{GW}$ conditioned to have $n$ vertices. Then,
  we have the local convergence for the metric $\dmap$ as
  $n\rightarrow\infty$ to a random infinite tree
  $\cT_\infty$,
\[\cT_n\overset{(d)}{\longrightarrow} \cT_\infty.\]  
\end{prop} 
In the case $\mathsf{GW}=\mathsf{GW}_\nu$ for $\nu$ a critical one-type
offspring distribution, $\cT_\infty$ is often called \textit{Kesten's tree}
associated to $\nu$, and simply Kesten's tree if $\nu=\mu_{1/2}$. We will
use the same terminology if $(\nuw,\nub)$ is a critical pair of offspring
distributions and $\mathsf{GW}=\mathsf{GW}_{\nuw,\nub}$. In this case, we
write $\cT_\infty(\nuw,\nub)$ for Kesten's tree associated to
$(\nuw,\nub)$. Note that the condition $\mathsf{GW}(\{\#V(\tr)=n\})>0$ can
be relaxed, provided we can find a subsequence along which this condition
is satisfied.

Galton-Watson trees conditioned to survive enjoy an explicit construction,
which we briefly recall for the two-type case. Details can be found
in~\cite{St}. Let $(\nuw,\nub)$ be a critical pair of offspring
distributions with mean $(m_\circ,m_\bullet)$, and recall that the
size-biased distributions $\bar{\nu}_\circ$ and $\bar{\nu}_\bullet$ are
defined by
\[\bar{\nu}_\circ(k)=\frac{k\nuw(k)}{m_\circ}\quad \text{and} \quad
\bar{\nu}_\bullet(k)=\frac{k\nub(k)}{m_\bullet}, \quad k\in \N_0.\]
Kesten's tree $\cT_\infty$ associated to $(\nuw,\nub)$ is an infinite
locally finite (two-type) tree that has a.s. a unique infinite
self-avoiding path called the \textit{spine}. It is constructed as
follows. The root vertex (white) is the first vertex on the spine. It has offspring distribution $\bar{\nu}_\circ$. Among its offspring,
a child (black) is chosen uniformly at random to be the second vertex on the
spine. It has offspring distribution $\bar{\nu}_\bullet$, and a
child (white) chosen uniformly at random among its offspring becomes the third
vertex on the spine. The spine is constructed by iterating this procedure.

The construction of the tree is completed by specifying that vertices at
even (resp.\ odd) height lying not on the spine have offspring distribution
$\nuw$ (resp.\ $\nub$), and that the numbers of offspring of the different
vertices are independent.

The construction is similar in the mono-type case. In the particular case
when $\nu=\mu_{1/2}$ is the geometric distribution with parameter $1/2$,
Kesten's tree can be represented by an infinite half-line (isomorphic to
$\N$) and a collection of independent Galton-Watson trees with law
$\mathsf{GW}_{\mu_{1/2}}$ grafted to the left and to the right of every
vertex on the spine; see, for instance,~\cite[Example 10.1]{Ja}. We will
exploit this representation in our proof of Proposition~\ref{prop:kesten}.

\subsubsection{Random looptrees}
\label{sec:RandomLooptrees}
Our description of the $\UIHPQ_p$ in Theorem~\ref{thm:BranchingStructure}
makes use of so-called \textit{looptrees}, which were introduced
in~\cite{CuKo1}. A looptree can informally be seen as a collection of loops
glued along a tree structure. The following presentation is inspired
by~\cite[Section 2.3]{CuKo2}. We use, however, slightly different
definitions which are better suited to our purpose. In particular, given a
plane tree $\tr$, we will only replace vertices $v\in V(\tr_\bullet)$ at
{\it odd} height by loops of length $\deg(u)$. Consequently, several loops
may be attached to one and the same vertex (at even height). 

Let us now make things more precise. Let $\tr$ be a finite plane tree, and
recall that vertices at even height are white, and those at odd height are
black (with respective subsets of vertices $V(\tr_\circ)$ and
$V(\tr_\bullet)$). We associate to $\tr$ a rooted looptree $\Loop(\tr)$ as
follows. Around every (black) vertex in $V(\tr_\bullet)$, we connect its
incident white vertices in cyclic order, so that they form a loop. Then
$\Loop(\tr)$ is the planar map obtained from erasing the black vertices and
the edges of $\tr$. We root $\Loop(\tr)$ at the edge connecting the origin
of $\tr$ to the last child of its first sibling in $\tr$; see
Figure~\ref{fig:TreeOfComponentsQuad}.

The reverse application associates to a looptree $\lt$ a plane tree, which
we call the \textit{tree of components} $\Tree(\lt)$. In order to obtain
$\Tree(\lt)$ from $\lt$, we add a new vertex in every internal face of
$\lt$ and connect this vertex to all the vertices of the face. The root edge of $\Tree(\lt)$ connects the origin of $\lt$ to the
new vertex added in the face incident to the left side of the root edge of
$\lt$.
\begin{figure}[h!]
  \begin{center}
    \includegraphics[scale=.95]{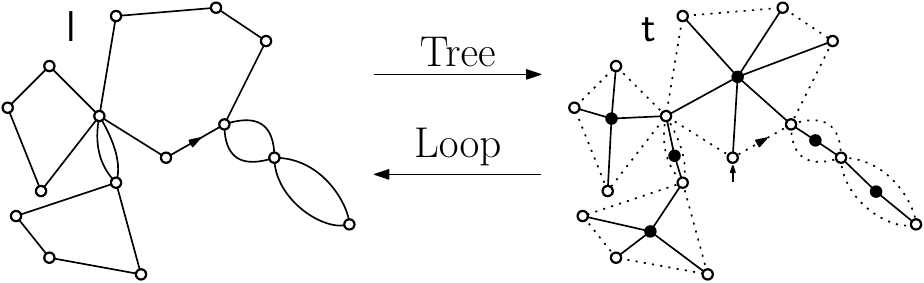}
  \end{center}
  \caption{A looptree and the associated tree of components.}
  \label{fig:TreeOfComponentsQuad}
\end{figure}

The procedures $\Tree$ and $\Loop$ extend to infinite but locally
finite trees, by considering the consistent sequence of maps
$\{\Loop(\cb_{2k}(\tr)) : k\in\N_0\}$. We will be interested in the random infinite
looptree associated to Kesten's two-type tree.

\begin{defn}
\label{def:Kestenslooptree}
If $(\nuw,\nub)$ is a critical pair of offspring laws and 
$\cT_\infty$ the corresponding Kesten's tree, we call the
random infinite looptree $\Loop(\cT_\infty)$
\textit{Kesten's looptree} associated to $\cT_\infty$.
\end{defn}
Note that a formal way to construct $\Loop(\cT_\infty)$ is to define it as
the local limit of $\Loop(\cT_n)$, where $\cT_n$
is a two-type Galton-Watson tree with offspring distribution $(\nuw,\nub)$ 
conditioned to have $n$ vertices. 

\begin{remark}
\label{rem:gluing}
In a looptree $\lt$, every loop is naturally rooted at the edge whose
origin is the closest vertex to the origin of $\lt$, such that the outer
face of $\lt$ lies to the right of that edge.
The gluing of a (rooted) quadrangulation with a simple boundary of perimeter
$2\sigma$ into a loop of the same length is then determined by the convention
that the root edge of the quadrangulation is glued on the root edge of the loop.
\end{remark}

\section{Construction of the $\normalfont{\UIHPQ}_p$}
\label{sec:UIHPQp}
A Schaeffer-type bijection due to Bouttier, Di Francesco and
Guitter~\cite{BoDFGu} encodes quadrangulations with a boundary in terms of
labeled trees that are attached to a bridge. We shall first describe a
bijective encoding of finite-size planar quadrangulations, and then extend
it to infinite quadrangulations with an infinite boundary. This will allow
us to construct and define the $\UIHPQ_{p}$ for $p\in[0,1/2]$ in terms of
the encoding objects, which we define first.

\subsection{The encoding objects}
We briefly review well-labeled trees, forests, bridges and
contour and label functions. Our notation bears similarities
to~\cite{CuMi,CaCu,BaMiRa}, differs, however, at some places. Each of these
references already contains the construction of the standard $\UIHPQ$.
\subsubsection{Forest and bridges} 
A {\it well-labeled tree} $(\tr,\ell)$ is a pair consisting of a finite
rooted plane tree $\tr$ and a labeling $(\ell(u))_{u\in V(\tr)}$ of its
vertices $V(\tr)$ by integers, with the constraints that the root vertex
receives label zero, and $|\ell(u)-\ell(v)| \leq 1$ if $u$ and $v$ are
connected by an edge.

A {\it well-labeled forest} with $\sigma\in\N$ trees is a pair $(\f,\la)$,
where $\f=(\tr_0,\ldots,\tr_{\sigma-1})$ is a sequence of $\sigma$ rooted
plane trees, and $\la :V(\f)\rightarrow\mathbb{Z}$ is a labeling of the
vertices $V(\f)=\cup_{i=0}^{\sigma-1}V(\tr_i)$ such that for every
$0\leq i\leq\sigma-1$, the pair $(\tr_i,\la\restriction V(\tr_i))$ is a
well-labeled tree. Similarly, a {\it well-labeled infinite forest} is a
pair $(\f,\la)$, where $\f=(\tr_{i},i\in\Z)$ is an infinite collection of
rooted plane trees, together with a labeling
$\la :\cup_{i\in\Z}V(\tr_i)\rightarrow\mathbb{Z}$ such that for each
$i\in\Z$, the restriction of $\la$ to $V(\tr_i)$ turns $\tr_i$ into a
well-labeled tree.

A {\it bridge of length }$2\sigma$ for $\sigma\in\N$ is a sequence
$\br=(\br(0),\br(1),\ldots,\br(2\sigma-1))$ of $2\sigma$ integers with
$\br(0)=0$ and $|\br(i+1)-\br(i)|=1$ for
$0\leq i\leq 2\sigma-1$, where we agree that
$\br(2\sigma)=0$. In a similar manner, an {\it infinite bridge} is a
two-sided sequence $\br=(\br(i):i\in\Z)$ with $\br(0)=0$ and
$|\br(i+1)- \br(i)|=1$ for all $i\in\Z$.

Given a bridge $\br$, an index $i$ for which $\br(i+1)=\br(i)-1$ is called a
{\it down-step} of $\br$. The set of all down-steps of $\br$ is denoted
$\DS(\br)$. If $\br$ is a bridge of length $2\sigma$, $\DS(\br)$ has
$\sigma$ elements, and we write $\dwst_{\br}(i)$ for
the $i$th smallest element in $\DS(\br)$, for $i=1,\ldots,\sigma$.  If $\br$
is an infinite bridge and $i\in\N$, $\dwst_{\br}(i)$ denotes the $i$th
smallest element in $\DS(\br)\cap\N_0$, and $\dwst_{\br}(-i)$ denotes the
$i$th largest element in $\DS(\br)\cap\Z_{<0}$. If there is no
danger of confusion, we write simply $\dwst$ instead of $\dwst_{\br}$.

The size of a forest $\f$ is the number $|\f|\in\N_0\cup\{\infty\}$ of tree
edges. If $\f=(\tr_0,\dots,\tr_{\sigma-1})$ and $u\in V(\tr_i)$, we write
$\textup{H}_{\f}(u)$ for the height of $u$ in the tree $\tr_i$, i.e., the
graph distance to the root of $\tr_i$. Moreover, $\mathcal{I}_\f(u)=i$
denotes the index of the tree the vertex $u$ belongs to. Both
$\textup{H}_{\f}$ and $\mathcal{I}_\f$ extend in the obvious way to
infinite forests. If it is clear which forest we are
referring to, we drop the subscript $\f$ in H and $\mathcal{I}$.

We let
$\Fo_\sigma^n = \{(\f,\la): \f\textup{ has }\sigma\textup{ trees and size }
|\f| = n\}$
be the set of all well-labeled forests of size $n$ with $\sigma$ trees and
write $\Fo_\infty$ for the set of all well-labeled infinite forests. The
set of all bridges of length $2\sigma$ is denoted $\Br_{\sigma}$. As far as
infinite bridges are concerned, it will be sufficient to consider only
those bridges $\br$ which satisfy $\inf_{i\in \N}\br(i)=-\infty$ and
$\inf_{i\in \N}\br(-i)=-\infty$, and we denote the set of them by
$\Br_\infty$.

\subsubsection{Contour and label function}
\label{sec:contour-label-fct}
We first consider the case $((\f,\la),\br)\in\Fo_\sigma^n\times\Br_\sigma$
for some $n,\sigma\in\N$. By a slight abuse of notation, we write
$\f(0),\ldots,\f(2n+\sigma-1)$ for the contour exploration of $\f$, that
is, the sequence of vertices (with multiplicity) which we obtain from
walking around the trees $\tr_0,\ldots,\tr_{\sigma-1}$ of $\f$, one after
the other in the contour order. See the left side of
Figure~\ref{fig:forest-finite}. We define the {\it contour function} of
$(\f,\la)$ by
\[
C_\f(j) =\textup{H}(\f(j))-\mathcal{I}(\f(j)), \quad 0\leq j\leq 2n+\sigma-1.
\]
Note that $C_\f(2n+\sigma-1)=-\sigma-1$, since the last visited vertex by
the contour exploration is the root of $\tr_{\sigma-1}$. We extend
$C_\f$ to $[0,2n+\sigma]$ by first letting $C_\f(2n+\sigma)=-\sigma$, and
then by linear interpolation between integers, so that $C_\f$ becomes a
continuous real-valued function on $[0,2n+\sigma]$ starting at zero and
ending at $-\sigma$. 

The {\it label function} associated to $((\f,\la),\br)$ is obtained from
shifting the vertex label $\la(\f(j))$  by the value of the
bridge $\br$ evaluated at its $(\mathcal{I}(\f(j))+1)$th down-step. Formally,
\[
\La_{\f}(j) = \la(\f(j)) +
\br\left(\dwst\left(\mathcal{I}(\f(j))+1\right)\right),\quad 0\leq j\leq 2n+\sigma-1.
\]
We let $\La_{\f}(2n+\sigma)=0$ and again linearly interpolate between
integer values, so that $\La_\f$ becomes an element of
$\mathcal{C}([0,2n+\sigma],\R)$.  Contour and label functions are depicted
on the right side of Figure~\ref{fig:forest-finite}.

\begin{figure}[ht]
\begin{minipage}{1\linewidth}
  \parbox{8.3cm}{\center\includegraphics[width=0.35\textwidth]{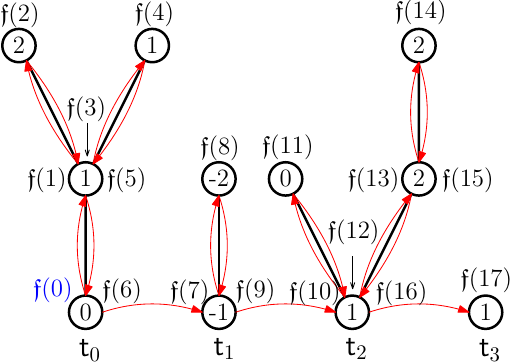}}
\parbox{5.1cm}{\center\includegraphics[width=0.45\textwidth]{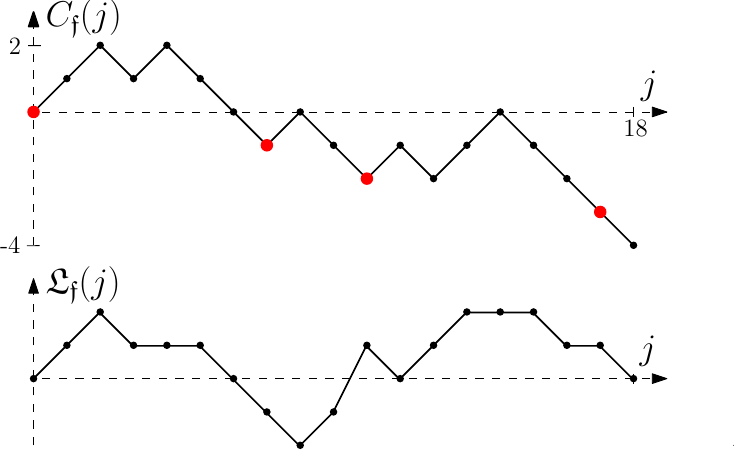}}
\end{minipage}
\caption{Contour and label functions $C_\f$ and $\La_{\f}$ of an element
  $((\f,\la),\br)\in\Fo_4^7\times\Br_4$.  The left side depicts the contour
  exploration of $\f$. The labels on the vertices are given by
  $\La_{\f}(j)$, $j=0,\ldots,18$. Note that the values of $\br$ at its four
  down-steps are equal to the values of $\La_{\f}$ at the tree roots: In
  this example, we have $\br(\dwst(1))=0$, $\br(\dwst(2))=-1$, and
  $\br(\dwst(3))=\br(\dwst(4))=1$. The red dots on the right indicate the
  encoding of a new tree.}
  \label{fig:forest-finite}
\end{figure}

In the case $((\f,\la),\br)\in\Fo_\infty\times\Br_\infty$, we explore the
trees of $\f$ in the following way: First, $(\f(0),\f(1),\ldots)$ is the
sequence of vertices of the contour paths of the trees $\tr_i, i\in\N_0$,
in the left-to-right order, starting from the root of $\tr_0$. Then, we let
$(\f(-1),\f(-2),\ldots)$ be the sequence of vertices of the contour paths
$\tr_{-1},\tr_{-2},\ldots$, in the counterclockwise or right-to-left
order, starting from the root of $\tr_{-1}$; see the left side of 
Figure~\ref{fig:forest-infinite}.
\begin{figure}[ht]
\begin{minipage}{1\linewidth}
\parbox{8.35cm}{\center\includegraphics[width=0.40\textwidth]{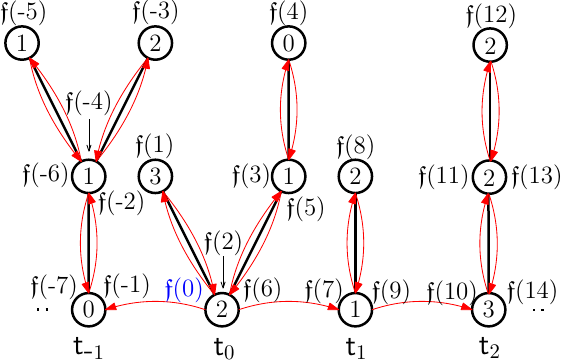}}
\parbox{5cm}{\center\includegraphics[width=0.40\textwidth]{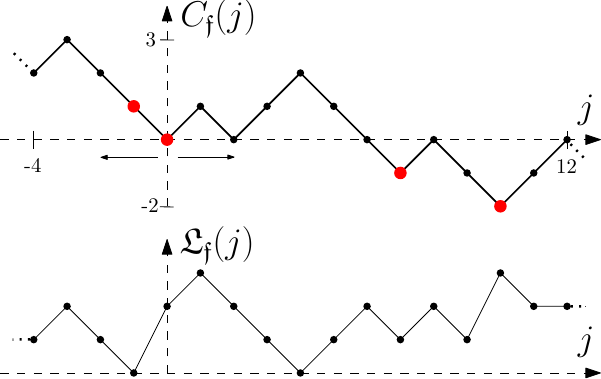}}
\end{minipage}
\caption{Contour and label functions $C_\f$ and $\La_{\f}$ of an element
  $((\f,\la),\br)\in\Fo_\infty\times\Br_\infty$.  The left side depicts the
  two-sided contour exploration of $\f$. The labels are given by
  $\La_{\f}(j)$, where now $j\in\Z$. The values of the infinite bridge
  $\br$ at its first three down-steps to the right of $0$ read here
  $\br(\dwst(1))=2$, $\br(\dwst(2))=1$ and $\br(\dwst(3))=3$, while the
  first down-step to the left of zero has value $\br(\dwst(-1))=0$. The
  arrows below the contour function indicate the direction of the encoding,
  and the red dots mark again the encoding of a new tree.}
\label{fig:forest-infinite}
\end{figure}
Contour and label functions $C_{\f}$ and $\La_{\f}$ are defined similarly to
the finite case, namely
\begin{align*}
C_{\f}(j) &=\textup{H}(\f(j))-\mathcal{I}(\f(j)),\quad j\in\Z,\\
\La_\f(j) &= \la(\f(j)) +
            \br\left(\dwst\left(\mathcal{I}(\f(j))+1\right)\right),\quad
            j\in\Z_{\geq 0},\\
\La_\f(j) &= \la(\f(j)) + \br\left(\dwst\left(\mathcal{I}(\f(j))\right)\right),\quad  j\in\Z_{<0}.
\end{align*}
Note that the asymmetry in the definition of $\La_{\f}$ stems from the
numbering of the trees.  By linear interpolation between integer values, we
interpret $C_\f$, $\La_\f$, and sometimes also $\la$, as continuous
functions (from $\R$ to $\R$).

\subsection{The Bouttier-Di Francesco-Guitter mapping}
\label{sec:BDG-bijection}
We denote the set of all rooted pointed quadrangulations with $n$ inner
faces and $2\sigma$ boundary edges by
 \[\cQ_n^{\sigma,\bullet}=\left\{(\q,\vd) :
   \q\in\cQ_n^{\sigma}, \vd\in V(\q)\right\},\] where $\vd$ stands for the
 distinguished pointed vertex. In the following part, we briefly recall the
 definition of the bijection 
 $\Phi_n:\Fo_\sigma^n \times \Br_{\sigma}\rightarrow \cQ_n^{\sigma,\bullet}$
 introduced in~\cite{BoDFGu}.

 \subsubsection{The encoding of finite quadrangulations}
 \label{sec:encoding-finitequad}
 We represent an element $((\f,\la),\br)\in\Fo_\sigma^n\times\Br_{\sigma}$
 in the plane as follows. Firstly, we view $\br$ as a labeled cycle of length
 $2\sigma$: We start from a distinguished vertex labeled $\br(0)=0$ and
 label the remaining $2\sigma-1$ vertices in the counterclockwise order by
 the values $\br(1),\br(2),\ldots,\br(2\sigma-1)$. Then we graft the trees
 $(\tr_0,\dots,\tr_{\sigma-1})$ of $\f$ to the $\sigma$ down-steps
 $0\leq i_0<i_1<\dots <i_{\sigma-1}\leq 2\sigma-1$ of $\br$, such that
 $\tr_j$ is grafted on the vertex corresponding to the value $\br(i_j)$, in
 the interior of the cycle. We do it in such a way that different trees do
 not intersect.  The vertices of $\tr_j$ are equipped with their labels
 shifted by $\br(i_j)$. Figure~\ref{fig:treedbridge-BDG} illustrates this
 procedure.
 
 \begin{figure}[ht]
\centering
  \includegraphics[width=0.35\textwidth]{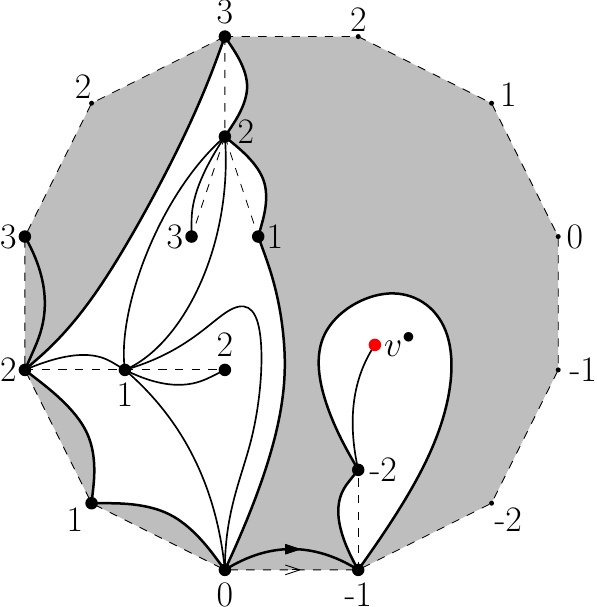}
  \caption{A representation of an element
    $((\f,\la),\br)\in\Fo_6^6\times\Br_{6}$ in the plane and the associated
    rooted pointed quadrangulation $(\q, \vd)=\Phi_n((\f,\la),\br)$. The
    distinguished vertex of the cycle is the down-most vertex labeled
    $0$. The trees are grafted to the $6$ down-steps of $\br$ (here,
    $\dwst(1)=0$, $\dwst(2)=1$, $\dwst(3)=7$, $\dwst(4)=9$, $\dwst(5)=10$,
    and $\dwst(6)=11$). The tree edges are indicated by the dashed lines in
    the interior of the cycle. Note that three trees (those above the
    first, fourth and sixth down-step) consist of a single vertex. The
    labels in a tree are shifted by the bridge value of the down-step above
    which the tree is attached. Note that the 12 boundary edges of the
    cycle are in a order-preserving correspondence with the 12 boundary
    edges of $\q$. (The two edges of $\q$ which lie entirely in the outer face
    are counted twice.)}
\label{fig:treedbridge-BDG}
\end{figure}

We now build a rooted and pointed quadrangulation $(\q,\vd)$ out of
$((\f,\la),\br)$. First, we put an extra vertex $\vd$ in the interior of
the cycle representing $\br$. The set of vertices of $\q$ is given by the
tree vertices $V(\f)\cup\{\vd\}$. As for the edges of $\q$, we define for
$0\leq i\leq 2n+\sigma-1$ the successor
$\suc(i)\in [0,2n+\sigma-1]\cup\{\infty\}$ of $i$ to be the first element
$k$ in the list $(i+1,\ldots,2n+\sigma-1,0,\ldots,i-1)$ (from left to
right) which has label $\La_{\f}(k)=\La_{\f}(i)-1$. If there is no such
element, we put $\suc(i)=\infty$. We extend the contour exploration
$\f(0),\ldots,\f(2n+\sigma-1)$ of $\f$ by setting $\f(\infty)=\vd$. We
follow the exploration starting from the vertex $\f(0)$ (which is the root
of $\tr_0$) and draw for each $0\leq i\leq 2n+\sigma-1$ an arc between
$\f(i)$ and $\f(\suc(i))$, such that arcs do not cross. Except for the
leaves, a vertex of $\f$ is visited at least twice in the contour
exploration, so that there are in general several arcs connecting the
vertices $\f(i)$ and $\f(\suc(i))$. The edges of $\q$ are given by all these
arcs between the vertices $V(\f)\cup\{\vd\}$.

It only
remains to root the quadrangulation. To that aim, we observe from
Figure~\ref{fig:treedbridge-BDG} that the $2\sigma$ boundary edges of $\q$
are in a order-preserving correspondence with the $2\sigma$ cycle edges. We
root $\q$ at the edge corresponding to the first edge of the cycle
(starting from the distinguished edge, in the clockwise order), oriented in
such a way that the face of degree $2\sigma$ becomes the outer face (i.e.,
lies to the right of the root edge). Upon erasing the tree and cycle edges
of the representation of $((\f,\la),\br)$, and the vertices of $\br$
corresponding to up-steps, we obtain a rooted pointed quadrangulation
$(\q, \vd)$. A description of the reverse mapping
$\Phi_n^{-1}:\cQ_n^{\sigma,\bullet}\rightarrow \Fo_\sigma^n \times
\Br_{\sigma}$ can be found in~\cite{BoDFGu} or~\cite{Be}.

 \subsubsection{The encoding of infinite quadrangulations}
\label{sec:encoding-infinitequad}
Recall that $\cQ$ is the completion of the set of finite rooted
quadrangulations with a boundary with respect to $\dmap$. The aim of this
section is to extend $\Phi_n$ to a mapping
\[\Phi:\left({\cup}_{n,\sigma\in\N}\Fo_{\sigma}^n\times\Br_\sigma\right)\cup\left(\Fo_\infty\times\Br_\infty\right)\longrightarrow
\cQ.\]
We proceed as
follows. If $((\f,\la),\br)\in \Fo_{\sigma}^n\times\Br_\sigma$, we put
$\Phi((\f,\la),\br)=\Phi_n((\f,\la),\br)$. (We forget the
distinguished vertex of $\Phi_n((\f,\la),\br)$ and view the quadrangulation as an
element in $\cQ_n^{\sigma}\subset\cQ$.)

Now assume $((\f,\la),\br)\in\Fo_\infty\times\Br_\infty$. We consider the
following representation of $((\f,\la),\br)$ in the upper half-plane: First,
we identify $\br$ with the bi-infinite line obtained from connecting
$i\in\Z$ to $i+1$ by an edge. Vertex $i$ is labeled $\br(i)$. We attach the
trees $\tr(0),\tr(1),\ldots$ of $\f$ to the down-steps of $\br$ to the
right of $0$, and the trees $\tr(-1),\tr(-2),\ldots$ to
the down-steps of $\br$ to the left of $-1$, everything
in the upper half-plane. Again, the labels in a tree are shifted by the
underlying bridge label.
\begin{figure}[ht]
  \centering
  \includegraphics[width=0.85\textwidth]{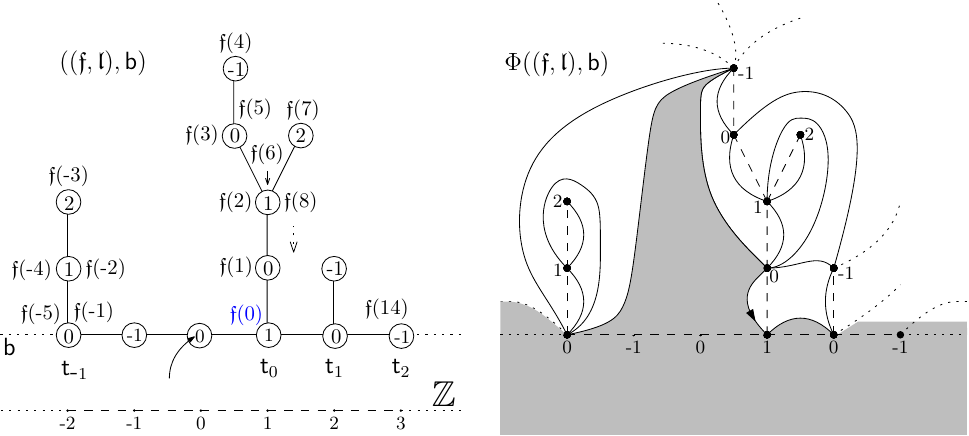}
  \caption{ The Bouttier-Di Francesco-Guitter mapping applied to an element
    $((\f,\la),\br)\in\Fo_\infty\times\Br_\infty$. On the right-hand side, the
    arcs connect the vertices $\f(i)$ with $\f(\suc_\infty(i))$, for
    $i\in\Z$. The other vertices and edges of the representation of
    $((\f,\la),\br)$ on the left-hand side do not appear in the
    quadrangulation. The oriented arc on the right indicated by
    an arrow represents the root edge of the map.}
  \label{fig:BDG-infinite}
\end{figure}

Similarly to the finite case, the vertex set of $\q=\Phi((\f,\la),\br)$ is
given by $V(\f)$; here, we add no additional vertex. For specifying the
edges, we let the successor $\suc_\infty(i)$ of $i\in\Z$ be the smallest
number $k>i$ such that $\La_{\f}(k)=\La_{\f}(i)-1$. Since by assumption
$\inf_{i\in\N}\br(i)=-\infty$, $\suc_\infty(i)$ is a finite number. We next
connect the vertices $\f(i)$ and $\f(\suc_\infty(i))$ by an arc for any
$i\in\Z$, such that the resulting map is planar. The arcs form the edges of
the infinite rooted quadrangulation $\q$ we are about to construct. In
order to root the map, we observe that the bi-infinite line $\Z$ is in
correspondence with the boundary edges of $\q$, and we choose the edge
corresponding to $\{0,1\}$ as the root edge of $\q$ (oriented such
that the outer face lies to its right). A representation of $((\f,\la),\br)$
and of the resulting quadrangulation $\Phi((\f,\la),\br)$ is depicted in
Figure~\ref{fig:BDG-infinite}.

\subsection{Definition of the $\normalfont{\UIHPQ}_p$}
\label{sec:defUIHPQp}
We are now in position to construct the $\UIHPQ_p$ by means of the above
mapping $\Phi$ applied to a (random) element in
$\Fo_\infty\times\Br_\infty$, which we introduce first.

Let $\tr$ be a finite random plane tree. Conditionally on $\tr$, we assign
to $\tr$ a random uniform labeling $\ell$ of its vertices, so that the pair
$(\tr,\ell)$ becomes a well-labeled tree. Namely, given $\tr$, we first
equip each edge of $\tr$ with an independent random variable uniformly
distributed in $\{-1,0,1\}$. Then we define the label $\ell(u)$ of a vertex
$u\in V(\tr)$ to be the sum over all labels along the unique
(non-backtracking) path from the tree root to $u$.

We consider Galton-Watson trees with a (sub-)critical geometric offspring law
$\mu_p$ of parameter $1-p$ with $p\in[0,1/2]$, that is,
$
\mu_p(k)=p^{k}(1-p)$, $k\in\N_0.
$
 If $\tr$ is such a tree, we
call it a {\it $p$-Galton-Watson tree}.  Equipped with a random uniform
labeling $\ell$ as described before, we say that the pair
$(\tr,(\ell(u))_{u\in V(\tr)})$ is a {\it uniformly labeled
  $p$-Galton-Watson tree}. 

A {\it uniformly labeled infinite $p$-forest} is a random element
$(\f_\infty^{(p)},\la_\infty^{(p)})$ taking values in $\Fo_\infty$, such that
$(\tr_i,\la_\infty^{(p)}\restriction V(\tr_i))$, $i\in\Z$, are independent
uniformly labeled $p$-Galton-Watson trees.
 
A {\it uniform infinite bridge} is a random element
$\br_\infty=(\br_\infty(i),i\in\mathbb{Z})$ in $\Br_\infty$ such that
$\br_\infty$ has the law of a two-sided simple symmetric random walk
starting from $\br_\infty(0)=0$. We stress that our wording differs
from~\cite{BaMiRa}, where a uniform infinite bridge refers to a two-sided
random walk with a geometric offspring law of parameter $1/2$. See also
Lemma~\ref{lem:bridge} below.

 \begin{defn}
\label{def:UIHPQp}
Fix $p\in[0,1/2]$. Let $(\f_\infty^{(p)},\la_\infty^{(p)})$ be a uniformly
labeled infinite $p$-forest, and independently of
$(\f_\infty^{(p)},\la_\infty^{(p)})$, let $\br_\infty$ be a uniform
infinite bridge. Then the $\UIHPQ_p$ with skewness parameter $p$ is given
by the (rooted) random infinite quadrangulation
$Q_{\infty}^{\infty}(p)=(V(Q_{\infty}^{\infty}(p)),\dgr,\rho)$ with an
infinite boundary, which is obtained from applying the Bouttier-Di
Francesco-Guitter mapping $\Phi$ to
$((\f_\infty^{(p)},\la_\infty^{(p)}),\br_\infty)$.  In case $p=1/2$, we
simply write $Q_\infty^\infty$, which denotes then the (standard) uniform
infinite half-planar quadrangulation with a general boundary.
\end{defn}

\begin{remark}
\label{rem:balls}
Let $\f_\infty^{(p)}$ be the encoding forest of the $\UIHPQ_p$. Instead of
working with metric balls around the root vertex in the $\UIHPQ_p$, it will
-- due to the specific construction of the latter -- often be more
practical to consider metric balls around the vertex corresponding to the
tree root $\f_\infty^{(p)}(0)$ in the $\UIHPQ_p$. Similarly, if
$Q_n^{\sigma}\in\cQ_n^{\sigma}$ is a uniform quadrangulation and $\f_n$ its
encoding forest, it will be more natural to consider balls around $\f_n(0)$
in $Q_n^{\sigma}$. Since the distance between $\f_\infty^{(p)}(0)$ or
$\f_n(0)$ and the root of the map is stochastically bounded (it may also be
zero), this makes no difference in terms of scaling limits whatsoever;
see~\cite[Lemma 5.6]{BaMiRa}. We shall use the notation
$B^{(0)}_{r}(Q_\infty^{\infty}(p))$ for the metric ball of radius $r$
around $\f_\infty^{(p)}(0)$ in the $\UIHPQ_p$. Analogously, we define
$B^{(0)}_{r}(Q_n^{\sigma})$.
\end{remark}

\section{Proofs of the limit results}
\label{sec:proofs-limits}
\subsection{The $\normalfont{\UIHPQ}_p$ as a local limit of
  uniform quadrangulations}
In this part, we prove Theorem~\ref{thm:local-conv} and
Proposition~\ref{prop:kesten}. We begin with the former. The case $p=1/2$
has already been treated in~\cite{BaMiRa}, and the case $p=0$ will be
considered afterwards, so we first fix $0<p< 1/2$ and let
$(\sigma_n,n\in\N)$ be a sequence of positive integers satisfying
$\sigma_n=\frac{1-2p}{p}n + o(n)$.  Recall that rooted pointed
quadrangulations in $\cQ_{n}^{\sigma_n,\bullet}$ are in one-to-one
correspondence with elements in $\Fo_{\sigma_n}^n\times\Br_{\sigma_n}$. For
proving Theorem~\ref{thm:local-conv}, the key step is to control the law of
the first $k$ trees in a forest $\f_n$ chosen uniformly at random in
$\Fo_{\sigma_n}^n$, for $k$ arbitrarily large but fixed. We will see in
Lemma~\ref{lem:pforest} below that their law is close to the law of $k$
independent $p$-Galton-Watson trees when $n$ is sufficiently
large. Together with a convergence result of bridges
(Lemma~\ref{lem:bridge}), this allows us to couple contour and label
functions of $Q_n^{\sigma_n}$ and the $\UIHPQ_p$, such that with high
probability, we have equality of balls of a constant radius around the
roots in $Q_n^{\sigma_n}$ and the $\UIHPQ_p$, respectively. This readily
implies the theorem.

We begin with the necessary control over the trees. Since the result on the
tree convergence is of some interest on its own, we formulate an optimal
version, which is stronger than we what need for mere local convergence as stated
in Theorem~\ref{thm:local-conv}. 

\begin{lemma}
\label{lem:pforest}
Fix $0<p< 1/2$, and let $(\sigma_n,n\in\N)$ be a sequence of positive
integers satisfying $\sigma_n=\frac{1-2p}{p}n+o(n)$. Let
$(\tr_i)_{1\leq i\leq \sigma_n}$ be a family of $\sigma_n$ independent
$1/2$-Galton-Watson trees, and let $(\tr_i^{(p)})_{1\leq i\leq \sigma_n}$ be a
family of $\sigma_n$ independent $p$-Galton-Watson trees. Then, if
$(k_n,n\in\N)$ is a sequence of positive integers satisfying
$k_n\leq \sigma_n$ and $k_n=o\left(n\right)$ as
$n\rightarrow\infty$, we have
\[
\lim_{n\rightarrow\infty}\left\|\textup{Law}\Big((\tr_i)_{1\leq i\leq
    k_n}\,\Big|\,\sum_{i=1}^{\sigma_n}|\tr_i|=n\Big)
  -\textup{Law}\left((\tr_i^{(p)})_{1\leq i\leq
      k_n}\right)\right\|_{\textup{TV}}=0.
\]
\end{lemma}
\begin{remark}
We stress that in particular, we can choose $k_n$ equals an 
arbitrary large constant $k\in\N$. This suffices to show
local convergence towards the $\UIHPQ_p$; see
Proposition~\ref{prop:couplingQnUIHPQ} below. Lemma~\ref{lem:pforest} may
be seen as a complement to the results on coupling of trees
in~\cite{BaMiRa}; it treats a regime not considered in that work.
\end{remark}

\begin{proof} Let
$(S^{(p)}(m),m\in\N_0)$ be a random walk on the integers starting from
$S^{(p)}(0)=0$ with increments distributed according to
$p\delta_1 +(1-p)\delta_{-1}$. Set, for $\ell\in\Z$, 
\[
T^{(p)}_{\ell}=\inf\left\{m\in\N: S^{(p)}(m)=\ell\right\}.
\] We also let $(S(m),m\in\N_0)$ be a simple symmetric random walk started
from $S(0)=0$ and write $T_{\ell}$ for its first hitting time of
$\ell\in\Z$. By the encoding of a forest by its contour function described
in Section \ref{sec:contour-label-fct}, the claim of the lemma boils down
to
\begin{align}\label{eq:DTV}
	\sup_{\ell \in\N}\sup_{\mathbf{x}=(x_0, \ldots,x_l) \in \Z^{\ell+1}} \Big| &\P\left( T^{(p)}_{-k_n}=\ell, \ \left( S^{(p)}(0), \ldots, S^{(p)}(\ell) \right)=\mathbf{x} \right)\\\nonumber
	 &- \P\left( T_{-k_n}=\ell, \ \left( S(0), \ldots, S(\ell) \right)=\mathbf{x} \ \big| \ T_{-\sigma_n}=2n+ \sigma_n \right) \Big| \longrightarrow 0  
\end{align} as $n \rightarrow \infty$. First, observe that $S(1)$ can be
obtained as the Cram\'er transform of $S^{(p)}(1)$, meaning that  
\[\P(S(1)=k) = \frac{\lambda_p^k}{G(p)} \P\left( S^{(p)}(1) = k \right),
\quad k \in \{-1,1\},\]
where $\lambda_p=\sqrt{ \tfrac{1-p}{p} }$ and
$G(p)= p \lambda_p + (1-p)/\lambda_p$. Let us fix $\ell \in \N$ and
$\mathbf{x}\in \Z^{\ell+1}$. We have
\begin{align*}
	&\P \left( T_{-k_n} = \ell, \ \left( S(0), \ldots, S(\ell) \right)=\mathbf{x}, \ T_{-\sigma_n}=2n+\sigma_n \right)\\
	&\qquad = \sum_{\gamma} \prod_{j=1}^{2n+\sigma_n}{\P\left( S(j)-S(j-1) = \gamma_j-\gamma_{j-1} \right)}\\
	&\qquad=\frac{\lambda_p^{- \sigma_n}}{(G(p))^{2n+\sigma_n}} \sum_{\gamma} \prod_{j=1}^{2n+\sigma_n}{\P\left( S^{(p)}(j)-S^{(p)}(j-1) = \gamma_j-\gamma_{j-1} \right)}\\
	&\qquad=\frac{\lambda_p^{- \sigma_n}}{(G(p))^{2n+\sigma_n}} \P \left( T^{(p)}_{-k_n} = \ell, \ \left( S^{(p)}(0), \ldots, S^{(p)}(\ell) \right)=\mathbf{x}, \ T^{(p)}_{-\sigma_n}=2n+\sigma_n \right),
\end{align*} where the sums are over all paths $\gamma:\{0,
\ldots,2n+\sigma_n\} \rightarrow \Z$ for which the probabilities on the
right-hand side are non-zero. By the same argument, we obtain
\[\P \left( T_{-\sigma_n}=2n+\sigma_n \right)=\frac{\lambda_p^{- \sigma_n}}{(G(p))^{2n+\sigma_n}} \P \left( T^{(p)}_{-\sigma_n}=2n+\sigma_n \right),\] so that finally
\begin{align*}
&\P\left( T_{-k_n}=\ell, \ \left( S(0), \ldots, S(\ell) \right)=\mathbf{x} \ \big| \ T_{-\sigma_n}=2n+ \sigma_n \right)\\
&\qquad=\P\left( T^{(p)}_{-k_n}=\ell, \ \left( S^{(p)}(0), \ldots, S^{(p)}(\ell) \right)=\mathbf{x} \ \big| \ T^{(p)}_{-\sigma_n}=2n+ \sigma_n \right).
\end{align*} By applying the Markov property at time $T^{(p)}_{-k_n}$, we have
\begin{align*}
  &\P\left( T^{(p)}_{-k_n}=\ell, \ \left( S^{(p)}(0), \ldots, S^{(p)}(\ell) \right)=\mathbf{x} \ \big| \ T^{(p)}_{-\sigma_n}=2n+ \sigma_n \right)\\
  &\qquad=\frac{1}{\P\left(\ T^{(p)}_{-\sigma_n}=2n+ \sigma_n \right)}
    \mathbb{E}\left( \mathbf{1}_{\left\lbrace T^{(p)}_{-k_n}=\ell, \
    \left(S^{(p)}(0), \ldots, S^{(p)}(\ell) \right)=\mathbf{x}
    \right\rbrace}  \P \left( T^{(p)}_{-\sigma_n}=2n+ \sigma_n \ \big| \ T^{(p)}_{-k_n}=\ell \right)\right)\\
  &\qquad=\P\left( T^{(p)}_{-k_n}=\ell, \  \left(S^{(p)}(0), \ldots, S^{(p)}(\ell) \right)=\mathbf{x}\right)\frac{\P \left( T^{(p)}_{-\sigma_n+k_n}=2n+ \sigma_n-\ell \right)}{\P\left(\ T^{(p)}_{-\sigma_n}=2n+ \sigma_n \right)}\,. 
\end{align*} Note that we can assume, without loss of generality, that $k_n \rightarrow \infty$ as $n \rightarrow \infty$. Now, by the law of large numbers, since $(S^{(p)}(m),m\in\N_0)$ has negative drift we have that
\[\P\left( T^{(p)}_{-k_n}>Mk_n \right)\leq \P\left( S^{(p)}_{Mk_n}> -k_n
\right)=\P\left( \frac{S^{(p)}_{Mk_n}}{M k_n}> -
  \frac{1}{M}\right)\longrightarrow 0\]
as $n \rightarrow \infty$, provided that $M$ is large enough. As a
consequence, we may restrict ourselves to the values
$\ell \in \{1, \ldots, Mk_n\}$. By Kemperman's formula \cite[Section
6.1]{Pi}, we get
\[\frac{\P \left( T^{(p)}_{-\sigma_n+k_n}=2n+ \sigma_n-\ell
  \right)}{\P\left(\ T^{(p)}_{-\sigma_n}=2n+ \sigma_n
  \right)}=\frac{\sigma_n - k_n}{2n+\sigma_n-\ell} \cdot
\frac{2n+\sigma_n}{\sigma_n} \cdot \frac{\P \left( S^{(p)}_{2n+
      \sigma_n-\ell}= -\sigma_n+k_n\right)}{\P \left(
    S^{(p)}_{2n+\sigma_n}=- \sigma_n \right)}.\]
Since we assumed that $\ell \leq Mk_n$ and $k_n=o(\sigma_n)$ we have
\[\lim_{n \rightarrow \infty}{\frac{\sigma_n -
    k_n}{2n+\sigma_n-\ell}}=\lim_{n \rightarrow
  \infty}{\frac{\sigma_n}{2n+\sigma_n}}=1-2p,\]
so that by the local limit theorem (see \cite{IbLi} for instance),
\[\sup_{1 \leq \ell \leq Mk_n}\frac{\P \left( S^{(p)}_{2n+ \sigma_n-\ell}=
    -\sigma_n+k_n\right)}{\P \left( S^{(p)}_{2n+\sigma_n}=- \sigma_n
  \right)}\longrightarrow 1\]
as $n \rightarrow \infty$, which yields \eqref{eq:DTV} and completes the proof.  \end{proof}

We continue with a convergence result for uniform bridges
$\br_n\in\Br_{\sigma_n}$ towards $\br_\infty$. 
\begin{lemma}
\label{lem:bridge}
Let $(\sigma_n,n\in\N)$ be a sequence of positive integers satisfying
$\sigma_n\rightarrow\infty$ as $n\rightarrow\infty$. Let $\br_n$ be
uniformly distributed in $\Br_{\sigma_n}$, and let $\br_\infty$ be a
uniform infinite bridge as specified in Section~\ref{sec:UIHPQp}. Then, if
$k_n$ is a sequence of positive integers with $k_n\leq \sigma_n$ and
$k_n=o(\sigma_n)$ as $n\rightarrow\infty$,
\begin{equation*}
\begin{split}\lim_{n\rightarrow\infty}&\left\|\textup{Law}((\br_n(2\sigma_n-k_n),\ldots,\br_n(2\sigma_n-1),\br_n(0),\br_n(1),\ldots,\br_n(k_n)))\right.\\
  &\left. -\,\textup{Law}((\br_\infty(-k_n),\ldots,\br_\infty(-1),\br_\infty(0),\br_\infty(1),\ldots,\br_\infty(k_n)))\right\|_{\textup{TV}}=0.\end{split}
\end{equation*}
\end{lemma}
The proof follows from a small adaption of~\cite[Proof of Lemma
5.5]{BaMiRa} and is left to the reader. Roughly speaking, it relies on the
exact computation of the probability that
$(\br_n(2\sigma_n-k_n),\ldots,\br_n(2\sigma_n-1),\br_n(0),\ldots,\br_n(k_n))$
and
$(\br_\infty(-k_n),\ldots,\br_\infty(-1),\br_\infty(0),\ldots,\br_\infty(k_n))$
equal a fixed sequence $(x_0,\ldots,x_{2k_n-1})\in \Z^{2k_n}$, which
involves binomial coefficients by definition of bridges. We stress,
however, that in~\cite{BaMiRa}, $\br_n$ and $\br_\infty$ were defined in a
slightly different manner, by grouping the $+1$-steps between two
subsequent down-steps together to one ``big'' jump. Clearly, this does
change the argument only in a minor way.

We are now in position to formulate an appropriate coupling of balls.
\begin{prop}
\label{prop:couplingQnUIHPQ}
Fix $0<p< 1/2$, and let $(\sigma_n,n\in\N)$ be a sequence of positive
integers satisfying $\sigma_n=\frac{1-2p}{p}n+o(n)$. Let also $(\xi_n,n\in\N)$ be a sequence of positive
integers satisfying 
$\xi_n=o(\sqrt{n})$.  Then, given any
$\varepsilon>0$, there exist $\delta>0$ and $n_0\in\N$ such that for every
$n\geq n_0$, we can construct on the same probability space copies of
$Q_n^{\sigma_n}$ and the $\UIHPQ_p$ such that with probability at least
$1-\varepsilon$, the metric balls $B_{\delta \xi_n}(Q_n^{\sigma_n})$ and
$B_{\delta \xi_n}(\UIHPQ_p)$ of radius $\delta \xi_n$ around the roots in
the corresponding spaces are isometric.
\end{prop}
The local convergence of $Q_n^{\sigma_n}$ towards $\UIHPQ^{(p)}$ is a
weaker statement, hence Theorem~\ref{thm:local-conv} in the case $0<p<1/2$
will follow from the proposition.

\begin{proof}
  The proof is in spirit of~\cite[Proof of Proposition 3.11]{BaMiRa},
  requires, however, some modifications. We will indicate at which place we
  may simply adapt the reasoning. We consider a random uniform
  element $((\f_n,\la_n),\br_n)\in\Fo_{\sigma_n}^n$, and a triplet
  $((\f_\infty^{(p)},\la_\infty^{(p)}),\br_\infty)$ consisting of a
  uniformly labeled infinite $p$-forest together with an (independent)
  uniform infinite bridge $\br_\infty$. We let
  $(Q_n^{\sigma_n},\vd)=\Phi_n((\f_n,\la_n),\br_n)$ and
  $Q_\infty^\infty(p)=\Phi((\f_\infty^{(p)},\la_\infty^{(p)}),\br_{\infty})$
  be the quadrangulations obtained from applying the Bouttier-Di
  Francesco-Guitter mapping to $((\f_n,\la_n),\br_n)$ and
  $((\f_\infty^{(p)},\la_\infty^{(p)}),\br_{\infty})$, respectively. Recall
  that $\f_n=(\tr_0,\ldots,\tr_{\sigma_n-1})$ consists of $\sigma_n$
  trees. For $0\leq k\leq\sigma_n-1$, we let $\tr(\f_n,k)=\tr_{k}$, i.e.,
  $\tr(\f_n,k)$ is the tree of $\f_n$ with index $k$, and we put
  $\tr(\f_n,\sigma_n)=\tr(\f_n,0)$.  In a similar manner,
  $\tr(\f_\infty^{(p)},k)$ denotes the tree of $\f_\infty^{(p)}$ indexed by
  $k\in\mathbb{Z}$.

  By Lemma~\ref{lem:pforest}, we find $\delta'>0$ and $n'_0\in\N$ such that
  for $n\geq n'_0$, we can construct $((\f_n,\la_n),\br_n)$ and
  $((\f_\infty^{(p)},\la_\infty^{(p)}),\br_{\infty})$ on the same
  probability space such that with $A_n=\lfloor \delta'\xi_n^2\rfloor$, the
  event
\begin{align*}
  \mathcal{E}^1(n,\delta') &=
                             \quad\left\{\tr(\f_n,i)=\tr(\f_\infty^{(p)},i),\,\tr(\f_n,\sigma_n-i)=\tr(\f_\infty^{(p)},-i)
                             \textup{ for
                             all }0\leq
                             i\leq A_n\right\}\\
                           &\,\quad\cap\left\{\la_n|_{\tr(\f_n,i)}=
                             \la_\infty^{(p)}|_{\tr(\f_\infty^{(p)},i)},\,\la_n|_{\tr(\f_n,\sigma_n-i)}=\la_\infty^{(p)}|_{\tr(\f_\infty^{(p)},-i)}
                             \textup{ for
                             all }0\leq i
                             \leq A_n\right\}
\end{align*}
has probability at least $1-\eps/8$. We now fix such 
a $\delta'$ for the rest of the proof. Recall that by our construction of
the Bouttier-Di Francesco-Guitter bijection, the trees of
$\f_n$ are
attached to the down-steps $\dwst_{n}(i)=\dwst_{\br_n}(i)$ of $\br_n$,
$1\leq i\leq \sigma_n$, and similarly, the trees of
$\f_\infty^{(p)}$ are attached to the down-steps
$\dwst_{\infty}(i)=\dwst_{\br_\infty}(i)$ of $\br_\infty$, where now $i\in\Z$. In view
of the above event, this incites us to consider additionally the event
\begin{align*}
\mathcal{E}^2(n,\delta') &= \quad\left\{\br_n(i)=\br_\infty(i)\textup{ for
                           all }1\leq
                             i\leq \dwst_{\infty}(A_n+1)\right\}\\
&\,\quad\cap\left\{\br_n(2\sigma_n+i)=\br_\infty(i)\textup{ for
                           all }\dwst_{\infty}(-A_n)\leq i\leq -1\right\}.
\end{align*}
Note that on $\mathcal{E}^2(n,\delta')$, we automatically have
$\dwst_n(i)=\dwst_{\infty}(i)$ for $1\leq i\leq A_n+1$, and
$\dwst_n(\sigma_n-i+1)=\dwst_\infty(-i)$ for $1\leq i\leq A_n$. Trivially,
we have that $\dwst_{\infty}(A_n+1)\geq A_{n}+1$ and
$\dwst_{\infty}(-A_n)\leq -A_{n}$, but also, with probability tending to
$1$, $\dwst_{\infty}(A_n+1)\leq 3A_n$ and $\dwst_{\infty}(-A_n)\geq
-3A_n$.
Since, in any case, $A_n= o(\sigma_n)$, we can ensure by
Lemma~\ref{lem:bridge} that the event $\mathcal{E}^2(n,\delta')$ has
probability at least $1-\eps/8$ for large $n$.

Now for $\delta>0$, $n\in\N$, define the events
 \[\mathcal{E}^3(n,\delta)=\left\{\min_{[0,\,\dwst_{\infty}(A_n+1)]}\br_\infty
   <-5\delta \xi_n,\,
   \min_{[\dwst_{\infty}(-A_n),-1]}\br_\infty<-5\delta \xi_n\right\},\]
\[\mathcal{E}^4(n,\delta)=\left\{\min_{[\dwst_{\infty}(A_n+1)+1,\,\dwst_{\infty}(-A_n)-1]}\br_n
  <-5\delta \xi_n\right\}.\]
By invoking Donsker's invariance principle together with
Lemma~\ref{lem:bridge} for the event $\mathcal{E}^3$ (and again the fact
that $A_{n}+1\leq \dwst_{\infty}(A_n+1)\leq 3A_n$ and
$-3A_n\leq \dwst_{\infty}(-A_n)\leq -A_n$ with high probability), we deduce
that for small $\delta>0$, provided $n$ is large enough,
\[
\P\left(\mathcal{E}^3(n,\delta)\right)\geq 1-\eps/8, \quad\textup{and}\quad \P\left(\mathcal{E}^4(n,\delta)\right)\geq 1-\eps/8.
\]
We will now assume that $n_0\geq n'_0$ and $\delta>0$ are such that for all
$n\geq n_0$, the above bounds hold true, and work on the event
$\mathcal{E}^1(n,\delta')\cap\mathcal{E}^2(n,\delta')\cap\mathcal{E}^3(n,\delta)\cap\mathcal{E}^4(n,\delta)$
of probability at least $1-\eps/2$. We consider the forest obtained from
restricting $\f_n$ to the first $A_n+1$ and the last $A_n$ trees,
\[\f'_n=\left(\tr(\f_n,0),\ldots,\tr(\f_n,A_n),\tr(\f_n,\sigma_n-A_n),\ldots,\tr(\f_n,\sigma_n-1)\right).\]
Similarly, we define ${\f_\infty^{'(p)}}$. We recall the cactus bounds in
the version stated in~\cite[(4.4) of Section 4.5]{BaMiRa}. Applied to
$Q_n^{\sigma_n}$, it shows that for vertices
$v\in V(\f_n)\setminus V(\f'_n)$, with $d_n$ denoting the graph distance,
\[
d_n(\f_n(0),v) \geq -\max\left\{\min_{[0,\,\dwst_\infty(A_n+1)]}\br_n,
  \min_{[\dwst_{\infty}(-A_n)\,,2\sigma_n]}\br_n\right\} \geq 5\delta \xi_n.
\]
Applying now the analogous cactus bound~\cite[(4.6) of Section 4.5]{BaMiRa}
to the infinite quadrangulation $Q_\infty^\infty(p)$, we obtain the same
lower bound for vertices
$v\in V(\f_\infty^{(p)})\setminus V(\f^{'(p)}_\infty)$, with $d_n$ replaced
by the graph distance $d_\infty^{(p)}$ in $Q_\infty^\infty(p)$, and
$\f_n(0)$ replaced by the vertex $\f_\infty^{(p)}(0)$ of
$Q_\infty^\infty(p)$. We recall the definition of the metric balls
$B_r^{(0)}(Q_n^{\sigma_n})$ and $B_r^{(0)}(Q_\infty^\infty(p))$; see
Remark~\ref{rem:balls}. With the same arguments as in~\cite[Proof of
Proposition 3.11]{BaMiRa}, we then deduce that vertices at a distance at
most $5\delta\xi_n-1$ from $\f_n(0)$ in $Q_n^{\sigma_n}$ agree with
those at a distance at most $5\delta\xi_n-1$ from $\f_\infty^{(p)}(0)$ in
$Q_\infty^{\infty}(p)$. Moreover,
\[
d_n(u,v)=d_{\infty}^{(p)}(u,v)\quad\hbox{whenever }u,v\in B^{(0)}_{2\delta\xi_n}(Q_n^{\sigma_n}).
\]
This proves that the balls $B^{(0)}_{2\delta\xi_n}(Q_n^{\sigma_n})$ and
$B^{(0)}_{2\delta \xi_n}(Q_\infty^{\infty}(p))$ are isometric on an event
of probability at least $1-\eps/2$. In order to conclude, it suffices to
observe that the distances from $\f_n(0)$ resp. $\f_\infty^{(p)}(0)$
to the root vertex in $Q_n^{\sigma_n}$ resp. $Q_\infty^{\infty}(p)$ are
stochastically bounded; see again Remark~\ref{rem:balls}. Clearly, this
implies that with probability tending to $1$ as $n$ increases, we have the
inclusions
$B_{\delta \xi_n}(Q_n^{\sigma_n})\subset B^{(0)}_{2\delta
  \xi_n}(Q_n^{\sigma_n})$
and
$B_{\delta \xi_n}(Q_{\infty}^{\infty}(p))\subset B_{2\delta
  \xi_n}^{(0)}(Q_{\infty}^{\infty}(p))$.
\end{proof}
As mentioned at the beginning, the case $p=1/2$ has already been treated
in~\cite[Proof of Proposition 3.11]{BaMiRa}: It is proved there that for
$\delta$ small, balls of radius
$\delta \min\{\sqrt{\sigma_n},\sqrt{n/\sigma_n}\}$ in $Q_n^{\sigma_n}$ and
in the standard $\UIHPQ=\UIHPQ_{1/2}$ can be coupled with high probability,
implying of course again local convergence of $Q_n^{\sigma_n}$ towards the
$\UIHPQ$.

Finally, it remains to consider the case $p=0$ corresponding to
$\sigma_n\gg n$. This case is easy. We have the following
coupling lemma.
\begin{lemma}
\label{lem:couplingQnUIHPQ}
Let $(\sigma_n,n\in\N)$ be a sequence of positive integers satisfying
$\sigma_n\gg n$. Put $\xi_n=\sigma_n/n$. Then, given any
$\varepsilon>0$, there exist $\delta>0$ and $n_0\in\N$ such that for every
$n\geq n_0$, we can construct on the same probability space copies of
$Q_n^{\sigma_n}$ and the $\UIHPQ_0$ such that with probability at least
$1-\varepsilon$, the metric balls $B_{\delta \xi_n}(Q_n^{\sigma_n})$
and $B_{\delta \xi_n}(\UIHPQ_0)$ of radius $\delta \xi_n$
around the roots in the corresponding spaces are isometric.
\end{lemma}
\begin{proof}
  Let $((\f_n,\la_n),\br_n)\in\Fo_{\sigma_n}^n\times\Br_{\sigma_n}$ be
  uniformly distributed.  By exchangeability of the trees, it follows that
  if $k_n=o(\sigma_n/n)$, then the first and last $k_n$ trees of $\f_n$ are
  all singletons with a probability tending to one. Applying
  Lemma~\ref{lem:bridge}, we can ensure that the event
\[
\left\{\br_n(i)=\br_\infty(i),\,\br_n(2\sigma_n-i)=\br_\infty(-i),\,1\leq
                             i\leq k_n\right\}
\]
has a probability as large as we wish, provided $n$ is large enough. Given
$\varepsilon>0$, the same arguments as in the proof of
Proposition~\ref{prop:couplingQnUIHPQ} yield an equality of balls
$B_{\delta \xi_n}(Q_n^{\sigma_n})$ and $B_{\delta \xi_n}(\UIHPQ_0)$ for
$\delta$ small and $n$ large enough, on an event of probability at least
$1-\varepsilon$.
\end{proof}

Let us now show that the space $\UIHPQ_0$ defined in terms of the
Bouttier-Di Francesco-Guitter mapping in Section~\ref{sec:defUIHPQp} is
nothing else than Kesten's tree associated to the critical geometric
offspring law $\mu_{1/2}$.
\begin{proof}[Proof of Proposition~\ref{prop:kesten}] Let
  $\br_\infty=(\br_\infty(i),i\in\mathbb{Z})$ be a uniform infinite bridge,
  and let $(\f_\infty^{(0)},\la_\infty^{(0)})$ be the infinite forest where
  all trees are just singletons (with label $0$); see
  Section~\ref{sec:defUIHPQp}. The $\UIHPQ_0$ is distributed as the
  infinite map
  $Q_\infty^\infty(0)=\Phi((\f_\infty^{(0)},\la_\infty^{(0)}),\br_\infty)$. Since
  every vertex in $\f_\infty^{(0)}$ defines a single corner, properties of
  the Bouttier-Di Francesco-Guitter mapping
  (Section~\ref{sec:BDG-bijection}) imply that $Q_\infty^\infty(0)$ is a
  tree almost surely.  Moreover, the set of vertices of
  $Q_\infty^\infty(0)$ is identified with the set of down-steps
  $\DS(\br_\infty)$ of the bridge.  Following~\cite[Section 2.2.3]{BaMiRi},
  conditionally on $\br_\infty$, we introduce a function
  $\varphi:\Z\rightarrow \DS(\br_\infty)$ that associates to $i\in \Z$ the
  next down-step larger than $i$ with label $\br_\infty(i)$ (and $i$ is mapped to
  itself if $i\in \DS(\br_\infty)$). According to our rooting convention,
  the root edge of $Q_\infty^\infty(0)$ connects $\varphi(0)$ to
  $\varphi(1)$. Note that $\varphi$ is not injective almost surely.

  We recall that Kesten's tree can be represented by a half-line of
  vertices $s_0,s_1,\ldots,$ together with a collection of independent
  Galton-Watson trees with offspring law $\mu_{1/2}$ grafted to the left
  and right side of each vertex $s_i$, $i\in\N_0$. We will now argue that
  the $\UIHPQ_0$ $Q_\infty^\infty(0)$ has the same structure. In this
  regard, let us introduce the stopping
  times\[ S_i=\inf \{ k\in \N_0 : \br_\infty(k)=-i \}, \quad i\in \N_0, \]
  and denote by $s_i$ the vertex of $Q_\infty^\infty(0)$ given by
  $\varphi(S_i)$. Together with their connecting edges, the collection
  $(s_i, i\in\N_0)$ forms a spine (i.e., an infinite self-avoiding path)
  in $Q_\infty^\infty(0)$.

  The subtree rooted at $s_i$ on the left side of the spine is encoded by
  the excursion $\{\br_\infty(k) : S_{i} \leq k \leq S_{i+1} \}$, in a way
  we describe next; see Figure~\ref{fig:kestens-tree} for an
  illustration. First note that by the Markov property, these subtrees for
  $i\in\Z$ are i.i.d.. In order to determine their law, let us consider the
  subtree encoded by the excursion $\{\br_\infty(k) : 0\leq k \leq S_1 \}$
  of $\br_\infty$. This subtree is rooted at $s_0=\emptyset$, and the
  number of offspring of $s_0$ is the number of down-steps with label $1$
  between $0$ and $S_1$. Otherwise said, this is the number
  $\#\{ 0 < k < S_1 : \br_\infty(k)=0 \}$ of excursions of $\br_\infty$
  above $0$ between $0$ and $S_1$. By the Markov property, this quantity
  follows the geometric distribution $\mu_{1/2}$ of parameter $1/2$. One
  can now repeat the argument for each child of $s_0$, by considering the
  corresponding excursion above $0$ encoding its progeny tree, inside the
  mother excursion. We obtain that the subtree stemming from $s_0$ on the
  left of the spine has indeed the law of a Galton-Watson tree with
  offspring distribution $\mu_{1/2}$.

  The subtrees attached to the vertices $s_i$, $i\in\N_0$, on the right of
  the spine can be treated by a symmetry argument. Namely, letting
  \[ S'_i=\inf \{ k\in\N_0 : \br_\infty(-k)=-i \}, \quad i\in \N_0, \]
  we observe that the subtree rooted at $s_i$ to the right of the spine is
  coded by the (reversed) excursion
  $\{\br_\infty(k) : -S'_{i+1} \leq k \leq -S'_{i} \}$. With the same
  argument as above, we see that it has the law of an (independent)
  $\mu_{1/2}$-Galton Watson tree. This concludes the proof. 
\end{proof}

\begin{figure}[h!]
	\begin{center}
	\includegraphics[scale=.85]{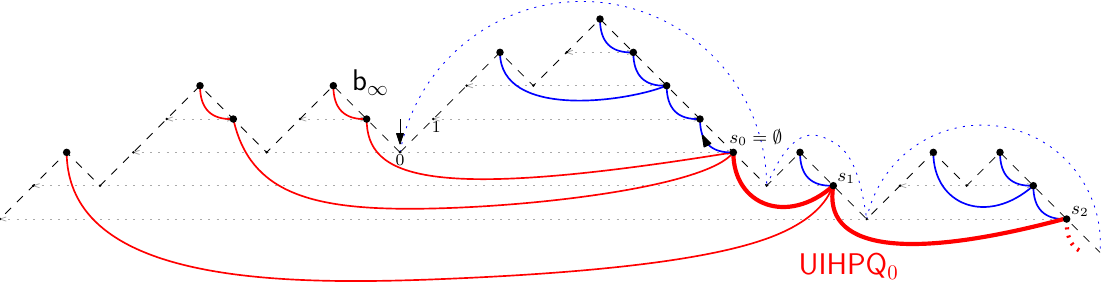}
	\end{center}
	\caption{The construction of the $\UIHPQ_0$ from a uniform infinite
          bridge $\br_\infty$. The spine is shown in bold red arcs. The
          trees on the left of the spine are drawn in blue and enclosed by
          dotted blue half-circles, which indicate the corresponding
          excursions of $\br_\infty$ encoding these trees. The trees on the
          right of the spine are drawn in red, as the spine itself.}
	\label{fig:kestens-tree}
	\end{figure}

\subsection{The $\normalfont{\UIHPQ_p}$ as a local limit of Boltzmann quadrangulations}
This section is devoted to the proof of Proposition
\ref{prop:Boltzmann}. It is convenient to first prove the analogous result
for pointed maps. For that purpose, we first extend the definitions of
Boltzmann measures from Section~\ref{sec:BoltzmannQuadrangulations} to
pointed maps and then use a ``de-pointing'' argument. We use the notation
$\cQ_f^{\bullet}$ for the set of finite rooted pointed quadrangulations,
and we write $\cQ^{\bullet,\sigma}_f$ for the set of finite pointed rooted
quadrangulations with $2\sigma$ boundary edges.  The corresponding
partition functions read
\[F^{\bullet}(g,z)=\sum_{\q \in
  \cQ_f^{\bullet}}g^{\#\Fq}z^{\#\Bq/2},\quad
F^{\bullet}_{\sigma}(g)=\sum_{\q \in \cQ^{\bullet,\sigma}_f}g^{\#\Fq},\]
and the associated pointed Boltzmann distributions are defined by

\[\Pgz^{\bullet}(\q)=\frac{g^{\#\Fq}z^{\#\Bq/2}}{F^{\bullet}(g,z)}, \quad
\q \in \cQ_f^{\bullet},\quad
\P_g^{\bullet,\sigma}(\q)=\frac{g^{\#\Fq}}{F^{\bullet}_\sigma(g)},\quad \q
\in \cQ^{\bullet,\sigma}_f.\]
We will need the following enumeration result for pointed rooted
maps. From~\cite[(23)]{Bu} and~\cite[Section 3.3]{BoGu}, we have for every
$0\leq p \leq 1/2$
\begin{equation}\label{eq:EnumFBullet}
  F_\sigma^{\bullet}(g_p)=\binom{2\sigma}{\sigma}\left(\frac{1}{1-p}\right)^\sigma,
  \quad \sigma\in\N_0.
\end{equation} 
Note that the result (3.29) in~\cite{BoGu} cannot be used directly, due to
a difference in the rooting convention (there, the root vertex has to be
chosen among the vertices of the boundary that are closest to the marked
point).

Recall that $g_p=p(1-p)/3$ for $0\leq p\leq 1/2$. The first step towards
the proof of Proposition~\ref{prop:Boltzmann} is the following convergence
result for pointed Boltzmann quadrangulations.

\begin{prop}\label{prop:PointedBoltzmann}
Let $0\leq p\leq 1/2$. For every $\sigma\in \N_0$, let
$Q^{\bullet}_\sigma(p)$ be a random rooted pointed quadrangulation
distributed according to $\P^{\bullet,\sigma}_{g_p}$. Then, we have the
local convergence for the metric $\dmap$ as $\sigma \to 
  \infty$
\[
Q_{\sigma}^\bullet(p) \overset{(d)}{\longrightarrow} \UIHPQ_p,
\]  

\end{prop}
\begin{proof} Let $\q\in \cQ^\sigma_f$, and
  $((\f,\la),\br)\in\cup_{n\geq 0}{\Fo_\sigma^n} \times \Br_\sigma$ such that
  $\q=\Phi((\f,\la),\br)$. Moreover, let
  $(\f_\sigma^{(p)},\la_\sigma^{(p)})$ be a uniformly labeled $p$-forest
  with $\sigma$ trees, i.e., a collection of $\sigma$ independent uniformly
  labeled $p$-Galton-Watson trees, and let $\br_\sigma$ be uniformly
  distributed in $\Br_\sigma$ and independent of
  $(\f_\sigma^{(p)},\la_\sigma^{(p)})$. We have
\begin{align*}\P\left(\Phi
  \big((\f_\sigma^{(p)},\la_\sigma^{(p)}),\br_\sigma\big) = \q  \right)
  &=\P\left(
    \big((\f_\sigma^{(p)},\la_\sigma^{(p)}),\br_\sigma\big) =
    \left(\left(\f,\la\right),\br \right)  \right)\\ 
  &=\left(\frac{p(1-p)}{3}\right)^{\vert \f
    \vert}\frac{(1-p)^\sigma}{\binom{2\sigma}{\sigma}}=\frac{g_p^{\#\Fq}}{F_\sigma^{\bullet}(g_p)},  
\end{align*}
Here, for the first equality in the second line, we have
used~\eqref{eq:lawGW2}, the fact that the label differences are
i.i.d. uniform in $\{-1,0,1\}$, and $|\Br_\sigma|=\binom{2\sigma}{\sigma}$.
The last equality follows from the enumeration result
\eqref{eq:EnumFBullet} and the fact that the number of edges of $\f$ equals
the number of faces of $\q$. Thus, $Q^\bullet_\sigma(p)$ is distributed as
$\Phi((\f_\sigma^{(p)},\la_\sigma^{(p)}),\br_\sigma)$.

Now observe that $\f_\sigma^{(p)}$ is already a collection of $\sigma$
independent $p$-Galton-Watson trees, and Lemma~\ref{lem:bridge} allows us
to couple the first and last $o(\sigma)$ steps of $\br_\sigma$ with the same number of steps of a uniform infinite bridge $\br_\infty$ around the origin. With exactly the same reasoning
as in Proposition~\ref{prop:couplingQnUIHPQ}, we therefore obtain with high
probability an isometry of balls
$B_{\delta\sqrt{\sigma}}(Q^\bullet_\sigma(p))$ and
$B_{\delta\sqrt{\sigma}}(\UIHPQ_p)$ for all $\sigma$ sufficiently large,
provided $\delta$ is small enough. The stated local convergence follows.\end{proof}

Proposition~\ref{prop:Boltzmann} is a consequence of the foregoing result
and the following de-pointing argument inspired by~\cite[Proposition
14]{Ab}. According to Remark~\ref{rem:BoltzmannCvgce}, it suffices to
consider the case $p\in[0,1/2)$.

In the following, by a small abuse of notation, we interpret
$\P^{\bullet,\sigma}_{g_p}$ as a probability measure on $\cQ_f$ by simply
forgetting the marked point.

\begin{lemma}\label{lem:Depointing}Let $0\leq p < 1/2$. Then,
  \[\lim_{\sigma\rightarrow\infty}\left\|
    \P^{\sigma}_{g_p}-\P^{\bullet,\sigma}_{g_p} \right\|_{\textup{TV}}=0.\]
\end{lemma}

\begin{proof}Let $\#V$ be the mapping 
  $\q \mapsto \# V(\q)$, which assigns to a finite quadrangulation $\q$ its
  number of vertices. We have the absolute continuity
  relation~\cite[(5)]{BeMi}
  \[\mathrm{d}\P^{\sigma}_{g_p}(\q)=\frac{K_\sigma}{\#V(\q)}\mathrm{d}\P^{\bullet,\sigma}_{g_p}(\q),\]
  where $K_\sigma=(\E^{\bullet,\sigma}_{g_p}[1/\#V)])^{-1}$. Then,
  \begin{equation}\label{eq:Depointing}
    \left\| \P^{\sigma}_{g_p}-\P^{\bullet,\sigma}_{g_p}
    \right\|_{\textup{TV}}=\frac{1}{2}\sup_{F:\cQ^\sigma_f \rightarrow
      [-1,1]}\left\vert \E^{\sigma}_{g_p}[F]-\E^{\bullet,\sigma}_{g_p}[F]
    \right\vert \leq \E^{\bullet,\sigma}_{g_p}\left[ \left\vert
        1-\frac{K_\sigma}{\# V} \right\vert \right]. 
  \end{equation}
  Let $(\tr^{(p)}_0, \ldots, \tr^{(p)}_{\sigma-1})$ be a collection of
  independent $p$-Galton-Watson trees. The proof of
  Proposition~\ref{prop:PointedBoltzmann} shows that under
  $\P^{\bullet,\sigma}_{g_p}$, $\#V$ has the same law as
  \[1+\sum_{i=0}^{\sigma-1} \# V(\tr^{(p)}_i).\]
  Note that the summand $+1$ accounts for the pointed vertex, which is
  added to the tree vertices in the Bouttier-Di Francesco-Guitter
  mapping. Using the fact that $\# V(\tr^{(p)}_0)$ has the same
  law as $(T_{-1}^{(p)}+1)/2$, where $T_{-1}^{(p)}$ is the first hitting
  time of $-1$ of a random walk with step distribution
  $p\delta_1+(1-p)\delta_{-1}$, an application of the optional stopping
  theorem gives
  \[\E^{\bullet,\sigma}_{g_p}[\#V]=1+\sigma\E\left[\#V(\tr^{(p)}_0)
  \right]=1+\sigma\left(\frac{1-p}{1-2p}\right). \]
  Moreover, using $p<1/2$ and the description in terms of $T_{-1}^{(p)}$, it is readily checked that the random variable
  $\#V(\tr^{(p)}_0) $ has small exponential moments. Cram\'er's theorem thus
  ensures that for every $\delta>0$, there exists a constant $C_\delta>0$
  such that
  \[\P^{\bullet,\sigma}_{g_p}\left( \left\vert \#V
      -\E^{\bullet,\sigma}_{g_p}[\#V]\right\vert > \delta\sigma \right)\leq
  \exp(-C_\delta \sigma).\]
  We now proceed similarly to~\cite[Lemma 16]{Ab}. Let $X_\sigma$ be
  distributed as $\# V /\E^{\bullet,\sigma}_{g_p}[\#V]$ under
  $\P^{\bullet,\sigma}_{g_p}$. Note that
  $X_\sigma^{-1}\leq \E^{\bullet,\sigma}_{g_p}[\#V]$
  $\P^{\bullet,\sigma}_{g_p}$-a.s. since $\#V \geq 1$. Moreover, it is seen
  that
  $\{\vert X_\sigma^{-1} -1 \vert > \delta \}\subset \{\vert X_\sigma \vert
  < 1/2 \}\cup \{\vert X_\sigma -1 \vert > \delta/2 \}$.
  From these observations, we obtain
\begin{align*}
  \E\left[\left\vert X_\sigma^{-1} - 1 \right\vert\right] &\leq \delta +
  \E\left[\left\vert X_\sigma^{-1} - 1 \right\vert \mathbf{1}_{\{\vert X_\sigma^{-1}-1 \vert >\delta \}} \right]\\ 
  &\leq \delta + \left(\E^{\bullet,\sigma}_{g_p}[\#V]+1\right)\P\left(\vert X_\sigma-1 \vert >\frac{\delta}{2} \wedge \frac{1}{2} \right).
\end{align*} 
The preceding two displays show that the expected number of vertices grows
linearly in $\sigma$, and the probability on the right decays
exponentially fast in $\sigma$. Since $\delta>0$ was arbitrary, we deduce
that $X_\sigma^{-1} \longrightarrow 1$ as $\sigma \rightarrow \infty$ in
$\mathbb{L}^1$. Finally,
\[\E^{\bullet,\sigma}_{g_p}\left[ \left\vert 1-\frac{K_\sigma}{\# V} \right\vert \right]=\E\left[\left\vert
    1-\frac{X_\sigma^{-1}}{\E[X_\sigma^{-1}]} \right\vert \right] \leq
\frac{1}{\E[X_\sigma^{-1}]}\left( \left\vert \E[X_\sigma^{-1}]-1
  \right\vert + \E[\vert X_\sigma^{-1} -1\vert] \right)
\longrightarrow 0 \]
as $\sigma \rightarrow \infty$, which concludes the proof
by~\eqref{eq:Depointing}.
\end{proof}

\subsection{The $\normalfont{\BHP}_\theta$ as a local scaling limit of
  the $\normalfont{\UIHPQ}_p$'s }
\label{sec:UIHPQpBHPtheta}
In this section, we prove Theorem~\ref{thm:GHconv-BHPtheta}. For the
reminder, we fix a sequence $(a_n,n\in\N)$ of positive reals tending to
infinity and let $r>0$ be given. Similarly to~\cite[Proof of Theorem
3.4]{BaMiRa}, the main step is to establish an absolute
continuity relation of balls around the roots of radius $ra_n$ between the
$\UIHPQ_p$ for $p\in(0,1/2]$ and the $\UIHPQ=\UIHPQ_{1/2}$.  To this aim,
we compute the Radon-Nikodym derivative of the encoding contour function of
the $\UIHPQ_p$ with respect to that of the $\UIHPQ$ on an interval of the
form $[-sa_n^2,sa_n^2]$ for $s>0$. From Theorem 3.8 of~\cite{BaMiRa} we
know that $a_n^{-1}\cdot \UIHPQ\rightarrow \BHP_0$ in distribution in the
local Gromov-Hausdorff topology, jointly with a uniform convergence on
compacts of (rescaled) contour and label functions. An application of
Girsanov's theorem shows that the limiting Radon-Nikodym derivative turns
the contour function of $\BHP_0$ into the contour function of
$\BHP_\theta$, which allows us to conclude.

In order to make these steps rigorous, we begin with some notation specific
to this section. Let $f\in\mathcal{C}(\R,\R)$ and $x\in\mathbb{R}$. We
define the last (first) visit to $x$ to the left (right) of $0$,
\[
U_x(f)=\inf\{t\leq 0: f(t)=x\}\in [-\infty,0],\quad T_x(f)=\inf\{t\geq 0:
f(t)=x\}\in [0,\infty].
\]
We agree that $U_x(f)=-\infty$ if the set over which the infimum is
taken is empty, and, similarly, $T_x(f)=\infty$ if the second set is
empty. We will also apply $U_x$ to functions in
$\mathcal{C}((-\infty,0],\R)$, and $T_x$ to functions in
$\mathcal{C}([0,\infty),\R)$.

If $f\in\mathcal{C}(\R,\R)$ is the contour function of
an infinite $p$-forest for some $p\in(0,1/2]$ (or part of it defined on
some interval), and if $x\in\N$, we use
the notation
\[
\bv(f,x) = \frac{1}{2}\left(T_{-x}(f)-U_{x}(f) -2x\right)
\]
for the total number of edges of the $2x$ trees encoded by $f$ along the
interval $[U_{x}(f),T_{-x}(f)]$. We set $\bv(f,x)=\infty$ if
$U_{x}(f)$ or $T_{-x}(f)$ is unbounded.

Given $s>0$, we put for $n\in\N$
\[s_n=\lfloor (3/2)sa_n^2\rfloor.\]
Now let $p\in(0,1/2]$. Throughout this section and as usual, we assume that
$((\f^{(p)}_{\infty},\la^{(p)}_\infty),\br_\infty)$ and
$((\f_{\infty},\la_\infty),\br_\infty)$ encode the $\UIHPQ_p$
$Q_\infty^{\infty}(p)$ and the standard $\UIHPQ$ $Q_\infty^{\infty}$,
respectively (see Definition~\ref{def:UIHPQp}). We stress that since the
skewness parameter $p$ does not affect the law of the infinite bridge
$\br_\infty$, we can and will use the same bridge in the construction of
both $Q_\infty^{\infty}(p)$ and $Q_\infty^{\infty}$. We denote by
$(C_\infty^{(p)},\La_\infty^{(p)})$ and $(C_\infty,\La_\infty)$ the
associated contour and label functions, viewed as elements in
$\mathcal{C}(\R,\R)$.

For understanding how the balls of radius $ra_n$ for some $r>0$ around the
roots in $Q_\infty^{\infty}(p)$ and $Q_\infty^{\infty}$ are related to each
other, we need to control the contour functions $C_{\infty}^{(p)}$ and
$C_\infty$ on $[U_{s_n},T_{-s_n}]$ for a suitable choice
of $s=s(r)$. In this regard, we first formulate an absolute continuity
relation between the probability laws $\P_{s,n}^{(p)}$ and $\P_{s,n}$ on
$\mathcal{C}(\R,\R)$ defined as follows:
\begin{align*}
\P_{n,s}^{(p)}&=\textup{Law}\left(\left(C_\infty^{(p)}(t\vee U_{s_n}(C_\infty^{(p)})\wedge
      T_{-s_n}(C_\infty^{(p)})),t\in\R\right)\right),\\
\P_{n,s}&=\textup{Law}\left(\left(C_\infty(t\vee
    U_{s_n}(C_\infty)\wedge T_{-s_n}(C_\infty)),t\in\R\right)\right).
\end{align*}

\begin{lemma}
\label{lem:rw-rwp-abscont}
Let $p\in(0,1/2]$ and $s>0$. The laws $\P_{n,s}^{(p)}$ and $\P_{n,s}$ are
absolutely continuous with respect to each 
other: For any 
$f\in\textup{supp}(\P_{n,s}^{(p)}) (=\textup{supp}(\P_{n,s}))$, with
$s_n$ as above,
\[
\P_{n,s}^{(p)}(f)=\left(4p(1-p)\right)^{\bv(f,s_n)}\left(2(1-p)\right)^{2s_n}\P_{n,s}(f).
\]
\end{lemma}

\begin{proof}
  By definition of $C_\infty^{(p)}$ and $C_\infty$, each element
  $f\in \mathcal{C}(\R,\R)$ in the support of $\P_{n,s}^{(p)}$ lies also in
  the support of $\P_{n,s}$ and vice versa (note that $p\notin\{0,1\}$).

  More specifically, for such an $f$ supported by these laws,
  $\P_{n,s}^{(p)}(f)$ resp. $\P_{n,s}(f)$ is the probability of a particular
  realization of $2s_n$ independent $p$-Galton-Watson trees resp.
  $(1/2)$-Galton-Watson trees with $\bv(f,s_n)$ tree edges in
  total. Therefore, by~\eqref{eq:lawGW2},
\[
\P^{(p)}_{n,s}(f)=p^{\bv(f,s_n)}(1-p)^{\bv(f,s_n)}(1-p)^{2s_n},\quad
\textup{and}\quad\P_{n,s}(f)=2^{-2(\bv(f,s_n) + s_n)}.
\]
This proves the lemma. 
\end{proof}

We turn to the proof of Theorem~\ref{thm:GHconv-BHPtheta}. To that aim, we
will work with rescaled and stopped versions of $(C_\infty^{(p)},\La_\infty^{(p)})$ and
$(C_\infty,\La_\infty)$, which encode the information of the first 
$s_n=\lfloor (3/2)sa_n^2\rfloor$ trees to the right of zero, and of the first
$s_n$ trees to the left of zero. Specifically, we let
\begin{align*}
  C^{\infty,p}_{n,s}=\left(C^{\infty,p}_{n,s}(t),t\in\R\right)&=\left(\frac{1}{(3/2)a_n^2}C_{\infty}^{(p)}\left((9/4)a_n^4t\vee
                                                                U_{s_n}(C_{\infty}^{(p)})\wedge T_{-s_n}(C_{\infty}^{(p)})\right),t\in \R\right),\\
  \La_{n,s}^{\infty,p}=\left(\La_{n,s}^{\infty,p}(t),t\in\R\right)&=\left(\frac{1}{a_n}\La_{\infty}^{(p)}\left((9/4)a_n^4t\vee
                                                                    U_{s_n}(C_{\infty}^{(p)})\wedge T_{-s_n}(C_{\infty}^{(p)})\right),t\in\R\right),\\
  C^{\infty}_{n,s}=\left(C^{\infty}_{n,s}(t),t\in\R\right)&=\left(\frac{1}{(3/2)a_n^2}C_{\infty}\left((9/4)a_n^4t\vee
                                                            U_{s_n}(C_{\infty})\wedge T_{-s_n}(C_{\infty})\right),t\in \R\right),\\
  \La_{n,s}^{\infty}=\left(\La_{n,s}^{\infty}(t),t\in\R\right)&=\left(\frac{1}{a_n}\La_{\infty}\left((9/4)a_n^4t\vee
                                                                U_{s_n}(C_{\infty})\wedge T_{-s_n}(C_{\infty})\right),t\in \R\right).
\end{align*}
Following our notation from Section~\ref{sec:def-BHPtheta}, we denote by $\Xha=(\Xha(t),t\in\R)$ and $\Wha=(\Wha(t),t\in\R)$
the  contour and label functions of the limit space
$\BHP_\theta$. We also put
\begin{align*}
  X^{\theta,s}=\left(X^{\theta,s}(t),t\in\R\right)&=\left(\Xha\left(t\vee
                                                    U_{s}(\Xha)\wedge T_{-s}(\Xha)\right),t\in \R\right),\\
  W^{\theta,s}=\left(W^{\theta,s}(t),t\in\R\right)&=\left(\Wha\left(t\vee U_{s}(\Xha)\wedge
                                                    T_{-s}(\Xha)\right),t\in \R\right).
\end{align*}
Accordingly, we write $X^0, W^0$ and $X^{0,s},W^{0,s}$ for
the corresponding functions associated to $\BHP_0$. We will make use of the
following joint convergence.

\begin{lemma}
\label{lem:convjoint}
Let $r,s> 0$. Then, in the notation from above, we have the joint
convergence in law in $\mathcal{C}(\R,\R)\times\mathcal{C}(\R,\R)\times\mathbb{K}$,
\[
\left(C^\infty_{n,s}, \La_{n,s}^\infty,B_r^{(0)}\left(a_n^{-1}\cdot
    Q_\infty^\infty\right)\right)\xrightarrow[]{(d)} \left(X^{0,s},W^{0,s},B_r(\BHP_0)\right).
\]
Moreover, for $n\rightarrow\infty$
\[
\frac{\bv(C_\infty,s_n)}{(9/4)a_n^4}\xrightarrow[]{(d)}\frac{1}{2}\left(T_{-s}-U_s\right)(X^{0}).
\]
\end{lemma}
\begin{proof}
  Both statements are proved in~\cite{BaMiRa}; to give a quick reminder,
  first note by standard random walk estimates that for each $\delta>0$,
  there exists a constant $c_\delta>0$ such that
  $\P(\bv(C_\infty,s_n)>c_\delta a_n^4)\leq \delta$; see~\cite[Proof of Lemma
  6.18]{BaMiRa} for details. Together with the joint convergence in law in
  $\mathcal{C}(\R,\R)^2\times\mathbb{K}$ obtained in~\cite[(6.30) of Remark
  6.17]{BaMiRa}, which reads 
\[
  \left(\frac{C_\infty((9/4)a_n^4\cdot)}{(3/2)a_n^2}, \frac{\La_\infty((9/4)a_n^4\cdot)}{a_n},B_r^{(0)}\left(a_n^{-1}\cdot
    Q_\infty^\infty\right)\right)\xrightarrow[]{(d)} \left(X^0,W^0,B_r(\BHP_0)\right),
\]
the first claim of the statement follows, and the second is then a
consequence of this.
\end{proof}

\begin{proof}[Proof of Theorem~\ref{thm:GHconv-BHPtheta}]
  We fix a sequence $(p_n,n\in\N)\subset (0,1/2]$ of the form
\[
p_n=\frac{1}{2}\left(1-\frac{2\theta}{3a_n^2}\right) +o\left(a_n^{-2}\right).
\]
By Remark~\ref{rem:balls} and the observations in Section~\ref{sec:GH}, the claim
follows if we show that for all $r>0$, as $n\rightarrow\infty$,
\[
B_r^{(0)}\left(a_n^{-1}\cdot
  Q_\infty^{\infty}(p_n)\right)\xrightarrow{(d)}B_r(\BHP_{\theta})\]
in distribution in $\mathbb{K}$. At this point, recall that
$B_r^{(0)}(a_n^{-1}\cdot Q_\infty^{\infty}(p_n))=a_n^{-1}\cdot
B_{ra_n}^{(0)}(Q_\infty^{\infty}(p_n))$
is the (rescaled) ball of radius $ra_n$ around the vertex
$\f_\infty^{(p_n)}(0)$ in $Q_\infty^{\infty}(p_n)$.  We consider the event
$$\mathcal{E}^1(n,s)=\left\{\min_{[0,\,s_n]}\br_\infty
   <-\,3ra_n,\,
   \min_{[-s_n,\,0]}\br_\infty<-3ra_n\right\}.
$$
We define a similar event in terms of the two-sided Brownian motion
$\gamma=(\gamma(t),t\in\R)$ scaled by the factor $\sqrt{3}$, which forms
part of the construction of the space $\BHP_\theta$ given in
Section~\ref{sec:def-BHPtheta},
$$\mathcal{E}^2(s)=\left\{\min_{[0,s]}\gamma <-3r,\,
  \min_{[-s,0]}\gamma<-3r\right\}.
$$
Using the cactus bound, it was argued in~\cite[Proof of Theorem
3.4]{BaMiRa} that on the event $\mathcal{E}^1(n,s)$, for any $p\in(0,1/2]$,
the ball $B_{ra_n}^{(0)}(Q_\infty^{\infty}(p))$ viewed as a submap of
$Q_\infty^{\infty}(p)$ is a measurable function of
$(C^{\infty,p}_{n,s},\La_{n,s}^{\infty,p})$. (In~\cite{BaMiRa}, only the
case $p=1/2$ was considered, but the argument remains exactly the same for
all $p$, since the encoding bridge $\br_\infty$ does not depend
on the choice of $p$.) Similarly, on $\mathcal{E}^2(s)$, the ball
$B_r(\BHP_{\theta})$ for any $\theta\geq 0$ is a measurable function of
$(X^{\theta,s},W^{\theta,s})$.

Now let $\eps>0$ be given. By the (functional) central limit theorem, we find that for
$s>0$ and $n_0\in\N$ sufficiently large, it holds that for all $n\geq n_0$,
$\P(\mathcal{E}^1(n,s))\geq 1-\eps$. By choosing $s$ possibly larger, we
can moreover ensure that $\P(\mathcal{E}^2(s))\geq 1-\eps$. We fix such
$s>0$ and $n_0\in\N$ such that for all $n\geq n_0$, both events
$\mathcal{E}^1(n,s)$ and $\mathcal{E}^2(s)$ have probability at least
$1-\eps$.

Next, consider the laws $\P_{n,s}^{(p_n)}$ and $\P_{n,s}$ defined
just above Lemma~\ref{lem:rw-rwp-abscont}, and put for $f\in\mathcal{C}(\R,\R)$
\begin{equation}
\label{eq:proof-thm2-lambda_n}
\lambda_{n,s}(f)=\left(4p_n(1-p_n)\right)^{\bv(f,s_n)}\left(2(1-p_n)\right)^{2s_n}.
\end{equation}
Then, with $F:
\mathcal{C}(\R,\R)^2\times\mathbb{K}\rightarrow\R$ measurable and bounded,
Lemma~\ref{lem:rw-rwp-abscont} shows
\begin{multline}
\label{eq:proof-thm2-eq1}
\E\left[F\left(C_{n,s}^{\infty,p_n},\La_{n,s}^{\infty,p_n},B_r^{(0)}\left(a_n^{-1}\cdot
      Q_\infty^\infty(p_n)\right)\right)\mathbf{1}_{\mathcal{E}^1(n,s)}\right]\\
=\E\left[\lambda_{n,s}(C_\infty)F\left(C^\infty_{n,s},
    \La_{n,s}^\infty,B_r^{(0)}\left(a_n^{-1}\cdot
      Q_\infty^\infty\right)\right)\mathbf{1}_{\mathcal{E}^1(n,s)}\right].
\end{multline}
Note that on the left side, we consider the closed ball of radius $ra_n$
around the vertex $\f_\infty(0)$ in the $\UIHPQ_{p_n}$
$Q_\infty^\infty(p_n)$, whereas on the right side, we look at the
corresponding ball in the standard $\UIHPQ$ $Q_\infty^\infty$ with contour
and label functions $C_\infty$ and $\La_\infty$. Plugging in
the value of $p_n$ in~\eqref{eq:proof-thm2-lambda_n}, we get
\begin{equation}
\label{eq:proof-thm2-eq4}
\lambda_{n,s}(f)=\left(1+\frac{2\theta}{3a_n^2}+o(a_n^{-2})\right)^{2s_n}\left(1-\frac{4\theta^2}{9a_n^4}+o(a_n^{-4})\right)^{\bv(f,s_n)}.
\end{equation}
Applying both statements of Lemma~\ref{lem:convjoint}, and using~\eqref{eq:proof-thm2-eq4},
it follows that for large $n\geq n_1(\varepsilon)$
\begin{multline}
\label{eq:proof-thm2-eq5}
\Big|\E\left[\lambda_{n,s}(C_\infty)F\left(C^\infty_{n,s},
    \La_{n,s}^\infty,B_r^{(0)}\left(a_n^{-1}\cdot
      Q_\infty^\infty\right)\right)\right] -\\
\E\left[\exp\left(2s\theta-(T_{-s}-U_s)(X^0)\theta^2/2\right)F\left(X^{0,s},W^{0,s},B_r\left(\BHP_0\right)\right)\right]\Big|\leq
\eps.
\end{multline}
The rest of the proof is now similar to~\cite[Proof of Theorem
3.4]{BaMiRa}. Applying Pitman's transform and
Girsanov's theorem, we have for continuous and bounded
$G:\mathcal{C}(\R,\R)^2\rightarrow\R$ 
\[
\E\left[\exp\left(2s\theta-(T_{-s}-U_s)(X^0)\theta^2/2\right)G\left(X^{0,s},W^{0,s}\right)\right]=\E\left[G\left(X^{\theta,s},W^{\theta,s}\right)\right].
\]
On $\mathcal{E}^2(s)$,
$B_r(\BHP_0)$ is a measurable function of $(X^{0,s},W^{0,s})$, and
$B_r(\BHP_\theta)$ is given by the {\it same} measurable function of
$(X^{\theta,s},W^{\theta,s})$. Consequently, 
\begin{multline}
\label{eq:proof-thm2-eq6}
\E\left[\exp\left(2s\theta-(T_{-s}-U_s)(X^0)\theta^2/2\right)F\left(X^{0,s},W^{0,s},B_r\left(\BHP_0\right)\right)\mathbf{1}_{\mathcal{E}^2(s)}\right]\\
=\E\left[F\left(X^{\theta,s},W^{\theta,s},B_r\left(\BHP_\theta\right)\right)\mathbf{1}_{\mathcal{E}^2(s)}\right].
\end{multline}
Recall that the events $\mathcal{E}^1(n,s)$ and
$\mathcal{E}^2(s)$ have probability at least $1-\eps$.
Using this fact together
with~\eqref{eq:proof-thm2-eq1},~\eqref{eq:proof-thm2-eq5},~\eqref{eq:proof-thm2-eq6}
and the triangle inequality, we find a constant $C=C(F,s,\theta)$
such that for sufficiently large $n$,
\[
\left|\E\left[F\left(C_{n,s}^{\infty,p_n},\La_{n,s}^{\infty,p_n},B_r^{(0)}\left(a_n^{-1}\cdot
      Q_\infty^{\infty}(p_n)\right)\right)\right] -
  \E\left[F\left(X^{\theta,s},W^{\theta,s},B_r\left(\BHP_\theta\right)\right)\right]\right|\leq
  C\eps.\]
This implies the theorem.
\end{proof}

\subsection{The $\normalfont{\ICRT}$ as a local scaling limit of the
  $\normalfont{\UIHPQ}_p$'s }
Theorem~\ref{thm:GHconv-ICRT} states that the $\ICRT$ appears as the
distributional limit of $a_n^{-1}\cdot\UIHPQ_{p_n}$ when
$a_n\rightarrow\infty$ and $p_n\in[0,1/2]$ satisfies
$a_n^2\left(1-2p_n\right)\rightarrow\infty$ as $n\rightarrow\infty$.  In
essence, the idea behind the proof is the following. Fix $r>0$, and
sequences $(a_n)_n$ and $(p_n)_n$ with the above properties. It turns out
that in the $\UIHPQ_{p_n}$, vertices at a distance less than $ra_n$ from
the root are to be found at a distance of order $o(a_n)$ from the
boundary. Therefore, upon rescaling the graph distance by a factor
$a_n^{-1}$, the scaling limit of the $\UIHPQ_{p_n}$ in the local
Gromov-Hausdorff sense will agree with the scaling limit of its
boundary. Upon a rescaling by $a_n^2$ in time and
$a_n^{-1}$ in space, the encoding bridge $\br_\infty$ converges to a
two-sided Brownian motion, which in turn encodes the $\ICRT$.

The above observations are most naturally turned into a proof using the
description of the Gromov-Hausdorff metric in terms of correspondences
between metric spaces; see~\cite[Theorem
7.3.25]{BuBuIv}. Lemma~\ref{lem:localGHconv} below captures the kind of
correspondence we need to construct. Our strategy of showing convergence of
quadrangulations with a boundary towards a tree has already been
successfully implemented before; see, for instance,~\cite[Proof of Theorem
5]{Be}.

For the remainder of this section, we write
$((\f_\infty^{(n)},\la_\infty^{(n)}),\br_\infty)$ for a uniformly labeled
infinite $p_n$-forest together with an (independent) uniform infinite
bridge $\br_\infty$, and we assume that the $\UIHPQ_{p_n}$ is given in
terms of $((\f_\infty^{(n)},\la_\infty^{(n)}),\br_\infty)$, {\it via} the
Bouttier-Di Francesco-Guitter mapping. We interpret the associated contour
function $C_\infty^{(n)}$, the bridge $\br_\infty$ and the (unshifted)
labels $\la_\infty^{(n)}$ as elements in $\mathcal{C}(\R,\R)$ (by linear
interpolation); see Section~\ref{sec:contour-label-fct}.

The core of the argument lies in the following lemma, which gives the
necessary control over distances to the boundary, {\it via} a control of
the labels $\la_\infty^{(n)}$. We will use it at the very end of the
proof of Theorem~\ref{thm:GHconv-ICRT}, which follows afterwards. 
\begin{lemma}
\label{lem:conv-contlabelbridge}
Let $(a_n,n\in\N)$ be a sequence of positive reals tending to infinity, and
$(p_n,n\in\N)\subset[0,1/2)$ be a sequence satisfying
$a_n^2(1-2p_n)\rightarrow\infty$ as $n\rightarrow\infty$. Then,
in the notation from above, we have the distributional convergence in
$\mathcal{C}(\R,\R^2)$ as $n\rightarrow\infty$,
\[
\left(\Big(\frac{1}{a_n^2}C_\infty^{(n)}\Big(\frac{a_n^2}{1-2p_n}s\Big),\frac{1}{a_n}\la_\infty^{(n)}\Big(\frac{a_n^2}{1-2p_n}s\Big)\Big),s\in\R\right)\xrightarrow[]{(d)}\left((-s,0),s\in\R\right).
\] 
\end{lemma}

\begin{proof}
  We have to show joint convergence of $C_\infty^{(n)}$
  and $\la_\infty^{(n)}$ on any interval of the form $[-K,K]$, for $K>0$.
  Due to an obvious symmetry in the definition of the contour function, we
  may restrict ourselves to intervals of the form $[0,K]$. Fix $K>0$, and
  put $\theta_n=(1-2p_n)^{-1}a_n^2$. We first show that
  $a_n^{-2}C_\infty^{(n)}(\theta_n s)$, $s\in\R$, converges on $[0,K]$ to
  $g(s)=-s$ in probability. For that purpose, recall that $C_\infty^{(n)}$ on
  $[0,\infty)$ has the law of an linearly interpolated random walk started
  from $0$ with step distribution $p_n\delta_1 +(1-p_n)\delta_{-1}$. Set
  $K_n=\lceil K\theta_n\rceil$, and let $\delta >0$. By using Doob's
  inequality in the second line,
\begin{align}
\label{eq:conv-cont-1}
  \lefteqn{\P\Big(\sup_{s\in[0,K]}\left|a_n^{-2}C_\infty^{(n)}(\theta_ns)+s\right|>\delta\Big)\leq
  \P\Big(\sup_{0\leq i\leq K_n}\left|C_\infty^{(n)}(i)+(1-2p_n)i\right|>\delta a_n^2\Big)}\nonumber\\
&\leq
  \frac{1}{\delta^2a_n^4}\E\left[\left|C_\infty^{(n)}(K_n)+(1-2p_n)K_n\right|^2\right]\leq
  \frac{4K_n}{\delta^2a_n^4}\leq \frac{4K}{\delta^2a_n^2(1-2p_n)}. 
\end{align}
Thanks to our assumption on $p_n$, the right-hand side converges to zero,
and the convergence of the contour function is established.  Showing joint
convergence together with the (rescaled) labels $\la_\infty^{(n)}$ is now
rather standard: First, we may assume by Skorokhod's theorem that
$a_n^{-2}C_\infty^{(n)}(\theta_ns)$ converges on $[0,K]$ almost surely. Now
fix $0\leq s\leq K$. Conditionally given $C_\infty^{(n)}$ on
$[0,K\theta_n]$, we have by construction, for $(\eta_i,i\in\N)$ a sequence
of i.i.d. uniform random variables on $\{-1,0,1\}$, and with
$\underline{C}_\infty^{(n)}(\lfloor \theta_ns\rfloor)=\min_{[0,\lfloor
  \theta_n s\rfloor]}C_\infty^{(n)}$,
\begin{equation}
\label{eq:conv-cont-2}
\la_\infty^{(n)}(\lfloor \theta_n s\rfloor)\eqd
\sum_{i=1}^{C_\infty^{(n)}(\lfloor \theta_n s\rfloor)-\underline{C}_\infty^{(n)}(\lfloor \theta_ns\rfloor)}\eta_i.
\end{equation}
Conditionally given $C_\infty^{(n)}$ on $[0,K\theta_n]$, for
$\delta>0$, Chebycheff's inequality gives
\[
\P\left(\la_\infty^{(n)}(\lfloor \theta_n s\rfloor)> \delta
  a_n\,|\,C_\infty^{(n)}\restriction
  [0,\theta_ns]\right)\leq\frac{1}{\delta^2a_n^2}\left(
  C_\infty^{(n)}(\lfloor \theta_n
  s\rfloor)-\underline{C}_\infty^{(n)}(\lfloor \theta_ns\rfloor)\right).
\]
By our assumption,
$a_n^{-2}( C_\infty^{(n)}(\lfloor \theta_n
s\rfloor)-\underline{C}_\infty^{(n)}(\lfloor \theta_ns\rfloor))$
converges to zero almost surely, and we conclude
\[
\left(a_n^{-2} C_\infty^{(n)}(\lfloor \theta_n s\rfloor),
  a_n^{-1}\la_\infty^{(n)}(\lfloor \theta_n s\rfloor)\right)
\xrightarrow[]{(d)}(-s,0)\quad\textup{ as }n\rightarrow\infty.
\]
Since both $C_\infty^{(n)}$ and $\la_\infty^{(n)}$ are Lipschitz almost
surely, the claim follows with $\lfloor \theta_n s\rfloor$ replaced by
$\theta_n s$.  Joint finite-dimensional convergence can now be shown
inductively: As for two-dimensional convergence on $[0,K]$, we simply note
that when $0\leq s_1<s_2\leq K$ are such that
$C_\infty^{(n)}(\lfloor\theta_n s_1\rfloor)$ and
$C_\infty^{(n)}(\lfloor\theta_n s_2\rfloor)$ encode vertices of different
trees of $\f_\infty^{(n)}$, then, conditionally on
$C_\infty^{(n)}\restriction[0,\lfloor\theta_n s_2\rfloor]$,
$\la_\infty^{(n)}(\lfloor \theta_n s_1\rfloor)$ and
$\la_\infty^{(n)}(\lfloor \theta_n s_2\rfloor)$ are independent sums of
i.i.d. uniform variables on $\{-1,0,1\}$, and we have a representation
similar to~\eqref{eq:conv-cont-2}. If
$C_\infty^{(n)}(\lfloor\theta_n s_1\rfloor)$ and
$C_\infty^{(n)}(\lfloor\theta_n s_2\rfloor)$ encode vertices of the
same tree of $\f_\infty^{(n)}$, then, with the abbreviation
\[
\check{C}_\infty^{(n)}(s_1,s_2)=\min_{[\lfloor \theta_n s_1\rfloor,\lfloor
  \theta_n s_2\rfloor]}C_\infty^{(n)}-\underline{C}_\infty^{(n)}(\lfloor
\theta_n s_1\rfloor),
\]
it holds that
\begin{align*}
  \la_\infty^{(n)}(\lfloor \theta_n s_1\rfloor)&\eqd
                                                 \sum_{i=1}^{\check{C}_\infty^{(n)}(s_1,s_2)}
                                                 \eta_i + \sum_{i=\check{C}_\infty^{(n)}(s_1,s_2)+1}^{C_\infty^{(n)}(\lfloor \theta_n s_1\rfloor)} \eta'_i,\\
  \la_\infty^{(n)}(\lfloor \theta_n s_2\rfloor)&\eqd
                                                 \sum_{i=1}^{\check{C}_\infty^{(n)}(s_1,s_2)}
                                                 \eta_i + \sum_{i=\check{C}_\infty^{(n)}(s_1,s_2)+1}^{C_\infty^{(n)}(\lfloor \theta_n s_2\rfloor)} \eta'_i,
\end{align*}
where $(\eta'_i,i\in\N)$ is an i.i.d. copy of $(\eta_i,i\in\N)$. Using
almost sure convergence of $a_n^{-2}C_\infty^{(n)}(\theta_ns)$ on $[0,K]$
and an argument similar to that in the one-dimensional convergence
considered above, we get two-dimensional convergence of
$(a_n^{-2} C_\infty^{(n)}(\theta_ns), a_n^{-1}\la_\infty^{(n)}(\theta_ns))$
on $[0,K]$, as wanted. Some more details can be found in~\cite[Proof of
Theorem 4.3]{LGMi}. Higher-dimensional convergence is now shown inductively
and is left to the reader. It remains to show tightness of the rescaled
labels. We begin with the following lemma.
\begin{lemma}
\label{lem:contour-hoelder}
Let $K>0$, $(a_n,n\in\N)$ and $(p_n,n\in\N)$ be as above. Then, for any
$q\geq 2$, there exists a constant $C_q>0$ such that for any $n\in\N$ and
any $0\leq s_1,s_2\leq K$, we have (with $\theta_n=(1-2p_n)^{-1}a_n^2$, as before)

\[
a_n^{-2q}\,\E\left[\left|C_\infty^{(n)}(\theta_n
    s_1)-C_\infty^{(n)}(\theta_n
    s_2)\right|^q\right]\leq C_q|s_1-s_2|^{q/2}.
\]
\end{lemma}
\begin{proof}
If $|s_1-s_2| \leq \theta_n^{-1}$, then, using linearity of
$C_\infty^{(n)}$,
\[
a_n^{-2q}\,\E\left[\left|C_\infty^{(n)}(\theta_n
    s_1)-C_\infty^{(n)}(\theta_n s_2)\right|^q\right]\leq
a_n^{-2q}\theta_n^{q}|s_1-s_2|^q\leq
a_n^{-2q}\theta_n^{q/2}|s_1-s_2|^{q/2}.
\]
Since $a_n^{-2q}\theta_n^{q/2}\leq a_n^{-q}(1-2p_n)^{-q/2}\rightarrow 0$ by
assumption on $p_n$, the claim of the lemma follows in this case. Now
let $|s_1-s_2| >\theta_n^{-1}$. We may assume $s_2\geq s_1$. Using the
triangle inequality and again the assumption on $p_n$, we see that it
suffices to establish the claim in the case where $\theta_ns_1$ and
$\theta_ns_2$ are integers. Recall that $(C_\infty^{(n)}(t), t \in \R)$ is a two-sided random walk with steps distributed according to $p_n \delta_1 + (1-p_n) \delta_{-1}=  p_n \delta_{2(1-p_n)-(1-2p_n)} + (1-p_n) \delta_{-2p_n-(1-2p_n)}$ (with linear interpolation). So we get 

\[C_\infty^{(n)}(\theta_n s_2)-C_\infty^{(n)}(\theta_n
s_1)\eqd\Big(\sum_{i=1}^{\theta_n(s_2-s_1)}\vartheta_i\Big)-\theta_n(s_2-s_1)(1-2p_n),\]
where $(\vartheta_i,i\in\N)$ are (centered) i.i.d. random variables with
distribution $p_n\delta_{2(1-p_n)} +(1-p_n)\delta_{-2p_n}$. Using that
$|a+b|^q \leq 2^{q-1}(|a|^q+|b|^q)$ for reals $a,b$, we get 
\[
\E\left[\left|C_\infty^{(n)}(\theta_n s_2)-C_\infty^{(n)}(\theta_n
s_1)\right|^q\right]\leq 2^{q-1}
\Big(\E\Big[\Big|\sum_{i=1}^{\theta_n(s_2-s_1)}\vartheta_i\Big|^q\Big] +
\theta_n^q(1-2p_n)^q(s_2-s_1)^q\Big).
\]
The second term within the parenthesis is equal to $a_n^{2q}|s_2-s_1|^q \leq
K^{q/2}a_n^{2q}|s_2-s_1|^{q/2}$. As for the sum, we apply Rosenthal's
inequality and obtain for some constant $C'_q>0$,
\[
\E\Big[\Big|\sum_{i=1}^{\theta_n(s_2-s_1)}\vartheta_i\Big|^q\Big]\leq C'_q\theta_n^{q/2}|s_2-s_1|^{q/2}.
\]
Using once more that $a_n^{-2q}\theta_n^{q/2}\rightarrow 0$ by
assumption on $p_n$, the lemma is proved.
\end{proof}

Let $\kappa>0$. By the theorem of Kolmogorov-\v{C}entsov (see~\cite[Theorem
2.8]{KaSh}), it follows from the above lemma that there exists
$M=M(\kappa)>0$ such that for all $n\in\N$, the event
\[\mathcal{E}_n=\left\{\sup_{0\leq s< t\leq K}\frac{|C_\infty^{(n)}(\theta_n
    s)-C_\infty^{(n)}(\theta_n t)|}{a_n^{2}|s-t|^{2/5}}\leq M\right\}\]
has probability at least $1-\kappa$. We will now work conditionally given $\mathcal{E}_n$.
\begin{lemma}
\label{lem:label-hoelder}
In the setting from above, there exists a
constant $C'>0$ such that for all $n\in\N$ and all
$0\leq s_1,s_2\leq K$,
\[
\E\left[a_n^{-6}|\la_\infty^{(n)}(\theta_ns_1)-\la_\infty^{(n)}(\theta_ns_2)|^{6}\,\Big
  |\,\mathcal{E}_n\right]\leq C'|s_1-s_2|^{6/5}.
\]
\end{lemma}
Tightness of the conditional laws of $a_n^{-1}\la_\infty^{(n)}(\theta_ns)$,
$0\leq s\leq K$, given $\mathcal{E}_n$ is a standard consequence of this
lemma; see~\cite[Problem 4.11]{KaSh}). Since $\kappa$ in the definition of
$\mathcal{E}_n$ can be chosen arbitrarily small, tightness of the
unconditioned laws of the rescaled labels follows, and so does 
Lemma~\ref{lem:conv-contlabelbridge}.
\end{proof}
It therefore only remains to prove Lemma~\ref{lem:label-hoelder}.
\begin{proof}[Proof of Lemma~\ref{lem:label-hoelder}]
With arguments similar to those in the proof of Lemma~\ref{lem:contour-hoelder}, we see that it suffices to
prove the claim in the case where $\theta_ns_1$ and
$\theta_ns_2$ are integers (and $s_1\leq s_2$). Let
\[\Delta C_\infty^{(n)}(s_1,s_2)
=C_\infty^{(n)}(\theta_ns_1)+C_\infty^{(n)}(\theta_ns_2)-2\min_{[\theta_ns_1,\theta_ns_2]}C_\infty^{(n)}.\]
By definition of $(C_\infty^{(n)},\la_\infty^{(n)})$, conditionally given
$C_\infty^{(n)}$ on $[0,K]$, the difference
$|\la_\infty^{(n)}(\theta_ns_2)-\la_\infty^{(n)}(\theta_ns_1)|$ is
distributed as a sum of i.i.d. variables $\eta_i$ with the uniform law on
$\{-1,0,1\}$. By construction, the sum involves at most
$\Delta C_\infty^{(n)}(s_1,s_2)$ summands: Indeed, it involves exactly
$\Delta C_\infty^{(n)}(s_1,s_2)$ many summands if
$C_\infty^{(n)}(\theta_ns_1)$ and $C_\infty^{(n)}(\theta_ns_2)$ encode
vertices of the same tree, and less than $\Delta C_\infty^{(n)}(s_1,s_2)$
many summands if they encode vertices of different trees.  Again with
Rosenthal's inequality, we thus obtain for some $\tilde{C}>0$,
\begin{align*}
  \E\left[a_n^{-6}|\la_\infty^{(n)}(\theta_ns_2)-\la_\infty^{(n)}(\theta_ns_1)|^{6}\,\Big
    |\,\mathcal{E}_n\right]&\leq
    a_n^{-6}\E\left[\Big|\sum_{i=1}^{\Delta C_\infty^{(n)}(s_1,s_2)}\eta_i\Big|^{6}\,\Big
    |\,\mathcal{E}_n\right]\\
  &\leq \tilde{C}a_n^{-6}\E\left[|\Delta C_{\infty}^{n}(s_1,s_2)|^{3}\,\Big |\,\mathcal{E}_n\right].
\end{align*}
On $\mathcal{E}_n$, we have the bound
\[
a_n^{-2}|\Delta C_{\infty}^{n}(s_1,s_2)|\leq 2\sup_{0\leq s<t\leq
  K}\frac{|C_{\infty}^{n}(\theta_ns)-C_{\infty}^{n}(\theta_nt)|}{a_n^2|s-t|^{2/5}}|s_1-s_2|^{2/5}\leq M|s_1-s_2|^{2/5},
\]
and the claim of the lemma follows.
\end{proof}

Finally, for proving Theorem~\ref{thm:GHconv-ICRT}, we will make use of the
following lemma.
\begin{lemma}[Lemma 5.7 of~\cite{BaMiRa}]
\label{lem:localGHconv}
  Let $r\geq 0$. Let $\mathbf{E}=(E,d,\rho)$ and
  $\mathbf{E}'=(E',d',\rho')$ be two pointed complete and locally compact length
  spaces. Consider a subset $\cR\subset E\times E'$ which has the following
  properties:
\begin{itemize}
\item $(\rho,\rho')\in\cR$,
\item for all $x \in B_r(\mathbf{E})$, there exists $x'\in E'$ such that
  $(x,x')\in\cR$,
\item for all $y' \in B_r(\mathbf{E}')$, there exists $y\in E$ such that $(y,y')\in\cR$.
\end{itemize}
Then, 
$\dgh(B_r(\mathbf{E}),B_r(\mathbf{E}'))\leq (3/2)\sup\left\{|d(x,y)-d'(x',y')| : (x,x'), (y,y')\in\cR\right\}.
$
\end{lemma}
A proof is given in~\cite{BaMiRa}. Although $\cR$ is not necessarily a
correspondence in the sense of~\cite{BuBuIv}, we might call the  supremum on the right side
of the inequality the {\it distortion} of $\cR$.

\begin{proof}[Proof of Theorem~\ref{thm:GHconv-ICRT}]
  We let $(a_n,n\in\N)$ and $(p_n,n\in\N)\subset [0,1/2]$ be two sequences
  as in the statement, and, as mentioned at the beginning of this section, we
  assume that the $\UIHPQ_{p_n}$ $Q_\infty^\infty(p_n)$ with skewness parameter $p_n$ is encoded
  in terms of $((\f_\infty^{(n)},\la_\infty^{(n)}),\br_\infty)$. Local Gromov-Hausdorff convergence in law of
  $a_n^{-1}\cdot Q_\infty^\infty(p_n)$ towards the $\ICRT$ follows if we prove that for
  each $r\geq 0$,
\begin{equation}
\label{eq:GHconv-ICRT-eq1}
  B_r^{(0)}\left(a_n^{-1}\cdot
    Q_\infty^{\infty}(p_n)\right)\xrightarrow{(d)}B_r(\ICRT)
\end{equation}
  in distribution in $\mathbb{K}$, where we recall again that
  $B_{r}^{(0)}(a_n^{-1}\cdot Q_\infty^{\infty}(p_n))$ denotes the ball of
  radius $r$ around the vertex $\f_\infty^{(n)}(0)$ in the rescaled
  $\UIHPQ_{p_n}$. 

  We will show the claim for $r=1$. The proof follows essentially the line
  of argumentation in~\cite[Proof of Theorem 3.5]{BaMiRa}; since the
  argument is short, we repeat the main steps for completeness. We will
  apply Lemma~\ref{lem:localGHconv} in the following way. The $\ICRT$ takes
  the role of the space $\mathbf{E}'$, with the equivalence class $[0]$ of
  zero being the distinguished point. Then, we consider for each $n\in\N$
  the space $a_n^{-1}\cdot Q_\infty^{\infty}(p_n)$ pointed at
  $\f_\infty^{(n)}(0)$, which takes the role of $\mathbf{E}$ in the
  lemma. We construct a subset $\cR_n\subset E\times E'$ with the
  properties of Lemma~\ref{lem:localGHconv}, such that its distortion, that
  is, the quantity
  \[\dis(\cR_n)=\sup\left\{|d(x,y)-d'(x',y')| : (x,x'), (y,y')\in\cR_n\right\}\,\] is of
  order $o(1)$ for $n$ tending to infinity. By Lemma~\ref{lem:localGHconv},
  this will prove~\eqref{eq:GHconv-ICRT-eq1}.  We remark that
  $Q_\infty^{\infty}(p_n)$ is not a length space, hence
  Lemma~\ref{lem:localGHconv} seems not applicable at first sight. However,
  as explained in Section~\ref{sec:GH}, by identifying each edge with a
  copy of $[0,1]$ and upon extending the graph metric isometrically, we may
  identify $Q_\infty^{\infty}(p_n)$ with the (associated) length space,
  which we denote by
  $\mathbf{Q}_\infty^{\infty}(p_n)=(V(\mathbf{Q}_\infty^{\infty}(p_n)),\dgr,\rho)$. Here
  and in what follows, $\bdgr$ is the graph metric isometrically extended
  to $\mathbf{Q}_\infty^{\infty}(p_n)$. Note that the vertex set
  $V(\f_\infty^{(n)})$ may be viewed as a subset of
  $\mathbf{Q}_\infty^{\infty}(p_n)$, and between points of
  $V(\f_\infty^{(n)})$, the distances $\bdgr$ and $\dgr$ agree. Moreover,
  as a matter of fact, every point in $\mathbf{Q}_\infty^{\infty}(p_n)$ is
  at distance at most $1/2$ away from a vertex of $\f_\infty^{(n)}$.

  Recall that $(\br_\infty(t),t\in\R)$ has the law of a (linearly
  interpolated) two-sided symmetric simple random walk with
  $\br_\infty(0)=0$. Let $X=(X_t,t\in\R)$ be a two-sided Brownian motion
  with $X_0=0$. By Donsker's invariance principle, we deduce that as $n$
  tends to infinity,
\begin{equation}
\label{eq:GHconv-ICRT-eq2}
\left(a_n^{-1}\br_\infty(a_n^2t),t\in\R\right)\xrightarrow[]{(d)}\left(X_t,t\in\R\right).
\end{equation}
Using Skorokhod's representation theorem, we can assume that
the above convergence holds almost surely on a common probability space, uniformly over compacts. Now let
$\delta>0$, and fix $\alpha>0$ and $n_0\in \N$ such that the event
\[
\mathcal{E}(n,\alpha)=\left\{\max\Big\{\min_{[0,\alpha]}X_,\min_{[-\alpha,0]}X\Big\}<-1\right\}\bigcap
\left\{\max\Big\{\min_{[0,\alpha a_n^2]}\br_\infty,\min_{[-\alpha
    a_n^2,0]}\br_\infty\Big\} < -a_n\right\}
\]
has probability at least $1-\delta$ for $n\geq n_0$. From now on, we
argue on the event $\mathcal{E}(n,\alpha)$. We moreover assume that the $\ICRT$ $(\cT_X,d_{X},[0])$
is defined in terms of $X$, and we write $p_X:\R\rightarrow \cT_X$ for the
canonical projection.

Recall that the vertices of $\f_\infty^{(n)}=(\tr_i,i\in\Z)$ are identified
with the vertices of $Q_\infty^{\infty}(p_n)$. The mapping
$\mathcal{I}(v)\in\Z$ gives back the index of the tree a vertex
$v\in V(\f_\infty^{(n)})$ belongs to. We extend $\mathcal{I}$ to
the elements of the length space $\mathbf{Q}_\infty^{\infty}(p_n)$ as
follows. By viewing $V(\f_\infty^{(n)})$ as a subset of
$\mathbf{Q}_\infty^{\infty}(p_n)$ as explained above, we associate to every
point $u$ of $\mathbf{Q}_\infty^{\infty}(p_n)$ its closest vertex
$v\in V(\f_\infty^{(n)})$ satisfying $\bdgr(\f_\infty^{(n)}(0),v)\geq \dgr(\f_\infty^{(n)}(0),u)$.
Note again $\bdgr(v,u)\leq 1/2$. Put
\[
A_n=\{u\in \mathbf{Q}_\infty^{\infty}(p_n): \mathcal{I}(u)\in [-\alpha
a_n^2,\alpha a_n^2]\}.
\]
A direct application of the cactus bound~\cite[(4.6) of Section
4.5]{BaMiRa} shows that on $\mathcal{E}(n,\alpha)$, almost surely 
\[
\bdgr(\f_\infty^{(n)}(0),u)> a_n\quad\textup{whenever }\mathcal{I}(u)\notin A_n,
\]
implying that the set $A_n$ contains the ball
$B_1^{(0)}(\mathbf{Q}_\infty^{\infty}(p_n))$ of radius $1$ around the
vertex $\f_\infty^{(n)}(0)$. Moreover, still on $\mathcal{E}(n,\alpha)$,
\[
d_X([0],t)>1 \quad\textup{whenever }|t|>\alpha.
\]
We now define $\cR_n\subset
\mathbf{Q}_\infty^{\infty}(p_n)\times \cT_X$ by
\[
\cR_n=\left\{(u,p_X(t)): u\in A_n, t\in[-\alpha,\alpha]\textup{ with
  }\mathcal{I}(u)=\lfloor ta_n^2\rfloor\right\}.
\]
Letting
${\bf E}=(\mathbf{Q}_\infty^{\infty}(p_n),a_n^{-1}\bdgr,\f_{\infty}^{(n)}(0))$,
${\bf E'}=(\cT_X,d_X,[0])$, $r=1$, we find that given the event
$\mathcal{E}(n,\alpha)$, the set $\cR_n$ satisfies the requirements of
Lemma~\ref{lem:localGHconv}. We are now in the setting of~\cite[Proof
of Theorem 3.5]{BaMiRa}: All what is left to show is that on
$\mathcal{E}(n,\alpha)$, the distortion of $\cR_n$ converges to $0$ in
probability. However, with the same arguments as in the cited proof and using
the convergence~\eqref{eq:GHconv-ICRT-eq2}, we obtain
\[\limsup_{n\rightarrow\infty}\dis(\cR_n)\leq \limsup_{n\rightarrow\infty}\frac{10\max_{A_n}|\la_\infty^{(n)}|}{a_n}\,.\]
An appeal to Lemma~\ref{lem:conv-contlabelbridge} shows that the right-hand
side is equal to zero, and the proof of the theorem is completed.
\end{proof}

\section{Proofs of the structural properties}\label{sec:proofs-struct}
\subsection{The branching structure behind the $\normalfont{\UIHPQ}_p$}
In this section, we describe the branching structure of the $\UIHPQ_p$ and
prove Theorem~\ref{thm:BranchingStructure}. We will first study a similar
mechanism behind Boltzmann quadrangulations $Q$ and $Q_\sigma$ drawn
according to $\P_{g_p,z_p}$ and $\P_{g_p}^\sigma$, respectively
(Proposition~\ref{prop:LawTreeComponents} and
Corollary~\ref{cor:LawTCsigma}), and then pass to the limit
$\sigma\rightarrow\infty$ using Proposition~\ref{prop:Boltzmann}.

To begin with, we follow an idea
of~\cite{CuKo2}: We associate to a (finite) rooted map a tree that
describes the branching structure of the boundary of the map. Precisely,
for every finite rooted quadrangulation $\q$ with a boundary, we define the
so-called \textit{scooped-out quadrangulation} $\Scoop(\q)$ as follows. We
keep only the boundary edges of $\q$ and duplicate those edges which lie
entirely in the outer face (i.e., whose both sides belong to the outer
face). The resulting object is a rooted looptree; see Figure
\ref{fig:ScoopedOutQuad}.

\begin{figure}[h!]
  \begin{center}
    \includegraphics[scale=.85]{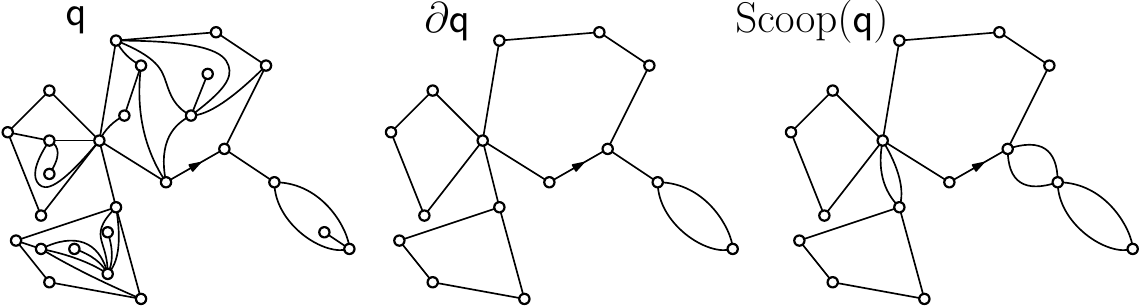}
  \end{center}
  \caption{A rooted quadrangulation, its boundary and the associated
    scooped-out quadrangulation.}
  \label{fig:ScoopedOutQuad}
\end{figure}

To a scooped-out quadrangulation we associate its tree of components
$\Tree(\Scoop(\q))$ as defined in Section
\ref{sec:RandomLooptrees}. Following~\cite{CuKo2}, we call this tree, by a
slight abuse of terminology, the tree of components of $\q$ and use the
notation $\tr=\Treeb(\q)$. It is seen that vertices in $V(\tr_\bullet)$ have
even degree in $\tr$, due to the bipartite nature of $\q$.

By gluing the appropriate rooted quadrangulation with a simple boundary
into each cycle of $\Scoop(\q)$, we recover the quadrangulation $\q$. This
provides a bijection
\[\Psi : \q \mapsto (\Treeb(\q),(\widehat{\q}_u : u\in
V(\Treeb(\q)_{\bullet}))\]
between, on the one hand, the set $\cQ_f$ of finite rooted quadrangulations
with a boundary and, on the other hand, the set of plane trees $\tr$ with
vertices at odd height having even degree, together with a collection
$(\widehat{\q}_u : u\in V(\tr_{\bullet}))$ of rooted quadrangulations with
a simple boundary and respective perimeter $\deg(u)$, for $\deg(u)$ the
degree of $u$ in $\tr$. We remark that the inverse mapping $\Psi^{-1}$ can
 be extended to an infinite but locally finite tree together with a
collection of quadrangulations with a simple boundary attached to vertices
at odd height, yielding in this case an infinite rooted quadrangulation
$\q$.

Recall from Section~\ref{sec:BoltzmannQuadrangulations} the definitions
of the Boltzmann laws $\P_{g,z}$ and $\P_g^{\sigma}$, and their analogs
with support on quadrangulations with a simple boundary,
$\widehat{\P}_{g,z}$ and $\widehat{\P}_g^{\sigma}$. Their corresponding
partition functions are $F$, $F_{\sigma}$ and $\widehat{F}$,
$\widehat{F}_{\sigma}$. We are now interested in the law of the tree of
components under $\Pgz$.  To begin with, we adapt some enumeration results
from~\cite{BoGu} to our setting. For every $0 \leq p\leq 1/2$, recall that
$g_p=p(1-p)/3$ and $z_p=(1-p)/4$. Then, (3.15), (3.27) and (5.16)
of~\cite{BoGu} all together provide the identities 
\begin{equation}\label{eq:EnumF}
  F(g_p,z_p)=\frac{2}{3}\frac{3-4p}{1-p}, \quad
  F_{\sigma}(g_p)=\frac{(2\sigma)!}{\sigma!(\sigma+2)!}\left(2+\sigma
    \frac{1-2p}{1-p}\right)\left(\frac{1}{1-p}\right)^\sigma, 
\end{equation} 
for $0\leq p\leq 1/2$ and $\sigma\in\N_0$.  Moreover, for $\sigma\in\N$ and
$0<p\leq 1/2$,
\begin{equation}\label{eq:EnumFHat}
  \widehat{F}_{\sigma}(g_p)=\left(\frac{p}{3(1-p)^2}\right)^\sigma
  \frac{(3\sigma-2)!}{\sigma ! (2\sigma-1)!}\left( \frac{3\sigma
      (1-p)}{p}+2-3\sigma\right),  
\end{equation} 
while $\widehat{F}_{0}(g_p)=1$. If $p=0$ and hence $g_p=0$, then
$\widehat{F}_{k}(0)=\delta_0(k)+\delta_1(k)$ for all $k\in\N_0$. (Indeed,
under the maps with no inner faces, the vertex map and the map consisting
of one oriented edge are the only maps with a simple boundary.)

We already introduced in Section~\ref{sec:TreeStructure} two probability
measures $\muw$ and $\mub$ on $\N_0$ given by
\begin{align}
\label{eq:muw}
  \muw(k)&=\frac{1}{F(g_p,z_p)}\left(1-\frac{1}{F(g_p,z_p)}\right)^k,
           \quad k\in \N_0, \\
\label{eq:mub} 
  \mub(2k+1)&=\frac{1}{F(g_p,z_p)-1}\left[z_p F^2(g_p,z_p)
              \right]^{k+1}\widehat{F}_{k+1}(g_p), \quad k\in \N_0, 
\end{align}with $\mub(k)=0$ if $k$ even. 
The tree of
components of the scooped-out quadrangulation $\Scoop(Q)$ when $Q$ is
drawn according to $\P_{g_p,z_p}$ may now be characterized as follows.

\begin{prop}\label{prop:LawTreeComponents}Let $0\leq p\leq 1/2$, and let
  $Q$ be distributed according to $\P_{g_p,z_p}$. Then the tree of
  components $\Treeb(Q)$ is a two-type Galton-Watson tree with offspring
  distribution $(\muw,\mub)$ as given above. Moreover, conditionally on
  $\Treeb(Q)$, the quadrangulations with a simple boundary associated to
  $Q$ via the bijection $\Psi$ are independent with respective Boltzmann
  distribution $\widehat{\P}^{\deg(u)}_{g_p}$ for
  $u\in V(\Treeb(Q)_{\bullet})$, where $\deg(u)$ denotes the degree of $u$
  in $\Treeb(Q)$.
\end{prop}

\begin{proof} Note that vertices at even height of $\Treeb(Q)$ have an odd
  number of offspring almost surely. Let $\tr$ be a finite plane tree
  satisfying this property. Let also
  $(\widehat{\q}_u : u\in V(\tr_{\bullet}))$ be a collection of rooted
  quadrangulations with a simple boundary and respective perimeters
  $\deg(u)$, and set
  $\q=\Psi^{-1}(\tr,(\widehat{\q}_u : u\in V(\tr_{\bullet})))$. Then,
  writing $\Psi_\ast\P$ for the push-forward measure of $\P$ by $\Psi$,
	\begin{align*}
  \Psi_\ast\P_{g_p,z_p}\left(\tr,(\widehat{\q}_u : u\in
  V(\tr_{\bullet}))\right)=&\frac{z_p^{\#\Bq/2}g_p^{\#\textup{F}(\q)}}{F(g_p,z_p)}=\frac{1}{F(g_p,z_p)}\prod_{u\in
                   V(\tr_\bullet)}z_p^{\deg(u)/2}g_p^{\#\textup{F}(\widehat{\q}_u)}. 
\end{align*} For every $c>0$, we have
	\[1=\prod_{u\in V(\tr_\circ)}{c^{k_u}\left(\frac{1}{c}\right)^{\#V(\tr_\bullet)}}
\quad \text{and} \quad \frac{1}{c}=\prod_{u\in
  V(\tr_\bullet)}{c^{k_u}\left(\frac{1}{c}\right)^{\#V(\tr_\circ)}}. \]
Applying the first equality with $c=1-1/F(g_p,z_p)$ and the second one with
$c=F(g_p,z_p)$ gives
\begin{multline*}
  \Psi_\ast\P_{g_p,z_p}\left(\tr,(\widehat{\q}_u : u\in V(\tr_\bullet)\right)=\prod_{u\in
    V(\tr_\circ)}{\frac{1}{F(g_p,z_p)}\left(1-\frac{1}{F(g_p,z_p)}\right)^{\deg(u)-1}}\\
  \times\prod_{u\in V(\tr_\bullet)}{\frac{1}{F(g_p,z_p)-1}\left(z_p F^2(g_p,z_p)
    \right)^{\deg(u)/2}\widehat{F}_{\deg(u)/2}(g_p)}\prod_{u\in
    V(\tr_\bullet)}{\frac{g_p^{\#\textup{F}(\widehat{\q}_u)}}{\widehat{F}_{\deg(u)/2}(g_p)}},
\end{multline*} 
where we agree that $0/0=0$. Therefore,
\[\Psi_\ast\P_{g_p,z_p}\left(\tr,(\widehat{\q}_u : u\in V(\tr_\bullet))\right)=\prod_{u\in
  V(\tr_\circ)}{\muw(k_u)}\prod_{u\in V(\tr_\bullet)}{\mub(k_u)}\prod_{u\in
  V(\tr_\bullet)}{\widehat{\P}^{\deg(u)}_{g_p}(\widehat{\q}_u)},\]
which is the expected result.\end{proof}

\begin{corollary}\label{cor:LawTCsigma}Let $0\leq p \leq 1/2$, $\sigma\in\N$, and let $Q$ be
  distributed according to $\P^{\sigma}_{g_p}$. Then the tree of components
  $\Treeb(Q)$ is a two-type Galton-Watson tree with offspring distribution
  $(\muw,\mub)$ conditioned to have $2\sigma+1$ vertices.  Moreover,
  conditionally on $\Treeb(Q)$, the quadrangulations with a simple boundary
  associated to $Q$ via the bijection $\Psi$ are independent with
  respective Boltzmann distribution $\widehat{\P}^{\deg(u)}_{g_p}$, for
  $u\in V(\Treeb(Q)_{\bullet})$.
\end{corollary}

\begin{proof} Observing that $\# V(\Treeb(\q)) = \#\Bq +1$ for every rooted
  quadrangulation $\q$, we obtain
\[\P_{g_p,z_p}(\cQ^{\sigma}_f)=\Psi_\ast\P_{g_p,z_p}\left(\{ \tr \in
  \mathbb{T}_f : \# V(\tr) =2\sigma+1\}\right)=\mathsf{GW}_{\muw,\mub}(\{ \tr
\in \mathbb{T}_f : \# V(\tr) =2\sigma+1\}) .\]
Now let $\tr$ be a finite plane tree with an odd number of offspring at even
height, and let $(\widehat{\q}_u : u\in V(\tr_{\bullet}))$ and $\q$ be as in the proof
of Proposition~\ref{prop:LawTreeComponents}.  Then,
\[\Psi_\ast\P^{\sigma}_{g_p}\left(\tr,(\widehat{\q}_u : u\in V(\tr_\bullet))\right)=\frac{\mathbf{1}_{\{\#\Bq=2\sigma
    \}}}{\P_{g_p,z_p}(\cQ^{\sigma}_f)}\prod_{u\in
  V(\tr_\circ)}{\muw(k_u)}\prod_{u\in V(\tr_\bullet)}{\mub(k_u)}\prod_{u\in
  V(\tr_\bullet)}{\widehat{\P}^{\deg(u)}_{g_p}(\widehat{\q}_u)},\]
which concludes the proof.\end{proof}

\begin{lemma}\label{lem:Criticality} For $0\leq p <1/2$, the pair
  $(\muw,\mub)$ is critical and both $\muw$ and $\mub$ have small
  exponential moments. For $p=1/2$, the pair $(\muw,\mub)$ is subcritical
  (and $\mub$ has no exponential moment).
\end{lemma}

\begin{proof}Recall that $(\muw,\mub)$ is critical if and only if the
  product of their respective means $m_{\circ}$ and $m_{\bullet}$ equals
  one. Since by~\eqref{eq:muw}, $\muw$ is the geometric law with parameter
  $1/F(g_p,z_p)$, we have 
  \[m_{\circ}=F(g_p,z_p)-1.\]
  For $m_{\bullet}$, we let $G_{\mub}$ denote the generating function of
  $\mub$. By~\eqref{eq:mub}, it follows that 
\[G_{\mub}(s)=\frac{1}{F(g_p,z_p)-1}\frac{1}{s}\left(
  \widehat{F}\left[g_p,z_p F^2(g_p,z_p)s^2\right]-1\right),\quad s>0.
\]
Then, Identity (2.8) of~\cite{BoGu} ensures that
$\widehat{F}(g,z F^2(g,z))=F(g,z)$ for all non-negative weights $g$ and
$z$. When differentiating this relation with respect to  
the variable $z$, we obtain
\begin{equation}
\label{eq:crit-eq1}
\partial_z\widehat{F}(g,z F^2(g,z))=\frac{\partial_z F(g,z)}{F^2(g,z)+2z
  F(g,z)\partial_z F(g,z)}.
\end{equation}
Writing
\[
\partial_z F(g_p,z_p)=\sum_{\sigma\geq 0}\sigma F_{\sigma}(g_p)z_p^{\sigma-1},
\]
and using the exact expression for $F_{\sigma}(g_p)$
from~\eqref{eq:EnumF}, we see by means of Stirling's formula that
$\partial_z F(g_p,z_p)=\infty$ for $p\in [0,1/2)$, and
$\partial_z F(g_p,z_p)<\infty$ for $p=1/2$. Thus, for $p \in [0,1/2)$,
\[\partial_z\widehat{F}(g_p,z_p F^2(g_p,z_p))=\frac{1}{2z_p F(g_p,z_p)},\]
whereas if $p=1/2$, the derivative on the left-hand side
in~\eqref{eq:crit-eq1} is strictly smaller than the right-hand
side for $g=g_p$, $z=z_p$. Finally, applying Identity (2.8) of~\cite{BoGu} once again, we get
\begin{align*}
  m_{\bullet}=&G'_{\mub}(1)\\
  =&\frac{1}{F(g_p,z_p)-1}\left( -\left(\widehat{F}\left[g_p,z_p
     F^2(g_p,z_p)\right]-1\right)+2 z_p F^2(g_p,z_p)\partial_z
     \widehat{F}\left[g_p,z_p F^2(g_p,z_p)\right]  \right). 
\end{align*}
As a consequence, $m_{\circ}m_{\bullet}=1$ if $p<1/2$, and
$m_{\circ}m_{\bullet}<1$ if $p=1/2$. The fact that $\muw$ has exponential
moments is clear. For $\mub$, one sees from~\eqref{eq:EnumFHat} that the
power series
\[\sum_{k\geq 0}x^k \widehat{F}_k(g_p)\] has radius of convergence
$\widehat{r}_p=4(1-p)^2/(9p)$, while~\eqref{eq:EnumF} ensures that  
\[z_p F^2(g_p,z_p)=\frac{(1-4p)^2}{9(1-p)}.\]
Again, for $p\in[0,1/2)$, $\widehat{r}_p>z_p F^2(g_p,z_p)$, and these
quantities are equal for $p=1/2$. Thus, there exists $s>1$ such that
$G_{\mub}(s)<\infty$ if and only if $p<1/2$, which concludes the
proof.\end{proof}

We are now ready to prove Theorem~\ref{thm:BranchingStructure}.
\begin{proof}[Proof of Theorem~\ref{thm:BranchingStructure}] 
  Fix $0\leq p<1/2$. Let us denote by $\mathsf{Q}_\infty$ the random
  quadrangulation with an infinite boundary as constructed in the statement
  of Theorem~\ref{thm:BranchingStructure}, and let $Q_\sigma$ be
  distributed according to $\P^{\sigma}_{g_p}$. In view of
  Proposition~\ref{prop:Boltzmann}, it is sufficient to prove that in the
  local sense, as $\sigma\rightarrow\infty$,
\begin{equation}\label{eq:IntermediateLocalConvergence}
	Q_{\sigma}\xrightarrow[]{(d)} \mathsf{Q}_\infty.
\end{equation} 
For every real $r\geq 1$ and every (finite or infinite) plane tree $\tr$,
we define $\Cut_r(\tr)$ as the finite plane tree obtained from pruning all
the vertices at a height larger than $2r$ in $\tr$. If $\q\in\cQ$ is a
quadrangulation with a boundary such that
$\Psi(\q)=(\tr,(\widehat{\q}_u : u\in V(\tr_{\bullet})))$, we define
$\Cut_r(\q)$ to be the quadrangulation obtained from gluing the maps
$(\widehat{\q}_u : u\in V(\Cut_r(\tr)_\bullet))$ in the associated loops of
$\Loop(\Cut_r(\tr))$. With this definition, we have
$B_r(\q)\subset \Cut_r(\q)$ for every $r\geq 1$, where we recall that
$B_r(\q)$ stands for the closed ball of radius $r$ around the root in $\q$.

Let $r\geq 1$ and $\q\in \cQ_f$ such that
$\Psi(\q)=(\tr,(\widehat{\q}_u : u\in V(\tr_{\bullet})))$. Using
Proposition~\ref{prop:LawTreeComponents} and Corollary~\ref{cor:LawTCsigma}, we
get
\[\P\left( \Cut_r\left( Q_\sigma \right)=\q
\right)=\mathsf{GW}_{\muw,\mub}^{(2\sigma+1)}\left( \Cut_r = \tr
\right)\prod_{u \in
  V(\tr_\bullet)}\widehat{\P}^{\deg(u)}_{g_p}\left(\widehat{\q}_u
\right), \]where
we use the notation $\mathsf{GW}_{\muw,\mub}^{(2\sigma +1)}$ for the
$(\muw,\mub)$-Galton-Watson tree conditioned to have $2\sigma +1$ vertices
and interpret $\Cut_r$ as the random variable $\tr \mapsto \Cut_r(\tr)$.
Applying Proposition~\ref{prop:KestenTreeCvgce}, we get as
$\sigma \rightarrow \infty$
\[\P\left( \Cut_r\left( Q_\sigma \right)=\q \right) \longrightarrow
\mathsf{GW}_{\muw,\mub}^{(\infty)}\left( \Cut_r = \tr \right)\prod_{u \in
  V(\tr_\bullet)}\widehat{\P}^{\deg(u)}_{g_p}\left(\widehat{\q}_u
\right)=\P\left( \Cut_r\left( \mathsf{Q}_\infty \right)=\q \right).\] 
We proved that for every $r\geq 1$, as $\sigma \rightarrow \infty$,
\[\Cut_r\left( Q_\sigma \right)\overset{(d)}{\longrightarrow}
\Cut_r\left( \mathsf{Q}_\infty\right).\]
Since $B_r(\q)\subset \Cut_r(\q)$ for every $r\geq 1$ and
$\q\in\cQ$,~\eqref{eq:IntermediateLocalConvergence} holds and the theorem
follows.
\end{proof}

\subsection{Recurrence of simple random walk}
In this final part, we prove Corollary~\ref{cor:recurrence}, stating that
simple random walk on the $\UIHPQ_p$ for $0\leq p < 1/2$ is almost surely
recurrent. We will use a criterion from the theory of electrical networks;
see, e.g.,~\cite[Chapter 2]{LyPe} for an introduction into these
techniques.
\begin{proof}[Proof of Corollary~\ref{cor:recurrence}]
  Fix $0\leq p<1/2$. We interpret the $\UIHPQ_p$ as an electrical network,
  by equipping each edge with a resistance of strength one. A cutset
  $\mC$ between the root vertex and infinity is a set of edges that
  separates the root from infinity, in the sense that every infinite
  self-avoiding path starting from the root has to pass through at least
  one edge of $\mC$. By the criterion of Nash-Williams,
  cf.~\cite[(2.13)]{LyPe}, it suffices to show that there is a collection
  $(\mC_n,n\in\N)$ of disjoint cutsets such that
  $\sum_{n=1}^\infty(1/\#\mC_n)=\infty$ almost surely, i.e., for almost every
  realization of the $\UIHPQ_p$.

  We recall the construction of the $\UIHPQ_p$ in terms of the looptree
  associated to Kesten's two-type tree
  $\cT_\infty=\cT_\infty(\muw,\mub)$. Note that the white vertices in
  $\cT_\infty$, i.e., the vertices at even height, represent vertices in
  the $\UIHPQ_p$.  More precisely, by construction, they form the boundary
  vertices of the latter. In particular, the white vertices on the spine of
  $\cT_\infty$ are to be found in the $\UIHPQ_p$, and we enumerate them by
  $v_1,v_2,v_3,\ldots$, such that $v_1$ is the root vertex, and
  $\dgr(v_j,v_1)\geq \dgr(v_i,v_1)$ for $j\geq i$.  Now observe that for
  $i\in\N$, $v_i$ and $v_{i+1}$ lie on the boundary of one common
  finite-size quadrangulation with a simple boundary, which we denote by
  $\widehat{\q}_{v_i}$, in accordance with notation in the proof of
  Theorem~\ref{thm:BranchingStructure}.

  We define $\mC_i$ to be the set of all the edges of
  $\widehat{\q}_{v_i}$. Clearly, for each $i\in\N$, $\mC_i$ is a cutset
  between the root vertex and infinity, and for $i\neq j$, $\mC_i$ and
  $\mC_j$ are disjoint. The sizes $\#\mC_i$, $i\in\N$, are i.i.d. random
  variables. More specifically, using the construction of the $\UIHPQ_p$ in
  terms of Kesten's looptree, the law of $\#\mC_1$ can be described as
  follows: First, draw a random variable $Y$ according to the size-biased
  offspring distribution $\bar{\mu}_\bullet$, and then, conditionally on
  $Y$, $\#\mC_1$ is distributed as the number of edges of a Boltzmann
  quadrangulation with law $\widehat{\P}_{g_p}^{(Y+1)/2}$, where
  $g_p=p(1-p)/3$. Obviously, $\#\mC_1$ is finite almost surely, implying
  $\sum_{n=1}^\infty(1/\#\mC_n)=\infty$ almost surely, and recurrence of
  the simple symmetric random walk on the $\UIHPQ_p$ follows.
\end{proof}
\begin{remark}
  Let us end with a remark concerning the structure of the $\UIHPQ_p$ for
  $p<1/2$. Note that with probability $\bar{\mu}_\bullet(1)>0$, a cutset
  $\mC_i$ as constructed in the above proof consists exactly of one edge. By
  independence and Borel-Cantelli, we thus find with probability one an
  infinite sequence of such cutsets $\mC_{i_1},\mC_{i_2},\ldots$ consisting of
  one edge only. In particular, this proves that the $\UIHPQ_p$ for $p<1/2$
  admits a decomposition into a sequence of almost surely finite
  i.i.d. quadrangulations $Q_i(p)$ with a non-simple boundary (whose laws 
  can explicitly be derived from Theorem~\ref{thm:BranchingStructure}),
  such that $Q_i(p)$ and $Q_j(p)$ get connected by a single edge if and
  only if $|i-j|=1$. This parallels the decomposition of the spaces
  $\mathbb{H}_\alpha$ for $\alpha<2/3$ found in~\cite[Display (2.3)]{Ra}.
\end{remark}

  \bigskip
\noindent {\bf Acknowledgments.}
We warmly thank Gr\'egory Miermont for stimulating discussions and advice,
and an anonymous referee for useful comments.


\begin{thebibliography}{Heu}
\addcontentsline{toc}{section}{Literature}
\bibitem{Ab} Abraham, C. {Rescaled bipartite planar maps converge to the
    Brownian map.} {\it Ann. Inst. H.  Poincar\'e Probab. Statist.} {\bf
    52}(2) (2016), 575--595.
\bibitem{Al} {Aldous, D.} {The continuum random tree I.}
  {\it Ann. Probab.} {\bf 19} (1991), 1--28.
\bibitem{An} Angel, O. {Scaling of percolation on infinite planar maps,
i.} arXiv:0501006.
\bibitem{AnCu} Angel, O., Curien, N. {Percolation on random maps I:
    half-plane models}. {\it Ann. Inst. H.  Poincar\'e Probab. Statist.} {\bf
    51}(2) (2015), 405--431.
\bibitem{AnRa1} Angel, O., Ray, G. {Classification of half-planar maps.}
  {\it Ann. Probab.} {\bf 43}(3) (2015), 1315--1349.
\bibitem{AnRa2} Angel, O., Ray, G. {The half plane UIPT is recurrent.}
{\it Probab. Theory Relat. Fields} {\bf 170}(3-4) (2018), 657--683.
\bibitem{AnSc} Angel, O., Schramm, O. {Uniform infinite planar
    triangulations}. {\it Comm. Math. Phys.} {\bf 241}(2-3) (2003), 191--213.
\bibitem{BaMiRa} Baur, E., Miermont, G., Ray, G. {Classification of
    scaling limits of uniform quadrangulations with a boundary.} {\it
    Ann. Probab.} {\bf 6}(47) (2019), 3397--3477. 
\bibitem{BaMiRi} {Baur, E., Miermont, G., Richier, L.} {Geodesic rays in
    the Uniform Infinite Half-Planar Quadrangulation return to the
    boundary.} {\it Lat. Am. J. Probab. Math. Stat.} {\bf 13} (2016), 1123--1149.
\bibitem{BeSc}{Benjamini, I., Schramm, O.} {Recurrence of distributional
    limits of finite planar graphs.} {\it Electron. J. Probab.} {\bf 6}(23)
  (2001), 13 pp.
\bibitem{Be} {Bettinelli, J.} {Scaling limit of random planar
    quadrangulations with a boundary.} {\it Ann. Inst. H. Poincar\'e} {\bf
    51}(2) (2015), 432--477.
\bibitem{BeMi} {Bettinelli, J., Miermont, G.} {Compact Brownian surfaces
    I. Brownian disks.} {\it Probab. Theory Relat. Fields} {\bf 167}(3-4) (2017), 555--614.
\bibitem{BjSt1} {Bj\"ornberg, J.E., Stef\'ansson, S. \"O.} {Recurrence of
    bipartite planar maps.} {\it Electron. J. Probab.} {\bf 19}(31) (2014),
  1--40.
\bibitem{BoDFGu} {Bouttier, J., Di Francesco, P., Guitter, E.} {Planar maps
    as labeled mobiles.} {\it Electron. J. Combin.} {\bf 11}(1) (2004).
\bibitem{BoGu} {Bouttier, J., Guitter, E.} {Distance statistics in
    quadrangulations with a boundary, or with a self-avoiding loop.}  {\it
    J. Phys. A. }{\bf 42}(46) (2009), 465208, 44 pp.
\bibitem{Bu}{Budd, T.} {The Peeling Process of Infinite Boltzmann Planar
    Maps.} {\it Electron. J. Combin.} {\bf 23}(1), 1--28.
\bibitem{Bz} {Budzinski, T.} {The hyperbolic Brownian plane.} 
{\it Probab. Theory Relat. Fields} {\bf 167}(1-2) (2018), 503--541.
\bibitem{BuBuIv} {Burago, D., Burago, I., Ivanov, S.} {\it A course in
    metric geometry.} Graduate Studies in Mathematics (Book 33), AMS (2001).
\bibitem{CaCu} Caraceni, A., Curien, N. {Geometry of the uniform infinite
    half-planar quadrangulation.}  {\it Random Struct. Algor.} {\bf
    52}(3) (2018), 454--494.
\bibitem{Cu1}Curien, N. {Peeling random maps.}  Peccot lecture (2016)
  Available at http://www.math.u-psud.fr/$\sim$curien/cours/peccot.pdf.
\bibitem{Cu2}Curien, N. {Planar stochastic hyperbolic triangulations.}
  {\it Probab. Theory Relat. Fields} {\bf 165} (2016), 509--540.
\bibitem{CuKo1} Curien, N., Kortchemski, I. {Random
    stable looptrees.} {\it Electron. J. Probab.} {\bf 19}(108) (2014), 35 pp.
\bibitem{CuKo2} Curien, N., Kortchemski, I. {Percolation on random
    triangulations and stable looptrees.} {\it Probab. Theory Relat. 
    Fields} {\bf 163}(1-2) (2015), 303--337.
\bibitem{CuLG} Curien, N., Le Gall, J.-F. {The Brownian plane.}  {\it
    J. Theor. Probab.} {\bf 27}(4) (2014), 1249--1291.
\bibitem{CuMi} Curien, N., Miermont, G. {Uniform infinite planar
    quadrangulations with a boundary.} {\it Random Struct. Algor.} {\bf
    47}(1) (2015), 30--58.
\bibitem{GuNa} Gurel-Gurevich, O., Nachmias, A. {Recurrence of planar graph limits.}
{\it Ann. of Math.} {\bf 177} (2013), 761--781.
\bibitem{GwMi} Gwynne, E., Miller, J.  Scaling limit of the uniform infinite half-plane quadrangulation in the
  Gromov-Hausdorff-Prokhorov-uniform topology. {\it Electron. J. Probab.} {\bf 22}(84) (2017), 47 pp. 
\bibitem{IbLi} Ibragimov, I. A., Linnik, Ju V. {\it Independent and
    stationary sequences of random variables.} Wolters-{Noordhoff}
  {Publishing}, Groningen (1971).
\bibitem{Ja} {Janson, S.} {Simply generated trees, conditioned
    Galton-Watson trees, random allocations and condensation.}  {\it
    Probab. Surveys} {\bf 9} (2012), 103--252.
\bibitem{KaSh} {Karatzas, I., Shreve, S. E.} {\it Brownian Motion and
    Stochastic Calculus.} Graduate Texts in Mathematics (Book 113),
  Springer, 2nd edition (1991).
\bibitem{Ke} Kesten, H. {Subdiffusive behavior of random walk on a random
    cluster.} {\it Ann. Inst. H.  Poincar\'e Probab. Statist.} {\bf 22}(4)
  (1986), 425--487.
\bibitem{Kr} Krikun, M. {Local structure of random quadrangulations.}
  (2005), arXiv:0512304.
\bibitem{LG0} {Le Gall, J.-F.} {\it Spatial branching processes, random
    snakes and partial differential equations.} Lectures in Mathematics ETH
  Z\"urich. Birkh\"auser, Basel (1999).
\bibitem{LG} Le Gall, J.-F. {Random trees and applications.} {\it
    Probability Surveys} {\bf 2} (2005), 245--311.
\bibitem{LGMi} Le Gall, J.-F., Miermont, G. {\it Scaling Limits of Random
    Trees and Planar Maps.} Clay Mathematics Proceedings {\bf 15} (2012).
\bibitem{LyPe}Lyons, R., Peres, Y. {\it Probability on trees and
    networks.} {Cambridge University Press, (2016).}
\bibitem{Pi} Pitman, J. {\it Combinatorial Stochastic
    Processes. \'Ecole d'\'et\'e de Probabilit\'es de St. Flour.} Lecture
  Notes in Mathematics {\bf 1875}, Springer (2006).
\bibitem{Ra} Ray, G. {Geometry and percolation on half planar
    triangulations.} {\it Electron. J. Probab.} {\bf 19}(47) (2014), 28 pp.
\bibitem{Ri} Richier, L. {Universal aspects of critical percolation on
    random half-planar maps.} {\it Electron. J. Probab.} {\bf 20} (2015), 45 pp.
\bibitem{St} Stephenson, R. {Local convergence of large critical multi-type
    Galton-Watson trees and applications to random maps.} {\it
    J. Theor. Probab.} {\bf 31}(1) (2018), 159--205.
\end{thebibliography}
\end{document}